\documentclass[12pt, oneside]{amsart} 
\usepackage[margin= 1in]{geometry}   
\usepackage{graphicx}		
\usepackage{amssymb}
\usepackage{mathrsfs}
\usepackage{amsthm}
\usepackage{amsmath}
\usepackage{hyperref}
\usepackage{tikz-cd}
\usepackage{dsfont}
\usepackage{tikz}
\usepackage{tikz-3dplot}

\usepackage{xcolor}

\usetikzlibrary{math}
\newcounter{thm}
\newtheorem{remark}[thm]{Remark}
\newtheorem{proposition}[thm]{Proposition}
\newtheorem{definition}[thm]{Definition}
\newtheorem{theorem}[thm]{Theorem}
\newtheorem{example}[thm]{Example}
\newtheorem{lemma}[thm]{Lemma}
\newtheorem{corollary}[thm]{Corollary}
\newcommand{\Ag}{\mathcal{A}_{\gamma',\gamma}}
\newcommand{\At}{\mathcal{A}_{\tau',\tau}}

\newcommand{\NN}{\mathbb{N}}
\newcommand{\ZZ}{\mathbb{Z}}
\newcommand{\kk}{\mathds{k}}

\usetikzlibrary{fadings}

\numberwithin{equation}{section}
\numberwithin{thm}{section}
\numberwithin{figure}{section}

\title{Birational Invariance in Punctured Log Gromov-Witten Theory}
\author{Samuel Johnston}
\email{samuel.johnston@imperial.ac.uk}
\address{Imperial College London, South Kensington Campus, London SW7 2AZ, UK}
\date{}		

\begin{document}

\begin{abstract}
Given a log smooth scheme $(X,D)$, and a log \'etale modification $(\tilde{X},\tilde{D}) \rightarrow (X,D)$, we relate the punctured Gromov-Witten theory of $(\tilde{X},\tilde{D})$ to the punctured Gromov-Witten theory of $(X,D)$, generalizing results of Abramovich and Wise in the non-punctured setting in \cite{bir_GW}. Using the main comparison results, we show a form of log \'etale invariance for the logarithmic mirror algebras and canonical wall structures constructed in \cite{int_mirror} and \cite{scatt} respectively. 
\end{abstract}

\maketitle 

\tableofcontents
\section{Introduction}

A central construction in logarithmic geometry is a log \'etale modification. Generalizing toric blow-ups, these are basic surgeries which might greatly alter the underlying scheme, but respect much of the underlying logarithmic geometry. An example of this phenomenon in the context of log Gromov-Witten theory is the result of Abramovich and Wise on the invariance of log Gromov-Witten invariants under log \'etale modifications \cite{bir_GW}. Roughly speaking, this result states that given a log \'etale modification of fine and saturated log smooth schemes $\tilde{X} \rightarrow X$, the integrals on moduli spaces associated with $X$ computing log Gromov-Witten invariants of $X$ are equal to appropriate integrals on moduli spaces associated with $\tilde{X}$. 

Since this work, there have been numerous developments in the field of log Gromov-Witten theory. First, it has become clear that studying the associated tropical geometry of the tropicalization of a target $X$ can yield useful coarse methods of working with $X$, which has allowed for deeper understanding of decomposition results, see \cite{expand} for an example. On another front, the more flexible construction of punctured log curves has been developed, which allows for negative contact orders of marked points with boundary divisors. This is useful for cutting and gluing operations, and also shares a rich relationship with tropical geometry. 

With these developments in the field, we seek to generalize the result of Abramovich and Wise to the setting of punctured log Gromov-Witten theory. Let $X \rightarrow B$ be a projective log smooth morphism of fine and saturated log schemes, with $B$ either log smooth or a geometric point with rank $1$ log structure, and a realizable tropical type $\tau$ of punctured log curve to $X$ in the sense of \cite{punc}. Consider a log \'etale modification $\tilde{X} \rightarrow X$ which is derived from a subdivision of the corresponding fans $\Sigma(\tilde{X})\rightarrow \Sigma(X)$, or equivalently a log \'etale modification of corresponding Artin fans $\tilde{\mathcal{X}} \rightarrow \mathcal{X}$.

To compare punctured log maps with target $\tilde{X}$ to punctured log maps with target $X$, we fix discrete data in the form of realizable tropical types $\gamma$ to $\Sigma(\tilde{X})$ and $\tau$ to $\Sigma(X)$ such that $\gamma$ is a tropical lift of $\tau$ in the sense of Definition \ref{dtlift}. The data of the tropical types $\tau$ and $\gamma$ give a pair of cones, which we denote by the same name, and allow for the definition of the moduli stacks $\mathscr{M}(X,\tau)$ and $\mathscr{M}(\tilde{X},\gamma)$ respectively. The condition of $\gamma$ being a tropical lift of $\tau$ is a necessary condition for stabilization to induce a morphism $\mathscr{M}(\tilde{X},\gamma) \rightarrow \mathscr{M}(X,\tau)$, and gives an injective morphism of cones $i:\gamma \rightarrow \tau$. Both moduli spaces are virtually equidimensional, but the virtual dimensions of $\mathscr{M}(\tilde{X}/B,\gamma)$ will typically be greater than the virtual dimension of $\mathscr{M}(X/B,\tau)$, meaning we cannot expect a simple generalization of the main result of \cite{bir_GW}. Without making further assumptions, we establish the following relation:

\begin{theorem}\label{mthm1}
Assume $\tau$ and $\gamma$ are a pair of realizable tropical types on $X$ and $\tilde{X}$, with $\gamma$ a lift of $\tau$, such that stabilization induces a morphism $s: \mathscr{M}(\tilde{X}/B,\gamma) \rightarrow \mathscr{M}(X/B,\tau)$. Then there exists a log algebraic stack $\mathfrak{M}_{\gamma \rightarrow \tau}$, an \'etale morphism $\mathfrak{M}_{\gamma \rightarrow \tau} \rightarrow \mathfrak{M}(\widetilde{\mathcal{X}}/B,\gamma)$ and a cartesian square of algebraic stacks and fine and saturated log algebraic stacks:

\[\begin{tikzcd}
\mathscr{M}(\tilde{X}/B,\gamma) \arrow{r} \arrow{d} & \mathfrak{M}_{\gamma \rightarrow \tau} \arrow{d} \\
\mathscr{M}(X/B,\tau) \arrow {r} & \mathfrak{M}(\mathcal{X}/B,\tau)
\end{tikzcd}
\]
with $\mathfrak{M}_{\gamma \rightarrow \tau} \rightarrow \mathfrak{M}(\mathcal{X}/B,\tau)$ idealized log \'etale, DM type and proper. Furthermore, the obstruction theory for the bottom horizontal morphism pulls back to the obstruction theory for $\mathscr{M}(\tilde{X}/B,\gamma) \rightarrow \mathfrak{M}(\tilde{\mathcal{X}}/B,\gamma)$. 
\end{theorem}

Roughly speaking, Theorem \ref{mthm1} says that the geometry of $\mathscr{M}(\tilde{X}/B,\gamma)$ is described by the geometry of $\mathscr{M}(X/B,\tau)$ and the tropical moduli space of tropical curves marked by $\gamma$.

When the dimensions of the tropical moduli spaces $\tau$ and $\gamma$ coincide, there is no longer a problem with jumping virtual dimension, meaning we can compare virtual classes in this setting, as is done in the following theorem:

\begin{theorem}\label{mthm2}
Suppose $\gamma$ is a lift of $\tau$, and dim $\tau =$ dim $\gamma$. Let $m = |coker(\gamma^{gp}_{\NN} \rightarrow \tau^{gp}_{\NN})|$. Then for $s: \mathscr{M}(\tilde{X},\gamma) \rightarrow \mathscr{M}(X,\tau)$ the map induced by $\pi: \tilde{X} \rightarrow X$: 

\[s_*[\mathscr{M}(\tilde{X}/B,\gamma)]^{vir} = \frac{1}{m}[\mathscr{M}(X/B,\tau)]^{vir} .\] 
\end{theorem}

\begin{example}
Consider a log smooth scheme $X$, i.e, $B = Spec$ $\kk$ with the trivial log structure and $X \rightarrow B$ is log smooth. Consider additionally a tropical type $\tau$ which has $1$ vertex, and $k$ legs with contact orders $a_i$ in $\Sigma(X)$, such that the unique vertex maps to the zero cone of $\Sigma(X)$. Then $\mathscr{M}(X,\tau)$ is the ordinary moduli space of log stable maps with contact order given by the vector $(a_i)$. The tropical moduli space of such a tropical type consists of a single point, and any lift of the $\tau$ to a family of tropical curves mapping to $\tilde{X}$ must also have trivial tropical moduli. Hence, the lattice index is $1$, and we recover the fact that the pushforward of $[\mathscr{M}(\tilde{X},\gamma)]^{vir}$ is $[\mathscr{M}(X,\tau)]^{vir}$.

\end{example}

We use the above theorems to weaken necessary assumptions for the foundations of punctured log Gromov Witten theory, as well as give an application to the construction of mirrors.

First, in the paper \cite{punc} which constructed punctured log GW invariants, the moduli stacks $\mathscr{M}(X,\tau)$ were only shown to be finite type if a global generation assumption related to the log structure of $X$ was assumed. In the theory of ordinary stable log maps, this assumption was weakened in \cite{bound} by noting that the assumption held under suitable log modifications, and then using the ordinary birational invariance result. Using Theorem \ref{mthm1}, we produce as a corollary the analogous result in the more general setting of punctured log maps. 

Secondly, punctured invariants play a critical role in the mirror construction of Gross and Siebert. This construction takes as input a log Calabi-Yau pair $(X,D)$, and outputs an algebra over a monoid ring of effective curve classes in $X$ which is a candidate for the ring of functions on the mirror to $(X,D)$. A log \'etale modification $\tilde{X} \rightarrow X$ is another log Calabi-Yau compactificaton of $U:= X\setminus D$, and in particular give a pair of distinct mirror algebras $R_{(X,D)}$ and $R_{(\tilde{X},\tilde{D})}$. After applying the pushforward morphism $H_2(\tilde{X}) \rightarrow H_2(X)$ to all monomials in the mirror algebra, we wish to identify the mirror algebras produced by the two compactifications of $U$. The following corollary fulfills this wish:

\begin{corollary}\label{mcr1}
For $A_X = \kk[Q]/I$ the ring of functions on the base of the mirror family for $(X,D)$, with $I \subset\kk[Q]$ a monomial ideal whose radical is the maximal monomial ideal, let $Spec\text{ }R_{(X,D)} \rightarrow$ $Spec\text{ }A_{X}$ be the mirror family produced by the main construction of \cite{int_mirror}. Additionally, let $\pi_*: Q' \rightarrow Q$ be a morphism of monoids with $Q'$ finitely generated, $NE(\tilde{X}) \subset Q' \subset H_2(\tilde{X})$, and let $I' = \pi^{-1}_*(I)$ and $A_{\tilde{X}} = \kk[Q']/I'$. Then the mirror algebra $R_{(\tilde{X},\tilde{D})}$ with coefficients in $A_{\tilde{X}}$ is well-defined, and 
\[Spec\text{ }R_{(X,D)} = Spec\text{ }A_{X} \times_{Spec\text{ }A_{\tilde{X}}}Spec\text{ }R_{(\tilde{X},\tilde{D})},\]
with the morphism $Spec\text{ }A_X \rightarrow Spec\text{ }A_{\tilde{X}}$ induced by the spectra of monoid rings induced by the pushforward morphism $\pi_*:Q' \rightarrow Q$. 
\end{corollary}

 We will also use Theorems \ref{mthm1} and \ref{mthm2} to prove the canonical wall structure constructed by Gross and Siebert in \cite{scatt} using punctured log Gromov Witten invariants behaves well under log \'etale modifications in the following sense:
 
 \begin{corollary}
 With $\tilde{X}$ and $X$ as above and also satisfying Condition $1.1$ or $1.2$ of \cite{scatt}, then the pushforward of the canonical wall structure on $\widetilde{\mathscr{P}}$ to $\mathscr{P}$ is equivalent to the canonical wall structure on $\mathscr{P}$. 
 \end{corollary}

Under the prediction that the mirror algebras constructed in \cite{int_mirror} and \cite{scatt} are the algebro-geometric realization of degree zero symplectic cohomology of $U$, this result says at least that after changing Novikov parameters, the mirror algebra is invariant under certain modifications of the boundary, which is expected under this prediction.

We begin by relating the tropical type of a log stable map $C \rightarrow \tilde{X}$ with the tropical type of the stabilization of $C \rightarrow X$. Following this, we show a certain class of non-basic log stable maps $\overline{C} \rightarrow X$ determine log stable maps $C \rightarrow \tilde{X}$, and how this class of log stable maps is encoded by morphisms to an auxiliary stack $\mathfrak{M}_{\gamma \rightarrow \tau}$ which is constructed by considering an associated tropical moduli problem. By pulling back along morphisms from the log to the tropical moduli problem, as in \cite{stacktrop} \cite{expand}, we produce the desired stack. We then proceed to construct the upper horizontal arrows in the commutative diagram of Theorem \ref{mthm1}, with necessary care taken for subtleties regarding the typically non-saturated log structures on punctured log curves. After this diagram is constructed and shown to be cartesian, the necessary properties largely follow from pulling back appropriate properties from morphisms of associated Artin fans.

\section{Acknowledgement}
I would like to thank my supervisor Mark Gross for his encouragement, reviewing drafts and many helpful comments and corrections, Dhruv Ranganathan for many useful discussions related to this work and help in creating diagrams, and the referee for many helpful comments. I would also like to thank Bernd Siebert for many useful comments and suggestions, especially concerning Section $5.1$. This work was done with the financial support of the ERC grant ``Mirror Symmetry in Algebraic Geometry" as well as the Heilbronn Institute for Mathematical Research and Imperial College London.
\section{Logarithmic and Tropical preliminaries}

Let $X \rightarrow B$ be a log smooth morphism of fine and saturated log schemes, locally of finite type with $X$ locally noetherian over an algebraically closed field $\kk$ of characteristic $0$.

In \cite{punc} Definition $2.1$, the authors introduce the following notion of a puncturing of a log structure:

\begin{definition}
Let $Y = (Y^u,\mathcal{M}_Y)$ be an fs log scheme, together with a decomposition $\mathcal{M}_Y = \mathcal{M} \oplus_{\mathcal{O}_Y^\times} \mathcal{P}$. A puncturing of $Y$ along $\mathcal{P}$ is a fine subsheaf of monoids $\mathcal{M}_{Y^\circ} \subset \mathcal{M} \oplus_{\mathcal{O}_Y^\times} \mathcal{P}^{gp}$, containing $\mathcal{M}_Y$, together with a structure map $\alpha_{Y^\circ}:\mathcal{M}_{Y^\circ} \rightarrow \mathcal{O}_Y$, satisfying:

\begin{enumerate}
\item The inclusion $p^\flat: \mathcal{M}_Y \rightarrow \mathcal{M}_{Y^\circ}$ is a morphism of log structures on $Y^u$.

\item For any geometric point $x \in Y$, let $s_x \in \mathcal{M}_{Y^\circ,x}\setminus \mathcal{M}_{Y,x}$. Representing $s_x = (m_x,p_x) \in \mathcal{M}_x \oplus_{\mathcal{O}^\times_{Y_x}} \mathcal{P}^{gp}_x$, we have $\alpha_{Y^\circ}(s_x) = \alpha_{\mathcal{M}}(m_x) = 0$ in $\mathcal{O}_{Y,x}$.

\end{enumerate}

\end{definition}

\begin{example}\label{pex}
Take $C = $Spec $\kk[x,y]/(xy)$, with log structure induced by the monoid $Q = \langle log(x),log(y),t  \mid log(x) + log(y) = t\rangle$ with global chart $\alpha(log(x)) = x$, $\alpha(log(y)) = y$, and $\alpha(t) = 0$. The underlying scheme is two copies of $\mathbb{A}^1$ meeting at a point, with $x$ and $y$ being coordinates for these components. Restricting the log structure to a component defines a log structure on $\mathbb{A}^1$ which is a puncturing of the log structure on $(\mathbb{A}^1,\{0\}) \times \text{Spec } (\mathbb{\kk},\langle t\rangle)$.   

More generally, restriction of log structures on prestable log curves over log points to components is a puncturing of a natural log structure on the subscheme.
\end{example}

In this paper, we will only be interested in puncturing log curves along marked points, i.e. puncturings of log curves $C^\circ \rightarrow C$ such that $\mathcal{M}_{C,p} \rightarrow \mathcal{M}_{C^\circ,p}$ is an isomorphism when $p$ is not a marked point of $C$. Throughout, unless stated otherwise, we will refer to $C/S$ as an ordinary log curve, and $p: C^\circ \rightarrow C$ as the puncturing of $C$ at marked points. We will in general suppress the puncturing map unless necessary. Due to subtleties of pulling back families of punctured log curves along strict morphisms on the base, the monoids $\mathcal{M}_{Y^\circ}$ will typically be fine but not necessarily saturated. When $C^\circ$ is equipped with a map to an fs log scheme $C^\circ \rightarrow X$ however, the condition of prestability detailed in \cite{punc} Definition $2.14$ gives a well behaved choice of puncturing, which we will be making use of throughout this paper.

In the interest of tropicalizing log stacks, we let \textbf{Cones} be the category whose objects are rational polyhedral cones $\sigma$ equipped with an integral lattice $\Lambda_\sigma$, i.e., $\sigma \subset \Lambda_{\sigma} \otimes_{\mathbb{Z}} \mathbb{R}$. Throughout this paper, we refer to $\sigma$ as a real cone, $\sigma_{\mathbb{Q}} = \sigma \cap (\Lambda_{\sigma} \otimes_{\ZZ} \mathbb{Q})$ as the monoid of rational points, and $\sigma_\mathbb{N} = \sigma \cap \Lambda_{\sigma}$ as the monoid of integral points. We let \textbf{RPCC} be the category of rational polyhedral cone complexes. 

In order to translate combinatorial constructions in the category of cone complexes to the category of log stacks, we make use of category of Artin cones, Artin fans, and idealized Artin fans.

\begin{definition}\label{ACdef}
Let $\sigma$ be a cone, with associated affine toric variety $S_\sigma$ containing a dense torus $T_\sigma$. Then the Artin cone $\mathcal{A}_{\sigma}$ is defined to be the stack quotient $[S_\sigma/T_\sigma]$. If $\sigma' \subset \sigma$ is a lattice refinement, there is an associated root stack of $S_\sigma$ associated with $\sigma'$, which we denote by $S_{\sigma'}$. This toric stack in turn has an Artin fan given by $[S_{\sigma'}/T_{\sigma'}]$.

For a face $\omega \subset \sigma$, let $S_{\sigma,\omega}$ be the closure of the torus stratum in $S_\sigma$ associated with the interior of the cone $\omega$, and we define the associated idealized Artin cone $\mathcal{A}_{\sigma,\omega}$ to be the stack quotient $[S_{\sigma,\omega}/T_{\sigma}]$. 

We say an Artin fan $\mathcal{A}$ has faithful monodromy if the tautological morphism $\mathcal{A} \rightarrow \textbf{Log}$ is representable.
 
\end{definition}

In order to relate tropical moduli problems to log geometric moduli problems, we recall the equivalence detailed in \cite{stacktrop} Theorem $6.11$. 

\begin{lemma}\label{equiv}
There is an equivalence of $2$-categories between the categories of Artin fans and cone stacks. On a cone $\sigma \in Ob(\textbf{RPCC})$, the equivalence is given by sending $\sigma$ to $\mathcal{A}_\sigma$. 
\end{lemma}
Associated to any locally noetherian fine log algebraic stack $X$ is an Artin fan $\mathcal{A}_X$ which is initial among all Artin fans $\mathcal{A}$ with faithful monodromy and a strict map $X \rightarrow \mathcal{A}$. See \cite{bound} section $3$ for further details. While loc cit. constructs such an Artin fan, the assignment fails to be functorial. Nonetheless, letting $\mathcal{A}_B$ be the Artin fan for a log smooth scheme $B$ as above, for a log smooth morphism $X \rightarrow B$, there exists an Artin fan $\mathcal{A}_X'$, a strict morphism $X \rightarrow \mathcal{A}_X'$, and a morphism $\mathcal{A}_X' \rightarrow \mathcal{A}_B$ making the obvious square in the category of log stacks commute. Indeed, this fact follows by \cite{bound} Corollary $3.3.6$. In general, we will assume we have such a morphism. Hence, we assume we have an Artin fan $\mathcal{A}_X$, a strict morphism $X \rightarrow \mathcal{A}_X$, and a morphism $\mathcal{A}_X \rightarrow \mathcal{A}_B$ making the obvious square commute, and we let $\mathcal{X} = \mathcal{A}_X\times_{\mathcal{A}_B} B$ be the relative Artin fan of $X\rightarrow B$, as considered in the preamble of Section $3$ of \cite{punc}. Note in particular we have a strict map $X \rightarrow \mathcal{X}$.

In order to produce a tropcalization functor, we recall the category of generalized cone complexes, introduced in \cite{tropcurv}, and the construction from Appendix $C$ of \cite{punc} of the tropicalization functor $\Sigma$ from algebraic fine log stacks to generalized cone complexes. If an Artin fan $\mathcal{A}$ has a Zariski cover by Artin cones, we will say $\mathcal{A}$ is a Zariski Artin fan. The following proposition follows from the proof of Proposition $2.10$ of \cite{decomp}.

\begin{lemma}\label{mtoZA}
For a fine log scheme $T$, and a Zariski Artin fan $\mathcal{A}$, then $Hom_{log}(T,\mathcal{A})$ is the subset of $Hom_{Cones}(\Sigma(T),\Sigma(\mathcal{A}))$ such that for every induced morphism $\sigma \rightarrow \tau \in \Sigma(\mathcal{A})$ for $\sigma \in \Sigma(T)$, the morphism $\tau_{\NN}^\vee \rightarrow \sigma_{\NN}^{\vee} = \Gamma(U,\overline{\mathcal{M}}_U)^{sat}$ factors through $\Gamma(U,\overline{\mathcal{M}}_U)$.
\end{lemma}

At various points in this paper, we will be interested in constructing morphisms from a fine log scheme $T$ to an Artin fan $\mathcal{A}$, typically the Artin fan associated with a target log scheme $X$. The main case we are interested in is when we are provided with a morphism to an Artin fan $\mathcal{A}$, and wish to produce a lift along a log \'etale modification $\widetilde{\mathcal{A}} \rightarrow \mathcal{A}$. We use a lemma appearing in \cite{logDR} Lemma $6$, given in a slightly modified state below to account for the fine log schemes considered in our context:

\begin{lemma}\label{mtoA}
Let $p: \widetilde{\mathcal{A}} \rightarrow \mathcal{A}$ be a log modification, induced by a subdivision $\widetilde{\Sigma(\mathcal{A})} \rightarrow \Sigma(\mathcal{A})$. Then, the following two properties hold:
\begin{enumerate}
\item p is a representable monomorphism of logarithmic algebraic stacks.
\item A log map $S \rightarrow \mathcal{A}$ from an fs log scheme $S$ lifts to $\widetilde{\mathcal{A}}$ if and only if, \'etale locally on $S$, the map $\Sigma(S) \rightarrow \Sigma(\mathcal{A})$ lifts through $\widetilde{\Sigma(\mathcal{A})}$. If $S$ is only fine, we further require the factorization property of Lemma \ref{mtoZA}. 
\end{enumerate}
\end{lemma}

Having related the underlying sites of the tropical and logarithmic moduli problems, we now consider the tropical moduli problem associated with punctured tropical maps.

Consider a graph $G$, with vertex set $V(G)$, an edge set $E(G)$, and leg set $L(G)$. An abstract tropical curve over a cone $\sigma \in \textbf{Cones}$ is the data $(G,\textbf{g},l)$, with $G$ a graph, $\textbf{g}: V(G) \rightarrow \mathbb{N}$ a genus function, and $l: E(G) \rightarrow Hom(\sigma_\mathbb{N},\mathbb{N})\setminus 0$ an assignment of edge lengths. The data of $(G,l)$ determines a cone complex $\Gamma_\sigma$ over $\sigma$, with cones in bijection with features $L(G)\cup E(G) \cup V(G)$ a feature. We let $\sigma_F$ be the cone associated with the feature $F$. Let $\pi_{(G,l)}: \Gamma_\sigma \rightarrow \sigma$ be the projection. This map of cone complexes satisfies the condition that for $s \in int(\sigma)$, the fiber $\Gamma_s$ over $s$ is isomorphic to $G$, with the length of an edge $e \in E(G)$ given by $l(e)(s)$. Note in particular that when $l(e)(s) = 0$, the fiber $\Gamma_s$ has a $1$-dimensional cell associated with $e \in E(G)$ which is contracted. Along with families of ordinary tropical curves, we will also be interested in families of punctured tropical curves.

\begin{definition}\label{ptropcurve}
Given a family of ordinary tropical curves $\pi_{(G,l)}: \Gamma_\sigma \rightarrow \sigma$, a puncturing $\Gamma_\sigma^\circ \rightarrow \sigma$ of $\Gamma_\sigma$ is a complex of subcones $i: \Gamma_{\sigma}^\circ \rightarrow \Gamma_\sigma$ over $\sigma$ which is an isomorphism away from the cones $\sigma_l \in \Gamma_\sigma$ associated with legs $l \in L(G)$, and upon restricting $i$ to all cones mapping to $\sigma_l$, $i$ is given by the inclusion of a full dimensional subcone $\sigma_l^\circ \subset \sigma_l$ which contains $\sigma_v$ as a face for $v$ the unique vertex of $G$ contained in $l$. We let $\sigma_F^\circ$ denote the unique equidimensional cone mapping to $\sigma_F$ under $i$ for $F \in L(G) \cup E(G)\cup V(G)$, and call $\cup_F i(int(\sigma_F^\circ))$ the \emph{proper image} of $i$.
\end{definition}

\begin{figure}
\centering
\begin{tikzpicture}
\fill[white!40!violet, path fading = east] (-2,0,0)--(4,0,0)--(0,0,4)--cycle;
\fill[white!40!blue, path fading = east] (-2,0,0)--(4,0,0)--(2,4,2)--cycle;
\fill[white!40!blue, path fading = east] (-2,0,0)--(0,0,4)--(2,4,2)--cycle;
\draw[color=black] (-2,0,0)--(4,0,0)--cycle;
\draw[color=black] (4,0,0)--(2,4,2);
\draw[color=black] (0,0,4)--(2,4,2);
\draw[color=black] (-2,0,0)--(0,0,4)--cycle;
\draw[black](0,0,4)--(4,0,0);
\draw[color=black] (-2,0,0)--(2,4,2)--cycle;
\draw[->,color=red] (2,0,2)--(2,4,2);
\draw[ball color=red] (2,0,2) circle (0.75mm);

\fill[white!40!violet, path fading = east] (6,0,0)--(12,0,0)--(8,0,4)--cycle;
\fill[white!40!blue, path fading = east] (6,0,0)--(6,6,0)--(12,6,0)--(12,0,0)--cycle;
\fill[blue] (6,0,0)--(8,0,4)--(8,6,4)--(6,6,0)--cycle;
\draw[color=black] (6,6,0)--(12,6,0);
\draw[color=black] (12,6,0)--(12,0,0);
\draw[color=black] (6,6,0)--(8,6,4);
\draw[color=black] (8,6,4)--(8,0,4);
\draw[color=black] (6,0,0)--(12,0,0)--cycle;
\draw[-,dashed] (12,0,0)--(10,4,2);
\draw[-,dashed] (8,0,4)--(10,4,2);
\draw[color=black] (6,0,0)--(8,0,4)--cycle;
\draw[black](8,0,4)--(12,0,0);
\draw[-,dashed] (6,0,0)--(10,4,2)--cycle;
\draw[color=black] (6,0,0)--(6,6,0);
\draw[->,color=red] (10,0,2)--(10,6,2);
\draw[ball color=red] (10,0,2) circle (0.75mm);
\end{tikzpicture}
\caption{An example of a puncturing morphism $\Gamma_{\sigma}^\circ \rightarrow \Gamma_{\sigma}$ restricted to a cone $\sigma_{l}$ associated with a leg. In this case, we have $\sigma \cong \mathbb{R}_{\ge 0}^2$, and the leg $l$ of one of the paramterized tropical curves is depicted. Note that unlike for edges, the length of the leg need not be a linear function on the base.}
\end{figure}

Given a generalized cone complex $\Sigma$, a family of tropical maps to $\Sigma$ over $\sigma$ as above is morphism of generalized cone complexes $h: \Gamma_\sigma \rightarrow \Sigma$. As before, over a point $s \in \sigma$, we have a tropical curve $\Gamma_s$, and a piecewise linear map $\Gamma_s \rightarrow \Sigma$. In particular, this gives two functions. First, we have a function $\pmb\sigma: V(G) \cup E(G) \cup L(G) \rightarrow \Sigma$, which assigns to a feature of $G$, the cone of the complex $\Sigma$ which the feature maps into. Secondly, we have a function $\textbf{u}$ which assigns to an element $e \in E(G)$ an integral tangent vector $u_e \in \Lambda_{\pmb\sigma(e)}$. The data $(G,\pmb\sigma,\textbf{u},\textbf{g})$ for $s \in int(\sigma)$ determines the tropical type of the family.

Given the data of a tropical type $(G,\pmb\sigma,\textbf{u},\textbf{g})$ above, we say the tropical type is \emph{realizable} if there exists a family of tropical curves realizing this type. In such a situation, there is a cone $\tau \in \textbf{Cones}$ which is the universal family of tropical curves of type $\tau$, i.e. there exists a family $\Gamma_\tau/\tau$ of tropical maps to $\Sigma(X)$ with type $(G,\pmb\sigma,\textbf{u},\textbf{g})$ over a cone $\tau$, and all other families of type $(G,\pmb\sigma,\textbf{u},\textbf{g})$ over a cone $\omega \in \textbf{Cones}$ arise as the pullback along an appropriate map $\omega \rightarrow \tau$. Whenever the type is realizable, we will refer to both the tropical type and the associated universal cone by the same name. 

Families of tropical maps in the context of log Gromov-Witten theory typically arise as follows: Suppose we have a punctured log map $(C^\circ/W,f)$ with target an fs log scheme $X$, or its Artin fan $\mathcal{A}_X$. Functorial tropicalization then allows us to tropicalize the family of maps. Care should be taken however after tropicalizing the typically non-saturated log scheme $C^\circ$. For a fine log scheme, the stalks of the saturation of the ghost sheaf are equal to the dual monoids of cones of the tropicalization, and in particular are generally larger than the fine monoids at stalks of the ghost sheaf of $C^\circ$. However, assuming the log structure is prestable in the sense of \cite{punc} Definition $2.1.4$, the data of the fine submonoid at a punctured section can be recovered from the data of the tropical map, see the remark following Proposition $2.21$ of \cite{punc}.

We also recall that how local sections of the sheaf $\mathcal{M}_{C^\circ}$ induce PL functions on subcomplexes of the tropical family $\Gamma = \Sigma(C)$. First, recall the ghost exact sequence of sheaves of abelian groups:

\[ 0 \rightarrow \mathcal{O}^{\times}_X \rightarrow \mathcal{M}_{X}^{gp} \rightarrow \overline{\mathcal{M}}_X^{gp} \rightarrow 0. \]

Since we assume $X$ has a fine log structure, we have injections of sheaves of monoids $\mathcal{M}_{X} \rightarrow \mathcal{M}_X^{gp}$ and $\overline{\mathcal{M}}_{X} \rightarrow \overline{\mathcal{M}}_X^{gp}$. Recall that the tropicalization $\Sigma(X)$ is constructed out of a colimit of atomic open neighborhoods $U \subset X$, i.e. connected open neighborhoods containing a unique closed stratum such that for all points $x$ contained in this closed stratum, we have $\Gamma(U,\overline{\mathcal{M}}_U) \cong \overline{\mathcal{M}}_{U,x}$ with isomorphism induced by restriction. On a given atomic neighborhood $U$, we have $\Sigma(U) = Hom_{mon}(\Gamma(U,\overline{\mathcal{M}}_X),\mathbb{R}_{\ge 0})$. This is a rational polyhedral cone, and by construction, local sections of the ghost sheaf determine non-negative linear functions on the cone. In general, given sections either of $\mathcal{M}_X$ or $\overline{\mathcal{M}}_X$ determined on some open set $U \subset X$, then $U$ is covered by atomic neighborhoods, and hence the restriction of the section determines linear functions on each of the associated cones, which furthermore are equal on overlapping faces inside $\Sigma(X)$. Hence, sections of $\mathcal{M}_X$ and $\overline{\mathcal{M}}_X$ define PL functions on $\Sigma(X)$. If $X$ is a family of punctured log curves over a base log scheme $S$, then the pullback of linear functions on the base correspond to PL functions on $\Sigma(X)$ which are constant on fibers of $\Sigma(X) \rightarrow \Sigma(S)$, and the remaining PL functions induce linear functions along the edges and legs of the tropical curves in the family.

In order to use the tropical data above to specify discrete data for the logarithmic moduli spaces under consideration in this paper, recall the definition of a marking of a tropical type $\tau'$ by $\tau$ from \cite{punc} Definition $3.7$. When $\tau'$ and $\tau$ are both realizable, this means that $\tau$ can be identified with a face of $\tau'$, and by pulling back the universal family of tropical maps $\Gamma'/\tau'$ along this face inclusion, we produce the universal family of tropical maps $\Gamma/\tau$. We may also decorate tropical types by curve classes in some monoid $H^+_2(X)$ by assigning an element $\textbf{A}(v) \in H_2^+(X)$ to each vertex $v \in V(G_\tau)$. For a tropical type $\tau$, we let $\pmb\tau = (\tau,\textbf{A}(v))$ denote a decorated tropical type with underlying tropical type $\tau$. The notion of markings above extends in a natural way to a marking of a decorated tropical type by another decorated tropical type.

 By \cite{punc} Section $3$, to a tropical type $\tau$, not necessarily realizable, there is an associated DM log stack $\mathscr{M}(X/B,\tau)$ of stable log maps to $X$ marked by $\tau$, and an algebraic log stack $\mathfrak{M}(\mathcal{X}/B,\tau)$ of prestable punctured maps to $\mathcal{X}$ marked by $\tau$. By \cite{punc} Theorem 3.24, the latter stack is idealized log smooth over $B$. Furthermore, there is a perfect obstruction theory associated to the morphism $\epsilon_{\tau}: \mathscr{M}(X/B,\tau) \rightarrow \mathfrak{M}(\mathcal{X}/B,\tau)$, which facilitates virtual pullback in the sense of \cite{vpull}. Without further assumptions on $\tau$, $\mathscr{M}(X/B,\tau)$ is not even virtually of pure dimension. However, when $\tau$ is realizable, $\mathfrak{M}(\mathcal{X}/B,\tau)$ is pure dimension and reduced, see \cite{punc} Proposition $3.28$. In particular, $\mathscr{M}(X/B,\tau)$ can be equipped with a virtual fundamental class defined by: 
 \[[\mathscr{M}(X/B,\tau)]^{vir} := \epsilon_{\tau}^!([\mathfrak{M}(\mathcal{X}/B,\tau)]).\]
 
\section{Tropical lifts and stabilization maps}
Fix a realizable punctured tropical type $\tau = (G,\textbf{u},\pmb\sigma,\textbf{g})$, with corresponding stacks $\mathscr{M}_\tau := \mathscr{M}(X/B,\tau)$ and $\mathfrak{M}_{\tau} := \mathfrak{M}(\mathcal{X}/B,\tau)$ of stable and prestable log maps marked by $\tau$ respectively. The latter log stack is idealized log smooth, with a sheaf of ideals $\mathcal{I} \subset \mathcal{M}_{\mathfrak{M}_\tau}$ given by the canonical idealized log structure defined in \cite{punc} Definition $3.22$. Moreover, the stack contains an open substack $\mathfrak{M}_\tau^\circ$ consisting of families of punctured log maps $(C^\circ/S,f)$ of type $\tau$. The idealized log structure determines a subset of cones of the tropicalization of the moduli problem $\mathfrak{M}_\tau$ corresponding to non-empty strata in $\mathfrak{M}_{\tau}$. In this case, the requirement for non-emptiness is that the cone $\tau' \in \Sigma(\mathfrak{M}_{\tau})$ must contain $\tau$ as a face.

Now suppose we have a log \'etale modification $\pi: \tilde{X} \rightarrow X$. After recalling notation in the remark following Lemma \ref{equiv}, note that when $X$ is log smooth, by \cite{bir_GW} Corollary $2.6.7$, there is a subdivision of Artin fans $\widetilde{\mathcal{A}}_{X} \rightarrow \mathcal{A}_X$ such that $\pi$ is induced by the following pullback of log algebraic stacks:

\begin{equation}\label{mod}
\begin{tikzcd}
\tilde{X} \arrow{r} \arrow[d,"\pi"] & \tilde{\mathcal{X}} \arrow{d} \arrow{r} & \widetilde{\mathcal{A}}_X \arrow{d} \\
X \arrow{r} & \mathcal{X} \arrow{r} \arrow{d} & \mathcal{A}_X \arrow{d} \\
& B \arrow{r} & \mathcal{A}_B
\end{tikzcd}
\end{equation}

In the above diagram, the object $\tilde{\mathcal{X}}$ is defined so that the top right square is cartesian. Note in particular that $\widetilde{\mathcal{A}}_X$ is also an Artin fan, and $\tilde{X} \rightarrow \tilde{\mathcal{X}}$ is strict. In general, we will assume $\pi$ is induced by pulling back a subdivision of Artin fans. By this description, note that $\tilde{X} \rightarrow X$ is log smooth, hence $\tilde{X} \rightarrow B$ is log smooth. Since $\tilde{X} \rightarrow X$ is log \'etale, $T^{log}_{\tilde{X}/B} \cong \pi^*(T^{log}_{X/B})$.

In order to compare the log Gromov-Witten theory of $X$ to the log Gromov-Witten theory of $\tilde{X}$, we define a class of tropical types of maps to $\Sigma(\tilde{X})$ determined by the tropical type $\tau$:

\begin{definition}\label{dtlift}
Suppose we have realizable tropical types $\gamma$ and $\tau$ of punctured tropical maps to $\Sigma(\tilde{X})$ and $\Sigma(X)$ respectively, together with an inclusion of cones $i: \gamma \rightarrow \tau$ with $int(\gamma) \subset int(\tau)$. Then $\gamma$ is a tropical lift of $\tau$ if there exists puncturing of the family of tropical curves $\overline{\Gamma}^\circ\subset \overline{\Gamma}_{\gamma}$ over the cone $\gamma$ induced by $i$ which is contained in the domain of definition of the punctured tropical map $(\overline{\Gamma}_{\gamma}^{\circ}/\gamma,f)$ with target $\Sigma(X)$ induced by $i$ such that the universal family of punctured tropical maps over $\gamma$ is given by $\overline{\Gamma}^\circ \times_{\Sigma(X)} \Sigma(\tilde{X}) \rightarrow \Sigma(\tilde{X})$ in the category of rational polyhedral cone complexes.

A decorated tropical type $\pmb\gamma$ is a \emph{decorated lift} of a decorated tropical type $\pmb\tau$ if the underlying tropical type $\gamma$ is a tropical lift of the underlying type $\tau$, and after applying the morphism $st_*:H_2(\tilde{X}) \rightarrow H_2(X)$ to the decoration of $\gamma$, the resulting decorated tropical type of map to $\Sigma(X)$ is marked by $\pmb\tau$. 
\end{definition}

\begin{figure}[h]
\centering
\begin{tikzpicture}
\fill[white!50!blue, path fading = north]  (0,0)--(4,0)--(4,4)--(0,4)--cycle;
\draw[black] (0,0)--(4,0);
\draw[black] (0,0)--(0,4);
\draw[ball color = red] (3/2,3/2) circle (0.5mm);
\draw[->, color = red] (3/2,3/2)--(3/2,4);
\draw[->, color = red] (3/2,3/2)--(4,3/2);
\draw[->, color = red] (3/2,3/2)--(0,0);

\fill[white!70!blue, path fading = north] (5,0)--(9,0)--(9,4)--(5,4)--cycle;
\draw[black] (5,0)--(9,0);
\draw[black] (5,0)--(5,4);
\draw[black] (5,0)--(9,4);
\draw[ball color = red] (13/2,3/2) circle (0.5mm);
\draw[->, color = red] (13/2,3/2)--(13/2,4);
\draw[->, color = red] (13/2,3/2)--(9,3/2);
\draw[->, color = red] (13/2,3/2)--(5,0);
\draw[ball color = red] (15/2,2) circle (0.5mm);
\draw[->, color = red] (15/2,2)--(15/2,5/2);
\draw[->, color = red] (15/2,2)--(9,2);
\draw[->, color = red] (15/2,2)--(11/2,0);
\draw[ball color = red] (14/2,3) circle (0.5mm);
\draw[->, color = red] (14/2,3)--(14/2,4);
\draw[-,color = red] (14/2,3)--(8,3);
\draw[ball color = red] (8,3) circle (0.5mm);
\draw[->, color = red] (8,3)--(9,3);
\draw[->, color = red] (14/2,3)--(5,1);

\end{tikzpicture}
\caption{A graphical depiction of choices of tropical lift. The left hand graphic is a tropical map to $\Sigma(X)$ with associated tropical type $\tau$, and the right hand graphic depicts tropical maps to a subdivision $\Sigma(\tilde{X})$ of $\Sigma(X)$ whose tropical types are various choices of tropical lift of $\tau$. These choices includes various constraints on the tropical modulus, as well as a choice of leg length.} 
\end{figure}

To produce tropical lifts of the type $\tau$ to the modification, we first investigate the above setup after tropicalization of Diagram \ref{mod} and the relevant moduli stacks of log stable maps.

Consider the universal family of tropical maps of type $\tau$, diagrammatically given below:

\[\begin{tikzcd}
\tilde{\Gamma}_{\tau} \arrow{d} & \tilde{\Gamma}_{\tau}^\circ \arrow{l} \arrow {r} \arrow[d,"m"] & \Sigma(\tilde{X}) \arrow{d}\\
\Gamma_{\tau}\arrow{dr} & \Gamma_\tau^\circ\arrow{l} \arrow{r} \arrow{d} & \Sigma(X) \arrow{d} \\
&\tau \arrow{r} & \Sigma(B)
\end{tikzcd}\]

In the above diagram, we pull back the subdivision $\Sigma(\tilde{X}) \rightarrow \Sigma(X)$ along $\Gamma_\tau^\circ \rightarrow \Sigma(X)$ to produce the vertical arrow $m: \tilde{\Gamma}_{\tau}^\circ \rightarrow \Gamma_{\tau}$.  As $\Gamma_\tau^\circ$ is a subcomplex of $\Gamma_\tau$, we extend the subdivision $\tilde{\Gamma}_{\tau}^\circ \rightarrow \Gamma_\tau^\circ$ to a subdivision $\tilde{\Gamma}_\tau \rightarrow \Gamma_\tau$ which includes $\tilde{\Gamma}_{\tau}^\circ$ as a union of cones. Now observe, as in \cite{ssred} Proposition $4.4$, that we may push forward the induced subdivision of $\Gamma_\tau$ to produce a subdivision $\widetilde{\tau} \rightarrow \tau$. Letting $\pi: \tilde{\Gamma}_\tau \rightarrow \tau$ be the induced morphism of cone complexes, the cones of $\tilde{\tau}$ are given by:

\[\{\cap_i \pi(\sigma_i) \mid \sigma_i \in \tilde{\Gamma}_\tau\}.\]

Note that we do not necessarily have a morphism $\tilde{\Gamma}_\tau \rightarrow \tilde{\tau}$. However, by taking a cone $\omega \in \tilde{\tau}$ and pulling back $\widetilde{\Gamma}^\circ_{\tau} \rightarrow \tau$ and $\widetilde{\Gamma}_\tau \rightarrow \tau$ along $\omega \rightarrow \tau$, we produce morphisms of cone complexes $p:\widetilde{\Gamma}_{\omega} \rightarrow \omega$ and $p^\circ: \widetilde{\Gamma}_{\omega}^\circ \rightarrow \omega$, illustrated in the cartesian squares of cone complexes below:

\[\begin{tikzcd}
\widetilde{\Gamma}_{\omega} \arrow{r}\arrow[d,"p"] & \widetilde{\Gamma}_{\tau} \arrow[d]& & \widetilde{\Gamma}_{\omega}^\circ \arrow{r}\arrow[d,"p^\circ"] & \widetilde{\Gamma}_{\tau}^\circ \arrow[d]\\
\omega \arrow{r} & \tau & & \omega \arrow{r} & \tau
\end{tikzcd}.\]
 Note that the morphism $p$ satisfies the condition that all cones of $\widetilde{\Gamma}_\omega$ which do not map into a proper face of the base cone $\omega$ surject onto $\omega$. We let $k = dim\text{ }\omega$.

We will now show that after an appropriate lattice refinement on the base, the morphism $(\tilde{\Gamma}_\omega^\circ /\omega,f^\circ:\tilde{\Gamma}_{\omega}^\circ \rightarrow \Sigma(\tilde{X}))$ is a family of tropical maps:

\begin{proposition}\label{tcurve}
For $\omega \in \tilde{\tau}$ as above, there exists a coarsening of the lattice of $\omega$ such that $\tilde{\Gamma}_\omega^\circ \rightarrow \Sigma(\tilde{X})$ is a family of punctured tropical maps to $\Sigma(\tilde{X})$ for some tropical type $\gamma$ of punctured tropical map to $\Sigma(\tilde{X})$. Moreover, the induced morphism from the subcomplex of cones $\omega$ parameterizing families of tropical maps of type $\gamma$ to $\gamma$ is injective.
\end{proposition}
\begin{proof}
First consider the family of tropical curves $\Gamma_{\omega} \rightarrow \omega$ pulled back along $\omega \rightarrow \tau$. Let $e \in E(G_\tau)\cup L(G_\tau)$ be an edge or leg containing a vertex $v \in V(G_\tau)$. We have a corresponding cone $\omega_v \in \Gamma_\omega$ surjecting onto the base cone $\omega$. Moreover, we have a map $\omega_e \rightarrow \Sigma(X)$, and a subdivision $\tilde{\omega}_e \rightarrow \omega_e$. By construction, any cone of $\tilde{\omega}_e$ whose image in $\omega$ is not contained in a face of $\omega$ must surject onto $\omega$. Letting $\sigma \in \tilde{\omega}_e$ be one such cone, consider the linear map of associated real vector spaces $\sigma^{gp} \rightarrow \omega^{gp}$. Since $\sigma$ can be identified with a subcone of $\omega_{e}$, we have $\sigma^{gp} \subset \omega_e^{gp} = \omega^{gp} \times \mathbb{R}$, and the morphism $\sigma^{gp} \rightarrow \omega^{gp}$ is induced by restriction of the projection morphism. Since $\sigma^{gp} \rightarrow \omega^{gp}$ is surjective, we must have dim $\sigma^{gp} \ge $ dim $\omega^{gp}$, i.e. dim $\sigma = $ dim $\omega +1$ or dim $\omega$. In the latter case, the linear morphism $\sigma^{gp} \rightarrow \omega^{gp}$ is an isomorphism, and thus the surjective morphism of cones $\sigma \rightarrow \omega$ is injective, hence an isomorphism. After composing the inverse of this isomorphism with the inclusion $\sigma \rightarrow \omega \times \mathbb{R}_{\ge 0}$, the image of this morphism is given by the graph of a non-negative linear function on $\omega$, which we call $\rho_\sigma$. Since the subdivision $\Sigma(\tilde{X}) \rightarrow \Sigma(X)$ is rational, the linear function $\rho_\sigma$ above is rational. Suppose we have $\sigma'$ is a distinct cone in $\tilde{\omega}_e$ of dimension $dim\text{ }\omega$. Then the projection $\sigma' \rightarrow \omega$ is also an isomorphism, and we have an associated linear map $\rho_{\sigma'} \in \omega^{\vee}_\mathbb{Q}$. Note that cones in $\tilde{\omega}_e$ may only intersect on faces. Hence, the linear map $\rho_{\sigma} - \rho_{\sigma'}$ cannot be zero anywhere in $int(\sigma)$. By the intermediate value theorem and convexity of $\omega$, we must have either $\rho_\sigma < \rho_{\sigma'}$ or $\rho_{\sigma'} < \rho_{\sigma}$ on $int(\omega)$. We note by construction of $\omega \subset \tau$, for $e \in L(G_\tau)$ a strictly punctured leg of $\tau$, there is a unique codimension $1$ cone $\omega_{v_b}$  of $\tilde{\omega}_e$ whose image in $\Gamma_\omega$ is contained in the closure of the proper image of $\Gamma_{\omega}^\circ$ but not the proper image of $\Gamma_{\omega}^\circ$.

Now let $\sigma \in \tilde{\omega}_e$ be a cone of dimension $k+1$. Then $\sigma$ has at least one proper face which surjects onto $\omega$. If it has at least two such faces, denote them by $\sigma_{v_1}$ and $\sigma_{v_2}$, and suppose $\rho_{\sigma_{v_1}} \le \rho_{\sigma_{v_2}}$. Then $\sigma$ must be a subcone of the cone:
  
\[\{(g,t) \in \omega \times \mathbb{R}_{\ge 0}\text{ }| \text{ } \rho_{\sigma_{v_1}}(g) \le t \le \rho_{\sigma_{v_2}}(t)\}.\]
Since $\sigma$ contains $\sigma_{v_1}$ and $\sigma_{v_2}$, by convexity, $\sigma$ must equal this cone. 
  
 We now define a tropical type $\gamma$ as follows. We let $V(G_\gamma)$ be in bijection with the $k$ dimensional cones of $\tilde{\omega}_e$ which surject onto $\omega$, as $e$ ranges over $E(G_\tau)\cup L(G_\tau)$, and $E(G_\gamma)\cup L(G_\gamma)$ be in bijection with the dim $k+1$ cones of $\tilde{\omega}_e$ as $e$ ranges over $E(G_\tau) \cup L(G_\tau)$. We say $e \in E(G_\gamma)$ if the corresponding cone has two $k$-dimensional faces, and $e \in L(G_\gamma)$ if it only contains one such face. The incidence relation is defined by sending a dimension $k+1$ cone to the dimension $n$ cones in $V(G_\gamma)$ featured as a face. It is straightforward to see that this gives a subdivision of the graph $G_\tau$. For $e \in E(G_\gamma)$, containing vertices $v_1,v_2$ with $\rho_{v_1} \le \rho_{v_2}$ we define $l_e = \rho_{v_2} - \rho_{v_1}$. It is straightforward to see that $\tilde{\Gamma}_{\omega} \rightarrow \omega$ is a family of tropical curves with associated dual graph $G_\gamma$ and edge lengths $l_e$ with \emph{rational} edge lengths. We note that it is possible that some cones associated with vertices $v \in V(G_{\gamma})$ exist such that $m(\omega_v)$ are not contained in the proper image of $i:\Gamma_{\tau}^\circ \rightarrow \Gamma_{\tau}$, defined in Definition \ref{ptropcurve}. Since these are precisely the vertices $v \in V(G)$ with $\rho_{\omega_v} \ge \rho_{\omega_{v_b}}$, these cones additionally do not intersect the proper image of $i$. By merging cones appropriately, corresponding to removing vertices and identifying edges of $G_{\gamma}$, we produce a family of tropical curves such that no such vertex exists. In particular, after this merging, we still have have the subdivision $\tilde{\Gamma}_{\omega}^\circ \rightarrow \Gamma_{\omega}^\circ$ is induced by pulling back $\tilde{\Gamma}_{\omega} \rightarrow \Gamma_{\omega}$. Since there are finitely many edges in $G_\gamma$, there exists a finite index subgroup of $\omega_\NN^{gp}$ which when intersected with $\omega \subset \omega_{\mathbb{R}}^{gp}$ consists of the points $p \in \omega_\NN$ such that $l_e(p) \in \NN$ for all $e \in E(G_\gamma)$. It is straightforward to see that $\tilde{\Gamma}_{\omega} \rightarrow \omega$ is a family of tropical curves with associated dual graph $G_\gamma$ and edge lengths $l_e$.
 
Moreover, $\tilde{\Gamma}^\circ_\omega \rightarrow \tilde{\Gamma}_\omega$ is a puncturing of this family. Indeed, this follows from the fact that the subdivision $\tilde{\omega}^{\circ}_l \rightarrow \omega^\circ_l$ is given by pulling back the subdivision $\tilde{\omega}_l \rightarrow \omega_l$ for every leg $l \in L(G_\omega)$, and that cone of $\tilde{\omega}_l$ associated with the corresponding leg of $G_{\gamma}$ contains a unique cone of $\tilde{\omega}_l^\circ$ not as a face by construction. Additionally by construction, we have a morphism $\tilde{\omega}_l^\circ \rightarrow \Sigma(\tilde{X})$ lifting $\omega_l^\circ \rightarrow \Sigma(X)$. For $x \in V(G_\gamma) \cup E(G_\gamma) \cup L(G_\gamma)$, we set $\pmb\sigma(x)$ to be the minimal cone of $\Sigma(\tilde{X})$ containing the image of the cone associated to $x$. For an edge $e' \in E(G_\gamma)\cup L(G_\gamma)$ subdividing an edge $e \in E(G_\tau)\cup L(G_\tau)$, we may identify $\pmb\sigma(e')^{gp}$ with $\pmb\sigma(e)^{gp}$, and we define $\textbf{u}(e') = \textbf{u}(e)$. It is now immediate that $\tilde{\Gamma}^\circ \rightarrow \Sigma(\tilde{X})$ is a punctured tropical map over $\omega$ of tropical type $\gamma$.

For the injection, note that since each $\tilde{\Gamma}_{\tau}^\circ|_{\omega}$ is a family of tropical type of type $\gamma$, we have a classifying morphism $\omega \rightarrow \gamma$. Additionally, there is a morphism $\gamma \rightarrow \tau$ which has a modular interpretation given by considering the universal map $\Gamma_\gamma^\circ \rightarrow \Sigma(X)$ forgetting vertices subdividing edges of $G_\tau$, and extending the legs of $\Gamma^\circ_\omega$ so that the resulting tropical family is prestable. The resulting morphism $\gamma \rightarrow \tau$ is clearly injective, and its image contains the union of the cones $\omega \subset \tau$ with associated tropical type $\gamma$. The composite map of the two maps considered above from the previous union of subcones of $\tau$ to $\tau$ is clearly the inclusion, and the desired injection statement follows.
\end{proof}

\begin{remark}
%Garbled sentence
%By the previous proposition, we can merge appropriate pairs of cones of $\tilde{\tau}$ sharing a codimension $1$ with same associated tropical type. The resulting subdivision $\tilde{\tau} \rightarrow \tau$ now only depends on the subdivision $\tilde{\Gamma}_{\tau}^\circ \rightarrow  \Gamma_{\tau}^\circ$, given by the coarsest subdivision of $\tau$ whose cones $\omega$ satisfy $\tilde{\Gamma}_{\omega}^\circ \rightarrow \Sigma(\tilde{X})$ is a family of tropical curves of a fixed type $\gamma$ associated with $\omega$. Moreover, the families of tropical curves $\tilde{\Gamma}_{\omega}$ defined above each pullback from the unpunctured curve $\Gamma_{\gamma}$ over a subcone of $\gamma$.

By the previous proposition, after merging pairs of cones of $\tilde{\tau}$ sharing codimension $1$ faces with same associated tropical types, the resulting subdivision $\tilde{\tau} \rightarrow \tau$ now only depends on the subdivision $\tilde{\Gamma}_{\tau}^\circ \rightarrow \Gamma_{\tau}^\circ$, given by the coarsest subdivision of $\tau$ such that for any cone $\omega \subset \tilde{\tau}$, we have $\tilde{\Gamma}_{\omega}^\circ \rightarrow \Sigma(\tilde{X})$ is a family of tropical curves of a fixed type. Moreover, the families of tropical curves $\tilde{\Gamma}_{\omega}$ defined above each pullback from the unpunctured curve $\Gamma_{\gamma}$ over a subcone of $\gamma$.

\end{remark}

We note that by construction, the tropical types $\gamma$ produced in Proposition \ref{tcurve} are tropical lifts of $\tau$. After conducting the merging of the cones of $\tilde{\tau}$ described in the remark above, for future purposes, we consider how we may produce all remaining tropical lifts $\omega$ of $\tau$ from the subdivision $\tilde{\tau}$ produced above. First, we call the types corresponding to cones of $\tilde{\tau}$ \emph{maximally extendable} tropical lifts. To produce all remaining tropical lifts, we note that for each maximally extendable tropical type $\omega$ produced above and each leg $l \in L(G_\tau)$, there exists a total ordering on the cones $\omega_e \subset \Gamma_{\omega}$ on the edges $e \in L(G_\omega)\cup E(G_\omega)$ which map to $l$ under the subdivision morphism $G_\omega \rightarrow G_\tau$. Each choice of edge or leg $e$ above produces a family of tropical curves tropical in which we remove the edges and legs $e'$ such that $e' > e$. This family has a fixed tropical type $\omega'$ and in this situation, we say that $\omega'$ is produced from the maximally extendable tropical type $\omega$, which we denote by $\omega \ge \omega'$.

\begin{figure}[h]
\centering
\begin{tikzpicture}
\fill[white!80!blue, path fading = north] (0,0)--(4,0)--(0,4)--cycle;
\draw[black, dashed] (0,0)--(4/3,8/3);
\draw[black, dashed] (0,0)--(8/3,4/3);
\draw[ball color = red] (1/2,0) circle (0.5mm);
\draw[ball color = red] (1/2,1/4) circle (0.5mm);
\draw[ball color = red] (1/2,1) circle (0.5mm);
\draw[color = red] (1/2,0)--(1/2,1/4);
\draw[color = red] (1/2,1/4)--(1/2,1);
\draw[->,color = red] (1/2,1)--(1/2,7/2);

\draw[ball color = violet] (1,0) circle (0.5mm);
\draw[ball color = violet] (1,1/2) circle (0.5mm);
\draw[color = violet] (1,0)--(1,1/2);
\draw[->,color = violet] (1,1/2)--(1,2);

\draw[ball color = blue] (3/2,0) circle (0.5mm);
\draw[->, color = blue] (3/2,0)--(3/2,3/4);
\end{tikzpicture}
\caption{Different choices of puncturing of a leg of $\tau$ mapping to the depicted cone giving different tropical types. The leftmost depicted leg is maximally extended.}
\end{figure}

Now given a tropical lift $\gamma$ of $\tau$, consider the subcomplex $\tilde{\gamma}$ of $\tilde{\tau}$ consisting of faces of cones with associated tropical types $\omega$ such that $\omega \ge \gamma$. Since the edge lengths of $\tilde{\Gamma}_{\omega}$ take values in $\omega_{\mathbb{Q}}^{\vee}$ and not typically in $\omega_{\NN}^{\vee}$, they do not in general define a family of tropical curves as defined in Section $2$. However, since there are finitely many edges in $G_\omega$, there exists a maximal finite index sublattice of $\omega_\NN^{gp}$, whose points in the cone $\omega \subset \omega^{gp}$ correspond to tropical curves with edge lengths in $\omega_{\NN}^{\vee}$. 

The construction of the families of tropical curves above give morphisms $\omega \rightarrow \gamma$ of cones. We note in the following lemma that these morphisms identify $\tilde{\gamma}$ with a subdivision of $\gamma$. Specifically, $\tilde{\gamma}$ is the cone complex giving the moduli space of maximally extended lifts $\omega$ of $\tau$ such that $\omega \ge \gamma$.

\begin{lemma}\label{tlemma}
There exists a morphism $\tilde{\gamma} \rightarrow \gamma$ identifying $\tilde{\gamma}$ as a subdivision of $\gamma$.
\end{lemma}

\begin{proof}
Given $\omega \in \tilde{\gamma}$, note the induced morphism $i: \omega \rightarrow \gamma$ has the modular definition given by taking a punctured tropical map $\Gamma^\circ_q \rightarrow \Sigma(\tilde{X})$ associated with $q \in \omega$, and forgetting legs/edges associated with certain linear subgraphs of the dual graph $G_\omega$, producing a punctured tropical map $\Gamma^{'\circ}_{i(q)} \rightarrow \Sigma(\tilde{X})$ of type $\gamma$. Since at least one segment from each of the linear subgraphs remain in $G_\gamma$, we may uniquely reconstruct the punctured tropical map $\Gamma^\circ_q \rightarrow \Sigma(\tilde{X})$ with maximally extended leg from $\Gamma^{'\circ}_{i(q)} \rightarrow \Sigma(\tilde{X})$, implying $\omega \rightarrow \gamma$ is injective. The previous fact also makes clear that for every point $p \in int(\gamma)$, there is a cone $\omega$ of $\tilde{\tau}$ such that $p$ is in the image of $\omega \rightarrow \gamma$. Hence, the morphisms $\omega \rightarrow \gamma$ collectively surject, and $\tilde{\gamma} \rightarrow \gamma$ is a subdivision.
\end{proof}

\begin{figure}[h]
\centering
\begin{tikzpicture}

\node at (-3,2,0) {$\Sigma(\tilde{X})$};
\node at (-3,-2,0) {$\tilde{\tau}$};
\fill[white!60!violet, path fading = east] (-2,0,0)--(4,0,0)--(0,0,4)--cycle;
\fill[white!20!blue, path fading = east] (-2,0,0)--(4,0,0)--(2,4,2)--cycle;
\fill[white!40!blue, path fading = east] (-2,0,0)--(0,0,4)--(2,4,2)--cycle;
\fill[white!40!blue, path fading = east] (-2,0,0)--(4,0,0)--(4,6,0)--(-2,6,0)--cycle;
\fill[white!50!blue] (-2,0,0)--(2,4,2)--(2,6,2)--(-2,6,0)--cycle;
\fill[white!30!blue] (-2,0,0)--(0,0,4)--(0,6,4)--(-2,6,0)--cycle;

\fill[white!40!violet, path fading=east] (-2,-2,0)--(4,-2,0)--(0,-2,4)--cycle;
\draw[color=black] (-2,-2,0)--(2,-2,2);
\draw[color=black] (-2,0,0)--(4,0,0)--cycle;
\draw[color=black] (4,0,0)--(2,4,2);
\draw[color=black] (0,0,4)--(2,4,2);
\draw[color=black] (-2,0,0)--(0,0,4)--cycle;
\draw[color=black] (-2,0,0)--(2,4,2)--cycle;
\draw[color=black] (-2,0,0)--(-2,6,0);

\draw[->,color=red] (2,0,2)--(2,4,2)--cycle;
\draw[ball color=red] (2,0,2) circle (0.75mm);
\draw[ball color=red] (2,4,2) circle (0.75mm);
\draw[-,dashed] (-2,0,0)--(2,0,2)--cycle;
\draw[color=black] (2,4,2)--(2,6.5,2);
\draw[->,color=red] (2,4,2)--(2,6.25,2);
\end{tikzpicture}
\caption{An example of how a subdivision of the tropical target induces a subdivision of the tropical moduli problem. The subdivision of the target above is a barycentric subdivision of $\mathbb{R}_{\ge 0}^3$, and the tropical type $\tau$ has a associated cone $\mathbb{R}_{\ge 0}^2$ given below.}
\end{figure}

\begin{figure}[h]
\centering
\begin{tikzpicture}
\node at (-1,2) {$\Sigma(\tilde{X})$};
\node at (-1,-2) {$\gamma$};
\fill[white!70!blue, path fading = north] (0,0)--(4,0)--(0,4)--cycle;
\draw[black] (0,0)--(4/3,8/3);
\draw[black](0,0)--(4,0);
\draw[black](0,0)--(0,4);
\draw[ball color=green] (2,2) circle (0.5mm);
\draw[,color=red](2,2)--(1,2);
\draw[->,color=red](2,2)--(2,0);
\draw[->,color=red](1,2)--(0,2);
\draw[ball color=red](1,2) circle (0.5mm);

\fill[white!70!violet, path fading= north] (0,-4)--(4,-4)--(4/3,-4/3)--cycle;
\draw[color=black] (0,-4)--(4,-4);
\draw[color=black] (0,-4)--(4/3,-4/3);
\draw[ball color= green] (0,-4) circle (0.5mm);
\draw[ball color = green] (1,-4) circle(0.5mm);
\draw[ball color=green] (2,-4) circle (0.5mm);
\draw[ball color=green] (1,-2) circle (0.5mm);
\draw[ball color = green] (3,-4) circle (0.5mm);
\draw[ball color= green] (4,-4) circle (0.5mm);
\draw[color=black] (1,-3) circle (0.5mm);
\draw[color=black] (2,-3) circle (0.5mm);
\draw[ball color=green] (2,-2) circle (0.5mm);
\draw[color=black] (3,-3) circle (0.5mm);

\end{tikzpicture}
\caption{An example showing the necessity of lattice coarsening to ensure integer valued edge lengths after a subdivision. Non-filled circles indicate integral points in $\tau_\NN \setminus \gamma_{\NN}$, after identifying $\gamma$ as a subcone of $\tau$. }
\end{figure}

We now fix a tropical lift $\gamma$ of $\tau$, and let $m$ denote the order of the torsion part of the quotient $\tau_{\NN}^{gp} /\gamma_{\NN}^{gp}$. 

The inclusion of cones $\gamma \rightarrow \tau$ induces a morphism of Artin cones $i: \mathcal{A}_{\gamma} \rightarrow \mathcal{A}_{\tau}$. Both of these Artin cones are equipped with puncturing ideals, associated with the universal families of punctured tropical curves of type $\gamma$ and $\tau$, defined in \cite{punc} Definition $3.22$, which we denote by $\mathcal{J}$ and $\mathcal{I}$ respectively. In our case, since the tropical types in question are realizable, the simpler definition from \cite{punc} Proposition $3.23$ suffices. They are given in both cases by the ideal of non-zero elements of $\gamma^{\vee}_{\NN}$ and $\tau^{\vee}_{\NN}$ i.e., all elements of $\gamma^{\vee}_{\NN}$ and $\tau^{\vee}_{\NN}$ which induce strictly positive functions on $int (\gamma)$ and $int(\tau)$. Let $\mathcal{A}_{\gamma,\gamma}$ and $\mathcal{A}_{\tau,\tau}$ be the respective idealized Artin cones. Since the induced morphism $i^{\vee}: \tau_\NN^{\vee} \rightarrow \gamma_\NN^{\vee}$ is sharp, we have $i^*(\mathcal{I}) \subset \mathcal{J}$. Thus, we have an induced map $\mathcal{A}_{\gamma,\gamma} \rightarrow \mathcal{A}_{\tau,\tau}$. By \cite{punc} Lemma B.3, this morphism is idealized log \'etale.

The following lemma essentially implies the tropical analogues of Theorem \ref{mthm1} and \ref{mthm2}, and we will make use of this lemma in the proofs of Theorems \ref{mthm1} and \ref{mthm2} later on.

\begin{lemma}\label{DM}
Let $f: \Ag \rightarrow \At$ be a morphism of idealized Artin cones such that $f$ is the restriction of a map $f:\mathcal{A}_{\gamma'} \rightarrow \mathcal{A}_{\tau'}$ induced by an inclusion of cones $\gamma' \rightarrow \tau'$. Then $f$ is DM type. If dim $\tau = $dim $\gamma$, $\gamma' = \gamma$, and $\tau' = \tau$, then the degree of $f$ is  $\frac{1}{m}$, where $m = |coker (\gamma^{gp}_{\NN} \rightarrow \tau^{gp}_{\NN})|$

\end{lemma}

\begin{proof}

Write $\Ag = [S_{\gamma',\gamma}/T_\gamma]$ and $\At = [S_{\tau',\tau}/T_{\tau}]$. By \cite{toricglue} Corollary $A.3$, we have the following cartesian diagram in all categories:

\begin{equation}
\begin{tikzcd}
\left[ S_{\gamma',\gamma} \times T_{\tau'}/T_{\gamma'} \right]  \arrow{r}\arrow{d} & S_{\tau',\tau} \arrow{d}\\
\Ag \arrow{r}{f} & \At
\end{tikzcd}
\end{equation}

In the above diagram $S_{\tau',\tau} \rightarrow \At$ is the canonical smooth chart for $\At$, the action of $T_{\gamma'}$ on $S_{\gamma',\gamma} \times T_\tau$ is given on the factor $S_{\gamma',\gamma}$ by restricting the standard toric action on $S_{\gamma'}$, and is given on the factor $T_{\tau'}$ by the action induced by the group homomorphism $T_{\gamma'} \rightarrow T_{\tau'}$ induced by the toric map $S_{\gamma'} \rightarrow S_{\tau'}$. Since a morphism of stacks is DM type if it is DM after base changing along a smooth chart of the base by \cite{SP} Lemma 06TZ, it suffices to show the upper arrow is DM type, for which it suffices to show $[S_{\gamma',\gamma} \times T_{\tau'}/T_{\gamma'}]$ is Deligne-Mumford. Finally, we have a representable projection $[S_{\gamma',\gamma} \times T_{\tau'}/T_{\gamma'}] \rightarrow [T_{\tau'}/T_{\gamma'}]$. If $[T_{\tau'}/T_{\gamma'}]$ is Deligne-Mumford, the conclusion follows. To see the final quotient stack is Deligne-Mumford, note that the morphism $T_{\gamma'} \rightarrow T_{\tau'}$ is induced by the corresponding morphism of cocharacter lattices, $(\gamma')_{\mathbb{Z}}^{gp} \rightarrow (\tau')^{gp}_{\mathbb{Z}}$. By assumption, we know this is an injective morphism of finitely generated abelian groups. Thus, we have a short exact sequence:

\[0 \rightarrow (\gamma')^{gp}_{\mathbb{Z}} \rightarrow (\tau')^{gp}_{\mathbb{Z}} \rightarrow coker(i)_{tor} \oplus coker(i)_{free} \rightarrow 0\]

In the rightmost non-trivial object above, $coker(i)_{free}$ refers to the torsion free component of the cokernel of $i$, while $coker(i)_{tor}$ refers to the torsion part. Note that the torsion part is a finite abelian group. Since $char\text{ }\kk = 0$, taking the long exact sequence associated to tensoring with $k^*$ gives:

\[0 \rightarrow coker(i)_{tor} \rightarrow T_{\gamma'} \rightarrow T_{\tau'} \rightarrow T_{coker(i)_{free}} \rightarrow 0\]

Thus, the stabilizer of a geometric point $[T_{\tau'}/T_{\gamma'}]$ is isomorphic to $coker(i)_{tor}$, which is a finite reduced group. Hence, $[T_{\tau'}/T_{\gamma'}]$ is Deligne-Mumford, and after unwinding implications, $\Ag \rightarrow \At$ is DM type. 

For the second statement, observe $\mathcal{A}_{\tau,\tau} = \mathcal{A}_{\gamma,\gamma} = B\mathbb{G}_m^n$. Therefore, the degree is $deg\text{ }[T_{\tau'}/T_{\gamma'}]$, which we have seen equals $\frac{1}{|coker(i)_{tor}|} = \frac{1}{|coker(i)|}$ since the torsion-free part must be $0$ for dimension reasons and our assumption that $\gamma \rightarrow \tau$ is an inclusion of cones. 
\end{proof}

Having now fixed a tropical lift $\gamma$ of $\tau$, and established a tropical version of our main result, we seek to relate the two corresponding logarithmic moduli problems $\mathscr{M}(\tilde{X},\gamma)$ and $\mathscr{M}(X,\tau)$.

We start with a lemma describing the pushforward of the log structure of a family of punctured log curves through a partial stabilization map.

\begin{lemma}\label{plog}
Let $(C/S,f)$ be a prestable punctured log map to $X$ and $c: C^u \rightarrow \overline{C}^u$ be a partial stabilization of the underlying family of maps i.e., fiberwise a contraction of rational components with at most two special points. Then $\overline{C} \rightarrow S$ can be given the structure of a family of punctured log curves, with $c_*\mathcal{M}_{C^\circ} = \mathcal{M}_{\overline{C}^\circ}$, which is a puncturing of the log structure with sheaf of monoids $c_*\mathcal{M}_{C}$. Moreover, the factorization of $C \rightarrow X$ through the partial stabilization in the category of schemes lifts to a factorization in log schemes.

\end{lemma}

\begin{proof}
This argument runs similar to \cite{stacktrop} Lemma $8.8$, but we must attend to some remaining details coming from the puncturing. The proof that both $\mathcal{M}_{\overline{C}}$ and $\mathcal{M}_{\overline{C}^\circ}$ are log structures is exactly as in \cite{stacktrop} Lemma $8.8$, and since $\overline{C} \rightarrow S$ is proper with connected geometric fibers, we only need to understand the local structure around a geometric point of $\overline{C}$ over a geometric point of $s \in S$. Thus, we assume $S$ is atomic, and $s \in S$ is contained in the unique closed stratum, with $\overline{\mathcal{M}}_{S,s} = \gamma^{\vee}_\NN$ for some cone $\gamma$. Moreover, the tropicalization is given by a map of cone complexes $\Gamma^\circ \rightarrow \gamma$. We note that $S$ has an idealized log structure $\mathcal{I} \subset \mathcal{M}_S$ coming from the punctured log curve $C^\circ \rightarrow S$. Letting $y\in \overline{C}$ be a geometric point over $s \in S$, if the fiber of the partial stabilization map $C \rightarrow \overline{C}$ over $y \in \overline{C}$ is a point, or does not contain a punctured point, then either the result is trivial or follows directly from \cite{stacktrop}. Thus, suppose there exists a punctured point $y' \in C$ which maps to $y$ over $s \in S$, and let $\gamma_{l'}$ be the cone of $\Gamma^\circ$ associated with $y'$. Consider the ghost sheaf sequence:

\[0\rightarrow \mathcal{O}_C^{\times} \rightarrow \mathcal{M}_{C^\circ} \rightarrow \overline{\mathcal{M}}_{C^\circ} \rightarrow 0 .\]

By \cite{compGW} Lemma $B.4$, $c_*\mathcal{M}_C$ commutes with base change on $S$. After restricting to the fiber over $y$, we push this sequence forward along the contraction. By taking the associated long exact sequence, restricting the sequence to appropriate submonoids, and recalling $c_*(\mathcal{O}^\times_{C}) = \mathcal{O}^\times_{\overline{C}}$, we produce the exact sequence:

\[0\rightarrow \mathcal{O}^{\times}_{\overline{C},y} \rightarrow \mathcal{M}_{\overline{C}^\circ,y} \rightarrow (c_*\overline{\mathcal{M}}_{C^\circ})_{y} \rightarrow Pic(c^{-1}y). \]

Since $c$ only contracts rational unstable components, by taking multidegrees of a line bundle on a rational nodal curve, we may identify the last term with $\mathbb{Z}^{|V|}$, $V$ being the set of vertices in the dual graph of $c^{-1}y$. Since $C \rightarrow S$ is a log curve, this graph is canonically metrized by elements of $\gamma_{\NN}^{\vee}$. After making this identification as well as the identification of the monoid $c_*(\overline{\mathcal{M}}_{C^\circ})_y$ with $\gamma^{\vee}$ valued PL functions on the meterized dual graph of $c^{-1}y$ with slopes along edges integer multiples of the edge lengths which take values in $\gamma_\NN^\vee$ for vertices of the dual graph, the connecting homomorphism sends a $PL$ function to the sum of the outgoing slope magnitudes at each vertex.  In particular, every element $f \in \gamma_{l',\NN}^{\vee,gp}$ uniquely determines an element of $\overline{\mathcal{M}}_{\overline{C}^\circ,y}^{gp}$, by noting there is a unique balanced $\gamma_\NN^{\vee,gp}$ valued PL function on the dual graph of $c^{-1}y$ which extends $f$. The monoid $\overline{\mathcal{M}}_{\overline{C}^{\circ},y}$ is then identified with the collection of linear functions $f \in \overline{\mathcal{M}}_{C^\circ,y'} \subset \gamma_{l',\NN}^{\vee}$ which induce a balanced $\gamma_{\NN}^{\vee}$ valued function on the dual graph of $c^{-1}y$. Since $\mathcal{M}_{C^\circ}$ contains the log structure of the non-punctured curve, the pushforward log structure contains a sub log structure which realizes $\overline{C}$ as an ordinary log curve. The elements in $\mathcal{M}_{\overline{C}^\circ,y} \setminus \mathcal{M}_{\overline{C},y}$ have image in $\overline{\mathcal{M}}_{\overline{C}^\circ,y}$ corresponding to PL functions on the dual graph of $c^{-1}y$ with negative slope along the leg associated with $y$. Since these functions are non-negative across the entire leg, they are in particular non-negative at the unique vertex contained in the leg. Thus, we have the containment $\overline{\mathcal{M}}_{\overline{C}^\circ,y} \subset \gamma^{\vee}_\NN \oplus \mathbb{Z}$.

To see that $\mathcal{M}_{\overline{C}^\circ}$ is locally a puncturing around $y$, consider $m \in \mathcal{M}_{\overline{C}^\circ,y} \setminus \mathcal{M}_{\overline{C},y}$, with image in the ghost sheaf denoted by $\overline{m} \in \overline{\mathcal{M}}_{\overline{C}^\circ,y}$. Note the map on stalks induces an embedding $\overline{\mathcal{M}}_{\overline{C}^\circ,y} \rightarrow \overline{\mathcal{M}}_{C^\circ,y'}$, which restricts to the embedding $\overline{\mathcal{M}}_{\overline{C},y} \rightarrow \overline{\mathcal{M}}_{C,y'}$. To better understand these maps, note first that by \cite{stacktrop} Lemma $8.8$, $\overline{\mathcal{M}}_{\overline{C},y}$ is isomorphic to $\gamma^\vee_{\NN} \oplus \mathbb{N}$. Letting $\rho$ be the sum of the lengths of all edges of $\Gamma$ associated to nodes mapping down to $y$, the monoid $\overline{\mathcal{M}}_{C,y'}$ consists of elements of $\gamma_\NN^{gp} \oplus \ZZ$ which are non-negative on the upper convex hull of the graph of $\rho$ in $\gamma^{gp}\oplus \mathbb{R}_{\ge 0}$. Thus, we have the isomorphism: 

\begin{equation}\label{exmon}
\overline{\mathcal{M}}_{C,y'} \cong \gamma^{\vee}_{\NN} \oplus \mathbb{N} + \langle( -\rho , 1)\rangle \subset (\gamma^{\vee}_{\NN})^{gp} \oplus \mathbb{Z}.
\end{equation}
Moreover, the morphism $\overline{\mathcal{M}}_{\overline{C},y} \rightarrow \overline{\mathcal{M}}_{C,y'}$ is simply the inclusion morphism. Furthermore, the morphism $\overline{\mathcal{M}}_{\overline{C}^\circ,y} \rightarrow \overline{\mathcal{M}}_{C^\circ,y'}$ is induced by restriction of the induced isomorphism $\overline{\mathcal{M}}_{\overline{C},y}^{gp} \rightarrow \overline{\mathcal{M}}_{C,y'}^{gp}$. Now write $\overline{m} = (\overline{m}_1,\overline{m}_2) \in \gamma^{\vee}_{\NN} \oplus \mathbb{Z}$, and note that since $\overline{m} \in \overline{\mathcal{M}}_{\overline{C}^\circ,y}\setminus \overline{\mathcal{M}}_{\overline{C},y}$, $m_2$ must be in $\ZZ_{< 0}$. In particular, we must have $\overline{m} \in \overline{\mathcal{M}}_{C^\circ,y'} \setminus \overline{\mathcal{M}}_{C,y'}$. Since $C \rightarrow S$ is a punctured log curve, by definition of the idealized log structure on $S$, we must have $\overline{m}_1 \in I \subset \gamma^{\vee}_{\NN}$. In particular, for any lift $m_1 \in \mathcal{M}_{S,s}$ of $\overline{m}_1$, we have $\alpha_S(m_1) = 0$. Since $-\overline{m}_2 \in \mathbb{N}_{>0} \subset \mathbb{Z}$, it follows that $(0,-\overline{m}_2) \in \overline{\mathcal{M}}_{\overline{C},y}$ and for $m_2 \in \mathcal{M}_{\overline{C}^\circ,y}^{gp}$ a lift of $\overline{m_2}$, $\overline{m_2^{-1}m} = \overline{m_1}$. Hence $\alpha_{\overline{C}^\circ}(-m_2^{-1}m) =  \alpha_{\overline{C}^\circ}(m_2^{-1})\alpha_{\overline{C}^\circ}(m) = 0$. Since $\alpha_{\overline{C}^\circ}(m_2^{-1}) \in \mathcal{O}_{\overline{C},y}$ is a power of coordinate on $\overline{C}$ vanishing at $y$, and $\overline{C}$ is locally smooth around $y$, we must have $\alpha_{\overline{C}^\circ}(m) = 0$, as required. 

The final factorization statement follows by adjunction of the pair $(c_*,c^*)$ of functors on the categories of sheaves of monoids. 
\end{proof}

By the description of the ghost sheaf $\overline{\mathcal{M}}_{\overline{C}^\circ}$ above, we see that the $\Sigma(\overline{C}^\circ)$ is the family of tropical curves parameterized by $\Sigma(S)$, which over $s \in \Sigma(S)$ is given by a tropical curve $\overline{\Gamma}^\circ_s$ having underlying graph given by forgetting the vertices associated to contracted components, and gluing resulting open edges. 

With the above lemma in mind, if we are given a stable punctured log map $C^\circ \rightarrow \tilde{X}$, then by composing the morphism with the log \'etale map $\tilde{X} \rightarrow X$, stabilizing the underlying map, and pushing forward the log structure, we produce a punctured log map $\overline{C}^\circ \rightarrow X$. We wish to relate the tropical type of the original stable log map $C^\circ \rightarrow \tilde{X}$ to the tropical type of $\overline{C}^\circ \rightarrow X$. The following theorem shows precisely what this relation must be.

\begin{theorem}\label{tlift}
Let $S = $Spec $\kk$, and $(C/S,f: C \rightarrow \tilde{X})$ be a basic prestable log curve, with tropical type $\gamma$. Then the underlying map of schemes $C^u \rightarrow \tilde{X}$ is stable if and only if $\gamma$ is a tropical lift of the tropical type $\tau$ of a basic stable log map $(\overline{C}/S, \overline{f}: \overline{C} \rightarrow X)$. 
\end{theorem}

\begin{proof}
Let $f: C \rightarrow \tilde{X}$ be a basic prestable punctured log curve over a base $(S,\mathcal{M}_{S,\gamma})$. By stabilizing the induced map $C \rightarrow X$ and pushing forward the puncturing on $C$ to the partial stabilization $\overline{C}$ as in Lemma \ref{plog}, we produce a stable log map $\overline{C} \rightarrow X$, which is pulled back from a basic stable log map $C^{stab} \rightarrow X$. Assuming $C$ is not an unmarked elliptic curve, an exceptional case we will handle in a final remark in this proof, the partial stabilization $c: C \rightarrow C^{stab}$ contracts $1$ and $2$ valent rational components, and the universal map to the basic log curve gives a morphism $\mathcal{M}_{S,\tau} \rightarrow \mathcal{M}_{S,\gamma}$ of log structures on the base $S$. If $\gamma$ is a tropical lift of $\tau$, for $\tau$ the tropical type of the log map $C^{stab} \rightarrow X$, then partial stabilization only contracts components corresponding to vertices of $G_{\gamma}$ contained in two elements of $E(G_\gamma)\cup L(G_\gamma)$ mapping to distinct cones of $\Sigma(\tilde{X})$. The tropical condition forces the component $C_v$ associated with $v$ to have points mapping to distinct strata of $\tilde{X}$. Indeed, dim $\pmb\sigma(v)<dim\text{ }\pmb\sigma(e)$ for $e \in E(G_\gamma)\cup L(G_\gamma)$ an edge containing $v$. In particular, $f$ does not contract $C_v$. Therefore, the map $f$ is stable. 

Now suppose to the contrary that $\gamma$ is not a tropical lift of $\tau$. Note that the graph $G_{\tau}$ associated with $\tau$ is obtained from $G_\gamma$ by forgetting appropriate vertices associated with contracted components of $C$. Given a $1$-valent vertex of $G_{\gamma}$ forgotten by $c$ contained in a unique edge $E$, by the logarithmic balancing condition applied to the map $C \rightarrow X$, we must have $\textbf{u}(E) = 0$. We call such a vertex a contracted $1$ valent vertex. Given a $2$-valent vertex of $G_{\gamma}$ associated with a contracted component such that the adjacent edges are contained in the same cone of $\Sigma(\tilde{X})$, their associated contact order must sum to $0$ again by the log balancing condition. In this case, $v$ must map into the interior of the cone mapped to by both adjacent edges, and we call such a vertex an unconstrained balanced $2$-valent vertex. Since $\gamma$ is not a tropical lift of $\tau$, $G_{\gamma}$ must have either a contracted $1$-valent vertex or an unconstrained balanced $2$-valent vertex. Letting $v \in V(G_{\gamma})$ be one such vertex, and $C_v$ the associated component of $C$, denote by $\overline{C} \rightarrow X$ the punctured log map given by contracting the component $C_v$ and pushing forward the log structure as in Lemma \ref{plog}. By composing with the strict morphism $X \rightarrow \mathcal{A}_X$, we also produce a log map $\overline{C} \rightarrow \mathcal{A}_X$. 

We now wish to find a morphism $\overline{C} \rightarrow \widetilde{\mathcal{A}}_X$ lifting $\overline{C} \rightarrow \mathcal{A}_X$. We will use the characterization of morphisms to subdivisions of Artin fans given in Lemma \ref{mtoA}. We start with noting that since $v$ is assumed either to be a contracted $1$-valent vertex or a unconstrained balanced $2$-valent vertex, in both cases, we have $\Sigma(C^\circ) \rightarrow \Sigma(\tilde{X})$ factors through the contraction $\Sigma(C^\circ) \rightarrow\Sigma(\overline{C}^\circ)$. Since $\overline{C}^\circ$ is fine and saturated away from punctured points, the factorization condition of item ii of Lemma \ref{mtoA} needs to only be checked in an atomic neighborhood of a punctured point $p$. Let $\gamma_{l'} \in \Sigma(C^\circ)$ and $\gamma_{l} \in \Sigma(\overline{C}^\circ)$ be the cones corresponding to punctured points $p'$ and $p$ of $C$ and $\overline{C}$ respectively such that the contraction morphism $\Sigma(C^\circ) \rightarrow\Sigma(\overline{C}^\circ)$ maps $\gamma_{l'}$ to $\gamma_l$. By the description of the monoid $Q_p:= \overline{\mathcal{M}}_{\overline{C}^\circ,p}$ associated to punctured point of $p \in \overline{C}^\circ$ described in Lemma \ref{plog}, we have $Q_p$ is the intersection of the image of $\gamma^{\vee}_{l,\NN} \subset \gamma^{\vee}_{l',\NN}$ with the prestable submonoid $Q_{p'} \subset Q_{p'}^{sat} = \gamma^{\vee}_{l',\NN}$. By the tropical factorization, we have the monoid morphism $\pmb\sigma(l')^{\vee}_{\NN} \rightarrow \gamma^{\vee}_{l',\NN}$ factors through $\gamma^{\vee}_{l,\NN} \subset \gamma^{\vee}_{l',\NN}$. Since the prestable monoid $Q_{p'}$ contains the image of $\pmb\sigma(l')^{\vee}_\NN \rightarrow \gamma^{\vee}_{l',\NN}$, we have the morphism $\pmb\sigma(l')^{\vee}_\NN \rightarrow \gamma^{\vee}_{l,\NN}$ factors through $Q_p \subset \gamma^{\vee}_{l,\NN}$. By pulling back the \'etale cover of $\mathcal{A}_X$ by Artin cones, we see that \'etale locally we have a morphism $\Sigma(\overline{C}^\circ) \rightarrow \Sigma(\tilde{X})$ satisfying the factorization property given in Lemma \ref{mtoA}. The locally defined morphism thus extends to a global morphism on $\overline{C}^{\circ}$.

This gives the remaining data necessary to define a morphism $\overline{C}^\circ \rightarrow \widetilde{\mathcal{A}}_X$ making the following commutative diagram:

\[\begin{tikzcd}
C^\circ \arrow{r} \arrow{d} & \tilde{X} \arrow{r} \arrow{d} & \widetilde{\mathcal{A}}_X \arrow{d} \\
\overline{C}^\circ \arrow{r} \arrow[rru,crossing over] \arrow[ru,dashrightarrow,crossing over] & X \arrow{r} & \mathcal{A}_X
\end{tikzcd}\]

Using the pullback diagram \ref{mod}, we have a filling of the dashed arrow making the above diagram commute. After taking underlying schemes, we see that $C^u \rightarrow \tilde{X}^u$ factors through a contraction of a rational component, hence cannot be stable. 

When $C$ is an unmarked elliptic curve, the ghost sheaf is the constant sheaf associated with the monoid $\pmb\sigma(v)^{\vee}_\NN$, for $v$ the unique vertex of $G_\gamma$. The tropical factorization problem is also solved in this case, and thus $C \rightarrow \tilde{X}$ factors through the stabilization of the underlying map $C \rightarrow X$. Since $C \rightarrow \tilde{X}$ cannot factor through a point, we must have $C \rightarrow X$ does not have image a point. In this case, $\gamma$ must be a tropical lift of the tropical type of $C \rightarrow X$ as well. 
\end{proof}

Now consider the DM stack $\mathscr{M}_\gamma:= \mathscr{M}(\tilde{X}/B,\gamma)$, and the corresponding algebraic stack of log curves to the relative Artin fan $\mathfrak{M}_\gamma := \mathfrak{M}(\tilde{\mathcal{X}}/B,\gamma)$. As with all log stacks of punctured log curves, $\mathfrak{M}_\gamma$ comes equipped with an idealized log structure $\mathcal{J} \subset \mathcal{M}_{\mathfrak{M}_\gamma}$ which is pulled back from an ideal $\overline{\mathcal{J}} \subset \overline{\mathcal{M}}_{\mathfrak{M}_\gamma}$. We construct a stabilization morphism $\mathscr{M}_\gamma \rightarrow \mathscr{M}_\tau$ as follows: On an object $(C^\circ/S,f)$, it is defined by composing $f$ with the log \'etale map $\pi: \tilde{X} \rightarrow X$, followed by stabilizing the underlying curve, pushing forward the log structure to the partially stabilized curve $\overline{C}^u$ and taking the unique basic family of punctured log curves associated with this family. Note that by identifying contact orders on $\Sigma(\tilde{X})$ with contact orders in $\Sigma(X)$ via the map of cone complexes $\Sigma(\tilde{X}) \rightarrow \Sigma(X)$, the tropical class $\beta'$ of $\gamma$, in the sense of \cite{punc} Definition $3.4$, specifies the tropical class $\beta$ of the stabilization of $C \rightarrow X$. As an intermediate step to showing the above object assignment extends to a morphism of DM stacks, we show the following:

\begin{proposition}\label{stab}
The stabilization map $s: \mathscr{M}(\tilde{X}/B, \gamma) \rightarrow \mathscr{M}(X/B,\beta)$ is a morphism of Deligne-Mumford stacks.
\end{proposition}

\begin{proof}

On the level of underlying stable curves, this is classical. That is, for a morphism of families of stable maps $(C'/S',f') \rightarrow (C / S,f)$ to $\tilde{X}$, we have a morphism $\overline{C}' \rightarrow \overline{C}$ making the following diagram commutative, with left composite rectangle and bottom left square cartesian in the category of schemes:

\[
\begin{tikzcd}
C' \arrow{r}\arrow{d} & C \arrow{d} \arrow{r} & \tilde{X} \arrow{d}\\
\overline{C}' \arrow{r} \arrow{d} \ar[dr,phantom,"\square"] & \overline{C} \arrow{d} \arrow{r} & X \arrow{d}\\
S' \arrow{r} & S \arrow{r} & B
\end{tikzcd}
\]

Now suppose we have $(C' /S', f')$ and $(C /S,f)$ are families of punctured log curves associated with strict morphisms $S' \rightarrow \mathscr{M}(\tilde{X}/B,\gamma)$ and $S \rightarrow \mathscr{M}(\tilde{X}/B,\gamma)$ respectively, with the induced morphism $S' \rightarrow S$ strict. We denote the resulting families of punctured curves by $C^{'\circ}$ and $C^\circ$. By Lemma \ref{plog}, we may push forward the punctured log structures on $C^\circ$ and $C^{'\circ}$ to produce punctured log curves $\overline{C}^\circ$ and $\overline{C}^{'\circ}$, which furthermore lift the above diagram in schemes to a diagram in log schemes. Since $(S',\mathcal{M}_{S',\gamma}) \rightarrow (S,\mathcal{M}_{S,\gamma})$ is strict, the bottom and outer square are cartesian in all categories. 

Note that the families of punctured log maps $\overline{C}^\circ \rightarrow X$ and $\overline{C}^{'\circ} \rightarrow X$ produced above will not necessarily be basic families, hence will not give $s((C/S,f))$ nor $s((C'/S',f'))$. However, $\overline{C}^\circ \rightarrow (S,\mathcal{M}_{S,\gamma})$ pulls back from a basic prestable punctured log curve $\overline{C}^{\circ,bas} \rightarrow (S,\mathcal{M}_{S,\tau})$. By pulling back the basic log structure $(S,\mathcal{M}_{\tau})$ along $S' \rightarrow S$, we produce a log scheme $(S',\mathcal{M}_{S',\tau})$ along with a strict morphism of log schemes $(S',\mathcal{M}_{S',\tau}) \rightarrow (S,\mathcal{M}_{S,\tau})$. Pulling back $\overline{C}^{\circ,bas}$ along $(S',\mathcal{M}_{S',\tau}) \rightarrow (S,\mathcal{M}_{S,\tau})$ gives $\overline{C}^{'\circ,bas} \rightarrow (S',\mathcal{M}_{S',\tau})$, together with a map $\overline{C}^{'\circ,bas} \rightarrow X$. By \cite{punc} Proposition 2.15, the resulting family of punctured log curves is prestable. Since basicness is preserved under pullback, the resulting family is also basic. The universal property of the pullback gives a morphism $\overline{C}^{'\circ} \rightarrow \overline{C}^{'\circ,bas}$, which must be the unique morphism from $\overline{C}^{'\circ}$ to a basic punctured log curve. Hence, we have an induced morphism $s((C'/S', f')) \rightarrow s((C/S,f))$ in the fibered category $\mathscr{M}(X,\beta)$. It is straightforward to check that this data defines a functor of fibered categories over schemes.

\end{proof}

In general, some care needs to be taken in understanding the tropical types of the stabilization. While by Theorem \ref{tlift}, we know that the tropical type of a stable punctured log map to $\tilde{X}$ is a lift of the tropical type of the stabilization, a given type $\gamma$ might be the lift of multiple types $\tau$. The problem is that if $v \in V(G_\gamma)$ is a vertex such that $\pmb\sigma(v) \in \Sigma(\tilde{X})$ is a cone mapping into a higher dimensional cone of $\Sigma(X)$ under $\Sigma(\tilde{X}) \rightarrow \Sigma(X)$, then $\gamma$ is a lift of at least two types, one which removes the resulting unconstrained vertex, and one which does not. We call vertices of $\tau$ coming from ambiguous $2$-valent vertices of $\gamma$ balanced $2$-valent vertices. Due to issues presented by balanced $2$-valent vertices, even if $\gamma$ is a lift of $\tau$, it is not the case in general that stabilization induces a morphism $\mathscr{M}(\tilde{X}/B,\gamma) \rightarrow \mathscr{M}(X/B,\tau)$. 

\begin{figure}[h]
\centering
\begin{tikzpicture}
\fill[white!70!blue, path fading = north]  (0,0)--(4,0)--(4,4)--(0,4)--cycle;
\draw[black] (0,0)--(4,0);
\draw[black] (0,0)--(0,4);
\draw[black] (0,0)--(4,4);
\draw[ball color = red] (3,0) circle (0.5mm);
\draw[ball color = red] (3/2,3/2) circle (0.5mm);
\draw[color = red] (3,0)--(3/2,3/2);
\draw[->, color = red] (3/2,3/2)--(0,3);

\fill[white!70!blue, path fading = north] (5,0)--(9,0)--(9,4)--(5,4)--cycle;
\draw[black] (5,0)--(9,0);
\draw[black] (5,0)--(5,4);
\draw[black,dashed] (5,0)--(9,4);
\draw[ball color = red] (8,0) circle (0.5mm);
\draw[ball color = red] (9,0) circle (0.5mm);
\draw[ball color = red] (13/2,3/2) circle (0.5mm);
\draw[color = red] (8,0)--(13/2,3/2);
\draw[color = red] (9,0)--(7,2);
\draw[->, color = red] (13/2,3/2)--(5,3);
\draw[->, color = red] (7,2)--(5,4);

\end{tikzpicture}
\caption{The lefthand figure depicts a tropical type with an ambiguous $2$-valent vertex, and the righthand figure depicts two possible tropical lifts associated with keeping or removing a balanced $2$-valent vertex.} 
\end{figure}

We can address this issue in two ways. First, if $\pmb\gamma$ is a decorated lift of a decorated tropical type, then this ambiguity is removed, as the push forward of the class of a component which is contracted by stabilization is zero, resolving this issue in the decorated case. Without assuming decoration, we note that the stabilization is at least marked by a tropical type $\omega$ which removes all balanced $2$-valent vertices from $\tau$. In this case, we can also guarantee the morphism has the intended target. In all cases, the resulting pairs are tropical lifts.

Assuming any of the above conditions ensures that stabilization induces a morphism $\mathscr{M}(\tilde{X}/B,\gamma) \rightarrow \mathscr{M}(X/B,\tau)$.

\section{Universal modification $\mathfrak{M}_{\gamma \rightarrow \tau} \rightarrow \mathfrak{M}_\tau$}
As in Proposition $1.6.2$ of \cite{bir_GW}, we wish to express $\mathscr{M}_{\gamma}$ as an appropriate pullback involving the morphism $\mathscr{M}_\tau \rightarrow \mathfrak{M}_\tau$. To do so, we need an auxiliary stack, which will be determined by an analogous tropical moduli problem.

Associated to $\mathfrak{M}_\tau$ is the Artin fan $\mathcal{A}_{\mathfrak{M}_\tau}$, which admits a strict map $\mathfrak{M}_\tau \rightarrow \mathcal{A}_{\mathfrak{M}_\tau}$. The Artin fan $\mathcal{A}_{\mathfrak{M}_\tau}$ has an \'etale cover by Artin cones $\mathcal{A}_{\tau'}$ as $\tau'$ ranges over strata of $\mathfrak{M}_\tau$. Each cone $\tau'$ is a tropical moduli space of tropical curves with a fixed tropical type $\tau'$, hence comes with a universal tropical family of maps $(\Gamma_{\tau'}/\tau',f_{\tau'})$ of punctured tropical maps of type $\tau'$.

By pulling back and pushing forward the modification $\Sigma(\tilde{X}) \rightarrow \Sigma(X)$ as done before in the case $\tau' = \tau$, followed by merging appropriate cones as was done following Lemma \ref{tcurve}, we produce subdivisions $\tilde{\tau}' \rightarrow \tau'$. These subdivisions respect face inclusions, i.e., pulling back the modification of $\tau'$ along a face $\sigma \subset \tau'$ gives the modification $\tilde{\sigma} \rightarrow \sigma$. Indeed, by the remarks following Lemma \ref{tcurve}, the subdivision depends only on the family $(\Gamma_{\tau'}/\tau',f_{\tau'})$ and the subdivision $\Sigma(\tilde{X}) \rightarrow \Sigma(X)$. Recall the subcomplex $\tilde{\gamma} \subset \tilde{\tau}$ of Lemma \ref{tlemma} which is a subdivision of $\gamma$, whose maximal cones correspond to maximally extendable tropical lifts $\omega$ of $\tau$ such that $\omega > \gamma$. Consider the subcomplex $\tilde{\tau}'_{\gamma \rightarrow \tau} \subset \tilde{\tau}'$ consisting of cones associated with tropical types which are marked by tropical types corresponding to cones of $\tilde{\gamma}$, together with the faces of these cones. Recall from Definition \ref{dtlift} that associated to any tropical lift $\gamma'$ of $\tau'$ is a choice of puncturing of the universal family $\Gamma^\circ_{\tau'}$. For $\omega' \in \tilde{\tau}'_{\gamma \rightarrow \tau}$ a cone associated with a maximally extendable tropical lift of $\tau'$, we note in the following lemma that there exists a unique tropical lift $\gamma'$ of $\tau'$ such that $\omega'$ is a maximal extension of $\gamma'$, and $\gamma'$ is marked by $\gamma$.

\
\begin{lemma}\label{upunc}
For any maximal cone $\omega' \in \tilde{\tau}'_{\gamma \rightarrow \tau}$, there is a unique tropical lift $\gamma'<\omega'$ of $\tau'$ which is marked by $\gamma$. 
\end{lemma}

\begin{proof}
 Note that the family of tropical curves $\Gamma^\circ_{\gamma}/\omega$ pulled back along $\omega \rightarrow \gamma$ is a subcomplex of $\Gamma^{\circ}_{\omega}/\omega$. Since $\omega$ marks $\omega'$, for every edge $e \in E(G_\omega)$ subdividing a leg $l \in L(G_\tau)$, there exists a unique edge $e' \in E(G_{\omega'})$ such that the cone $\omega'_{e'} \subset \Gamma^\circ_{\omega'}/\omega'$ contains $\omega_{e} \in \Gamma^\circ_{\gamma}/\omega$ as a face. In particular, after identifying every leg $l \in L(G_\gamma)$ with an edge or leg $e_l \in E(G_\omega) \cup L(G_\omega)$, by removing cones from $\Gamma^{\circ}_{\omega'}/\omega'$ associated to edges and legs $e_i \in E(G_{\omega'})\cup L(G_{\omega'})$ such that $e_i > e_{l}$ for some $l \in L(G_{\gamma})$, we produce a family of tropical curves marked by $\gamma$, and whose associated tropical type $\gamma'$ satisfies $\gamma' < \omega'$, as required. Uniqueness follows by the uniqueness of $e_l \in E(G_{\omega'})$ for each $l \in L(G_{\gamma})$. 

\end{proof}

Using this lemma, we generalize Lemma \ref{tlemma}, and coarsen $\tilde{\tau}'_{\gamma \rightarrow \tau}$ to a cone complex $\tau'_{\gamma\rightarrow \tau}$ so that cones of the resulting complex are tropical moduli spaces of tropical curves of type $\gamma'$, which are marked by $\gamma$ and lift $\tau'$.
 
\begin{lemma}\label{tropuni}
For any tropical type $\gamma'$ which is marked by $\gamma$ and lifts $\tau'$, the collection of cones $\omega'$ of $\tilde{\tau}'_{\gamma \rightarrow \tau}$ with $\omega' > \gamma'$ gives a subdivision $\tilde{\gamma}'$ of $\gamma'$. Moreover, the collections of such cones $\gamma'$ form a cone complex $\tau'_{\gamma \rightarrow \tau}$ which coarsens $\tilde{\tau}'_{\gamma \rightarrow \tau}$. In particular, the moduli problem of punctured tropical curves marked by $\gamma$ and lifting $\tau'$ is representable by the cone complex $\tau'_{\gamma \rightarrow \tau}$.
\end{lemma}
 
 \begin{proof}
 Given a tropical type $\gamma'$ above, as in Lemma \ref{tlemma}, there exists a subdivision $\tilde{\gamma}' \rightarrow \gamma'$ such that each cone parameterizes a family of maximally extended tropical curves for some tropical lift $\omega'$ of $\tau'$ such that $\omega' > \gamma'$, with the subdivision morphism induced by forgetting appropriate legs and edges of $L(G_{\omega'})\cup E(G_{\omega'})$. In particular, $\tilde{\gamma}'$ is a subcomplex of $\tilde{\tau}'$. Since $\gamma'$ is marked by $\gamma$, each maximally extendable tropical type associated with a cone of $\tilde{\gamma}'$ is marked by a maximally extendable tropical type associated with $\tilde{\gamma}$. In particular, $\tilde{\gamma}'$ is a subcomplex of $\tilde{\tau}'_{\gamma \rightarrow \tau}$. Since by Lemma \ref{upunc} every type associated with $\omega' \in \tilde{\tau}'_{\gamma\rightarrow \tau}$ is a maximal extension of one such $\gamma'$, we have $\tau'_{\gamma \rightarrow \tau}$ is realized as a coarsening of $\tilde{\tau}'_{\gamma \rightarrow \tau}$. 
 \end{proof}
 
Finally, we pass to a lattice refinement of each cone of $\tau'_{\gamma \rightarrow \tau}$ so that the edge lengths of the tropical families are integer valued. This refinement respects face inclusions by modularity. Note that $\tau'_{\gamma \rightarrow \tau}$ may be identified with a subcomplex of a subdivision of $\tau'$. This subdivision is a coarsening of the subdivision $\tilde{\tau}' \rightarrow \tau'$ compatible with restriction to faces, and from now on, we replace $\tilde{\tau}'$ with this coarsening.

\begin{example}
Consider the moduli space of punctured log maps to $X = \mathbb{P}^1 \times \mathbb{P}^1$ equipped with its toric log structure $\sum_{i} D_{i}$, with contact orders $-D^*_{1}, D^*_{1},D^*_{2},D^*_{3}$ at $4$ marked points respectively, where $D_i^*$ is a primitive integral tangent vector pointing along the $1$-dimensional cone of $\Sigma(X)$ corresponding to a toric divisor $D_i$. This has one minimal realizable tropical type, with all other tropical types with the legs marked by this unique tropical type. Let $\tau$ be the minimal tropical type, and $\tau'$ be the degenerate tropical type, which respectively are depicted below as graphs with filled and dashed edges and legs, along with their associated tropical moduli spaces. Note that the tropical modulus in either case is determined by the image of a vertex.

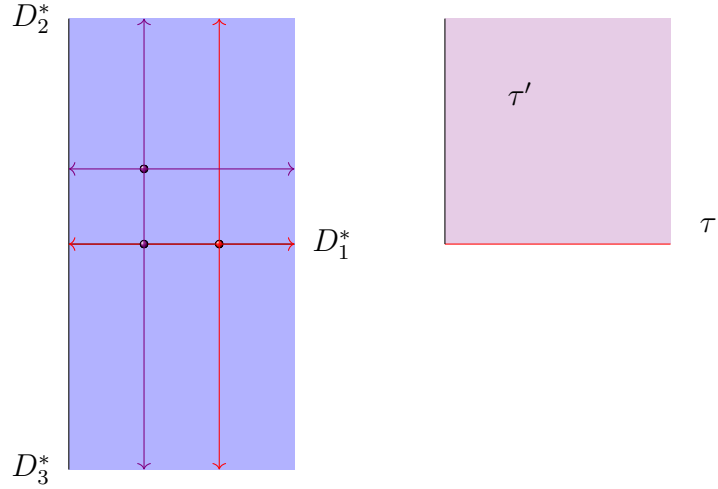
\begin{figure}[h]
\centering
\begin{tikzpicture}
\fill[white!70!blue, path fading = north] (0,0)--(3,0)--(3,3)--(0,3)--cycle;
\fill[white!70!blue, path fading = south] (0,0)--(3,0)--(3,-2)--(0,-2)--cycle;
\draw[black] (0,0)--(3,0);
\node at (3.5,0)  {$D_{1}^*$};
\draw[black] (0,0)--(0,3);
\node at (-.5,3) {$D_{2}^*$};
\draw[black] (0,0)--(0,-2);
\node at (-.5,-2) {$D_{3}^*$};
\draw[ball color = red] (2,0) circle (0.5mm);
\draw[->,color = red] (2,0)--(0,0);
\draw[->,color = red] (2,0)--(3,0);
\draw[->,color = red] (2,0)--(2,3);
\draw[->,color = red] (2,0)--(2,-2);
\draw[ball color = violet] (1,1) circle (0.5mm);
\draw[ball color = violet] (1,0) circle (0.5mm);
\draw[dashed, color = violet] (1,1)--(1,0);
\draw[->,dashed, color = violet] (1,1)--(1,3);
\draw[->,dashed, color = violet] (1,1)--(0,1);
\draw[->,dashed, color = violet] (1,1)--(3,1);
\draw[->,dashed, color = violet] (1,0)--(1,-2);

\fill[white!80!violet, path fading = south] (5,0)--(8,0)--(8,3)--(5,3)--cycle;
\node at (6,2) {$\tau'$};
\draw[red] (5,0)--(8,0);
\draw[black] (5,0)--(5,3);
\node at (8.5,0) {$\tau$};

\end{tikzpicture}
\end{figure}
If we consider the toric blow-up at the zero stratum given by $D_{1}\cap D_{2}$, then there exists four possible tropical lifts $\gamma$ of $\tau$ depending on the the choice of length of the legs pointing west and north. By varying the length of the leg pointing north, we produce two tropical lifts, one maximally extended tropical lift, and one type with shorter vertical leg and no two valent vertices. In either case, we see the subdivision $\tilde{\tau}'$ of $\tau'$ depicted below, along with the subcomplex $\tilde{\tau}'_{\gamma \rightarrow \tau}$ given by the unique cone of $\tilde{\tau}'$ containing the face $\gamma$.
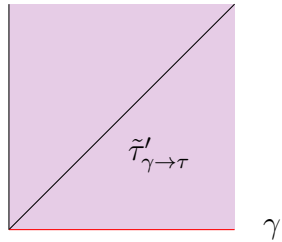
\begin{figure}[h]
\centering
\begin{tikzpicture}
\fill[white!60!violet, path fading = north] (0,0)--(3,0)--(3,3)--(0,3)--cycle;
\draw (0,0)--(3,3);
\draw[red] (0,0)--(3,0);
\draw[black] (0,0)--(0,3);
\node  at (2,1) {$\tilde{\tau}'_{\gamma \rightarrow \tau}$};
\node at (3.5,0) {$\gamma$};
\end{tikzpicture}
\end{figure}
\end{example}

To transport the tropical construction above to the log category, we use the equivalence of $2$-categories between Artin fans and cone stacks established in \cite{stacktrop} Theorem $6.11$. First, the compatible family of subdivisions $\tilde{\tau}' \rightarrow \tau'$ for cones associated with tropical types marked by $\tau$ induces a log \'etale morphism $\widetilde{\mathcal{A}_{\mathfrak{M}_\tau}} \rightarrow \mathcal{A}_{\mathfrak{M}_\tau}$. We further consider the closed substack given by the closure of the inclusion of the closed point $\gamma^\circ \in \mathcal{A}_{\gamma}$. Note that by pulling back this construction along an inclusion $\mathcal{A}_{\tau'} \rightarrow \mathcal{A}_{\mathfrak{M}_\tau}$, the Artin fan of the resulting log algebraic stack has associated cone complex $\tau'_{\gamma \rightarrow \tau}$ before the lattice refinement ensuring integral edge lengths. After pulling back the root stack construction associated with the lattice refinement, we produce a log algebraic stack which we denote by $\mathcal{A}_{\gamma \rightarrow \tau}$. A map $S \rightarrow \mathcal{A}_{\gamma \rightarrow \tau}$ is the data of a fine and saturated log structure on $S$, together with a family of punctured tropical maps $\Gamma \rightarrow \Sigma(X)$ over $\Sigma(S)$ which are marked by $\gamma$, and are a lift of a tropical type $\tau'$ which is marked by $\tau$.  Note by construction that the closed point of $\mathcal{A}_{\gamma}$ in $\mathcal{A}_{\gamma \rightarrow \tau}$ is dense. We equip $\mathcal{A}_{\gamma \rightarrow \tau}$ with the idealized log structure with ideal $\overline{\mathcal{J}} \subset \overline{\mathcal{M}}_{\mathcal{A}_{\gamma \rightarrow \tau}}$ which when evaluated on an \'etale map $U \rightarrow \mathcal{A}_{\gamma \rightarrow \tau}$ is given by $q^{-1}(\gamma^{\vee}_\NN\setminus 0)$, for $q: \Gamma(U,\overline{\mathcal{M}}_{\mathcal{A}_{\gamma\rightarrow \tau}}) \rightarrow \overline{\mathcal{M}}_{\mathcal{A}_{\gamma \rightarrow \tau,\gamma^\circ}}$ the restriction map on sections of the ghost sheaf. Note that this ideal contains the pullback of the ideal $\overline{\mathcal{I}}$ from $\mathcal{A}_{\mathfrak{M}_{\tau}}$. The following log stack constructed as a fiber product will be central to our main result:

\begin{equation}
 \mathfrak{M}_{\gamma \rightarrow \tau}:= \mathfrak{M}_\tau \times_{\mathcal{A}_{\mathfrak{M}_\tau}} \mathcal{A}_{\gamma \rightarrow \tau}.
 \end{equation}
 By the strictness of $\mathfrak{M}_\tau \rightarrow \mathcal{A}_{\mathfrak{M}_\tau}$, this stack is a fiber product in all log categories. As a stack on the category of log schemes, its functor of points associates to a log scheme $S$ the groupoid of log stable maps $(C/S,f)$ over $S$ with target $\mathcal{X}$ such that for all cones $\sigma \in \Sigma(S)$, the tropicalization of the classifying map $\sigma \rightarrow \Sigma(\mathfrak{M}_\tau)$ factors through $\tau'_{\gamma \rightarrow \tau} \rightarrow \Sigma(\mathfrak{M}_\tau)$ for some cone $\tau' \in \Sigma(\mathfrak{M}_\tau)$.

Now consider the following partially filled commutative diagram:

\begin{equation}\label{mdi2}
\begin{tikzcd}
\mathscr{M}_\gamma \arrow[r,dashrightarrow] \arrow{d}\arrow[rr,bend left = 20] & \mathfrak{M}_{\gamma \rightarrow \tau} \arrow[r,dashrightarrow] \arrow{d} & \mathfrak{M}_\gamma \\
\mathscr{M}_{\tau} \arrow{r} & \mathfrak{M}_\tau 
\end{tikzcd}
\end{equation}

We wish to construct the two dashed morphisms in the above diagram and show their composite is the natural map $\mathscr{M}_\gamma \rightarrow \mathfrak{M}_\gamma$. We end this section by constructing one of these morphisms.

\begin{proposition}\label{prop1}
There exists a strict and idealized strict morphism $\mathscr{M}_\gamma \rightarrow \mathfrak{M}_{\gamma \rightarrow \tau}$ making the square in diagram (\ref{mdi2}) commute.
\end{proposition}

\begin{proof}
As a result of the Theorem \ref{tlift}, the geometric fibers of the universal family $(\mathfrak{C}^\circ/\mathscr{M}_\gamma,f)$ must have tropical types which are lifts of tropical types associated with strata of $\mathscr{M}_\tau$. By Lemma \ref{tropuni} and the definition of $\mathcal{A}_{\gamma \rightarrow \tau}$, we have a strict morphism $\mathscr{M}_\gamma \rightarrow \mathcal{A}_{\gamma \rightarrow \tau}$, and by \cite{punc} Proposition $3.23$, together with $\mathscr{M}_\gamma \rightarrow \mathfrak{M}_\gamma$ being idealized strict, we see that this morphism is idealized strict. 

On the other hand, we may also conclude from Theorem \ref{tlift} that the morphism $\mathscr{M}_{\gamma} \rightarrow \mathcal{A}_{\mathfrak{M}_\tau}$ induced by taking the basic log structure associated with the stabilization of the universal family over $\mathscr{M}_{\gamma}$ factors through the inclusion of log stacks $\mathcal{A}_{\gamma \rightarrow \tau} \rightarrow \mathcal{A}_{\mathfrak{M}_\tau}$ via a strict morphism $\mathscr{M}_\gamma \rightarrow \mathcal{A}_{\gamma \rightarrow \tau}$. Hence, by the universal property characterizing $\mathfrak{M}_{\gamma \rightarrow \tau}$, we have a strict and idealized strict morphism $\mathscr{M}_\gamma \rightarrow \mathfrak{M}_{\gamma \rightarrow \tau}$. 
\end{proof}

\subsection{Construction of the morphism $\mathfrak{M}_{\gamma \rightarrow \tau}\rightarrow \mathfrak{M}_\gamma$}

We now turn to the question of constructing the morphism $\mathfrak{M}_{\gamma \rightarrow \tau} \rightarrow \mathfrak{M}_\gamma$. To do so, we will construct a basic family of punctured log curves $\tilde{C}^\circ \rightarrow \tilde{\mathcal{X}}$ from the data of a non-basic family of punctured log maps curve $\overline{C}^\circ \rightarrow \mathcal{X}$ associated to a morphism $S \rightarrow \mathfrak{M}_{\gamma \rightarrow \tau}$. We proceed by using the tropical data to both determine a log \'etale modification of the prestable log curve $\tilde{C} \rightarrow \overline{C}$, as well as a puncturing of the resulting log scheme which admits a map $\tilde{C} \rightarrow \tilde{X}$. 

Consider a strict morphism $S \rightarrow \mathfrak{M}_{\gamma \rightarrow \tau}$ from a log scheme $S$. We initially assume the log structure on $S$ is induced by a morphism to $\Ag$, for $\gamma'$ a tropical lift of the tropical type $\tau'$ of $S \rightarrow \mathfrak{M}_{\tau}$. By replacing $\gamma'$ by a larger stratum if necessary, we may assume the image of $S$ intersects the smallest dimensional stratum of $\Ag$. Let $G_{\gamma'}$ and $G_{\tau'}$ be the dual graphs associated with the types $\gamma'$ and $\tau'$ respectively. The morphism $S \rightarrow \mathfrak{M}_{\tau}$ given by composing with the projection map determines a non-basic family of punctured log maps $(\overline{C}^{\circ}/S,f)$. We denote by $\overline{C}/S$ the family of ordinary log curves associated with the morphism of log stacks $S \rightarrow \mathfrak{M}_\tau$, which tropicalizes to $\Gamma/\gamma'$.

Now consider $\Gamma^\circ$, the tropicalization of $\overline{C}^\circ$. By functorial tropicalization, we have a family of tropical maps $\Gamma^\circ \rightarrow \Sigma(X)$ over $\gamma'$ of type $\tau'$. By construction of $\gamma'$, there exists a universal family of punctured tropical maps $(\tilde{\Gamma}^\circ,f)$ over $\gamma'$ given by puncturing and subdivision of the tropical family $\Gamma^\circ$. In particular, as in the proof of Proposition \ref{tcurve}, we have a subdivision $\tilde{\Gamma}\rightarrow \Gamma$ of the unpunctured tropical curves which induces the subdivision of punctured tropical curves. We summarize the setup in the following commutative diagram, with the squares on the top row cartesian:

\[\begin{tikzcd}
\widetilde{\Gamma} \arrow{d} & \widetilde{\Gamma}^\circ \arrow{l} \arrow{r} \arrow{d} & \widetilde{\Sigma(X)} \arrow{d} \\
\Gamma \arrow{dr} & \Gamma^\circ \arrow{l} \arrow{d} \arrow{r} & \Sigma(X) \arrow{d}\\
& \gamma' \arrow{r} & \Sigma(B)
\end{tikzcd}\]

Under the equivalence of $2$-categories between Artin fans and cone stacks, $S$ and $\overline{C}$ possess strict morphisms to the Artin fans associated to the cone stacks of $\gamma'$ and $\Gamma$ respectively. The subdivision $\widetilde{\Gamma} \rightarrow \Gamma$ of the family of tropical curves over $\gamma'$ then give a subdivision of Artin fans $\widetilde{\mathcal{A}}_{\overline{C}} \rightarrow \mathcal{A}_{\overline{C}}$. By pulling back this log \'etale morphism along the strict map $\overline{C} \rightarrow \mathcal{A}_{\overline{C}}$, we produce a log \'etale modification $\tilde{C} \rightarrow \overline{C}$. Since we have ensured all edge lengths are integer valued, integral points of cones of $\tilde{\Gamma}$ surject onto the integral points of $\gamma'$. By the proof of \cite{wss} Theorem $2.1.4$, $\tilde{C} \rightarrow S$ is a saturated morphism of log schemes, and in particular integral. Hence, $\tilde{C} \rightarrow S$ is flat with fiber dimension equal to the log fiber dimension which is $1$. To see that the family has reduced geometric fibers, note that $\tilde{C} \rightarrow S$ is a log smooth and saturated morphism of log schemes, hence the claim follows by \cite{satmor} Theorem $II.3.4$. Since the morphism $\tilde{C} \rightarrow S$ is log smooth, we have $\tilde{C}/S$ is a family of log curves. 

To produce a puncturing of the log curve $\tilde{C}/S$, we note that the universal family of tropical maps over $\gamma'$, $\tilde{\Gamma}^\circ \rightarrow \Sigma(\tilde{X})$ determines a prestable monoid $Q_l$ associated with each leg of $G_{\gamma'}$, hence each punctured or marked section of $\tilde{C}$. Using this data, we define a submonoid $\overline{\mathcal{M}}_{\tilde{C}^\circ} \subset \overline{\mathcal{M}}_{\tilde{C}}^{gp}$ to be $\overline{\mathcal{M}}_{\tilde{C}}$ away from neighborhoods intersecting punctured sections of $\tilde{C} \rightarrow S$, and in an atomic neighborhoods $U \subset \tilde{C}$ with unique closed stratum an open subset of the puncturing section $p_l \subset \tilde{C}$ associated with a leg $l \in L(G_{\gamma'})$, we take sections which are restrictions of sections $Q_l \subset \Gamma(U,\overline{\mathcal{M}}_{\tilde{C}}^{gp})$. We define $\mathcal{M}_{\tilde{C}^\circ} := \mathcal{M}_{\tilde{C}}^{gp} \times_{\overline{\mathcal{M}}_{\tilde{C}}^{gp}} \overline{\mathcal{M}}_{\tilde{C}^\circ}$. To define $\alpha_{\tilde{C}^\circ}$, we extend $\alpha_{\tilde{C}}: \mathcal{M}_{\tilde{C}} \rightarrow \mathcal{O}_{\tilde{C}}$ by setting $\alpha(s) = 0$ for $s \notin \mathcal{M}_{\tilde{C}}$. To see this defines a new puncturing of $\tilde{C}$ at the punctured sections, note that $s \notin \mathcal{M}_{\tilde{C}}$ if and only if $s$ induces a linear function with negative slope along the leg $l \in L(G_{\gamma'})$. Writing $\overline{s} \in \gamma^{'\vee,gp} \oplus \mathbb{Z}$ as $\overline{s} = (\overline{s}_1,\overline{s}_2)$, we thus have $\overline{s}_2 < 0$. By the choice of idealized log structure on $\mathfrak{M}_{\gamma \rightarrow \tau}$, we must have $\overline{s}_1 \in J\subset \gamma^{'\vee}$. Using this fact, we show $\alpha_{\tilde{C}^\circ}: \mathcal{M}_{\tilde{C}^\circ} \rightarrow \mathcal{O}_{\tilde{C}}$ defines a log structure. Suppose we have $s \in \mathcal{M}_{\tilde{C}^\circ} \setminus \mathcal{M}_{\tilde{C}}$ and $t \in \mathcal{M}_{\tilde{C}^\circ}$. Then $0 = \alpha_{\tilde{C}^\circ}(s)\alpha_{\tilde{C}^\circ}(t)$. If $st \in \mathcal{M}_{\tilde{C}^\circ}\setminus \mathcal{M}_{\tilde{C}}$, then $\alpha_{\tilde{C}^\circ}(st) = 0$. If $st \in \mathcal{M}_{\tilde{C}}$, then $\overline{s}\overline{t} = (x,y) \in \gamma^{'\vee}\oplus \mathbb{N}$, with $x \in J \subset \gamma^{'\vee}$. Hence, we must have $\alpha_{\tilde{C}^\circ}(st) = \alpha_{\tilde{C}}(st) = 0$. Hence, $\tilde{C}^\circ$ is a puncturing of the log curve $\tilde{C}$.

Having defined a punctured log curve $\tilde{C}^\circ$, observe by Lemma \ref{plog} that we may pushforward the punctured log structure on $\tilde{C}^\circ$ to a puncturing of $\overline{C}$ which in general will contain the prestable puncturing on $\overline{C}^\circ$. Hence, we have a morphism of log schemes $\tilde{C}^\circ \rightarrow \overline{C}^\circ$. We need to lift the morphism $\overline{C}^\circ \rightarrow \mathcal{A}_X$ to a morphism $\tilde{C}^\circ \rightarrow \widetilde{\mathcal{A}}_X$. We will do so by using Lemma \ref{mtoA}. Since the log curve $\tilde{C}^\circ \rightarrow S$ tropicalizes to the universal family of tropical curves of type $\gamma'$, we have a morphism $\Sigma(\tilde{C}^\circ) \rightarrow \Sigma(\tilde{X})$. By definition of the prestable monoid at punctured points, the factorization condition in the second statement of Lemma \ref{mtoA} holds. By pulling back the \'etale cover of $\mathcal{A}_X$ by Artin cones to $\tilde{C}^\circ$ and restricting to components of this cover, we note that there locally exist morphisms $\tilde{C}^\circ \rightarrow \widetilde{\mathcal{A}_X}$ which tropicalize to a restriction of the morphism of cone complexes above. Thus, by Lemma \ref{mtoA}, we produce the desired morphism globally.

Since $\tilde{\mathcal{X}} \cong \mathcal{X} \times_{\mathcal{A}_X} \widetilde{\mathcal{A}}_X$, the data of the morphism $\tilde{C}^\circ \rightarrow \widetilde{\mathcal{A}}_X$, together with the data of $\tilde{C}^\circ \rightarrow \mathcal{X}$ given by precomposing $\overline{C}^\circ \rightarrow \mathcal{X}$ with the partial stabilization map determines a morphism $\tilde{C}^\circ \rightarrow \tilde{\mathcal{X}}$. By considering the tropicalization of the family, $\tilde{C}^\circ \rightarrow \widetilde{\mathcal{X}}$ is also basic. Hence, the non-basic family of maps $\tilde{C}^\circ \rightarrow \mathcal{X}$ over $S$ determines a morphism $S \rightarrow \mathfrak{M}_\gamma$.

To remove the assumption of a global chart $S \rightarrow \Ag$ for the log structure on $S$, we note that there is an \'etale cover $\{S_{\gamma'}\}$ of $S$ such that the restricted morphisms $S_{\gamma'} \rightarrow \mathcal{A}_{\gamma \rightarrow \tau}$ factor through $\Ag$. Moreover the two induced log structures $S_{\gamma',\gamma''} := S_{\gamma'} \times_S S_{\gamma''}$ determined by the maps to $\mathcal{A}_{\gamma',\gamma}$ and $\mathcal{A}_{\gamma'',\gamma}$ are the same for both morphisms, as both are given by pulling back the log structure on $\mathcal{A}_{\gamma \rightarrow \tau}$. We note that the morphism $S_{\gamma',\gamma''} \rightarrow \mathfrak{M}_\gamma$ constructed above depends only on this log structure and the morphism $S_{\gamma',\gamma''} \rightarrow \mathfrak{M}_{\tau}$. Thus, the two induced morphisms to $\mathfrak{M}_{\gamma}$ are equal. Hence the locally defined morphisms to $\mathfrak{M}_{\gamma}$ glue to give a morphism $S \rightarrow \mathfrak{M}_{\gamma}$. 

The construction above defines a potential functor of fibered categories $\mathfrak{M}_{\gamma \rightarrow \tau} \rightarrow \mathfrak{M}_\gamma$ on objects. We now show this actually extends to a functor of fibered categories over schemes, giving us the required morphism $\mathfrak{M}_{\gamma \rightarrow \tau} \rightarrow \mathfrak{M}_{\gamma}$. 

\begin{proposition}
The exists a morphism $\mathfrak{M}_{\gamma \rightarrow \tau} \rightarrow \mathfrak{M}_{\gamma}$ such that for a morphism $S \rightarrow \mathfrak{M}_{\gamma \rightarrow \tau}$ with $S$ an fs log scheme, the induced morphism $S \rightarrow \mathfrak{M}_{\gamma}$ is given by the construction above. 
\end{proposition}

\begin{proof}
As fibered categories over log schemes with strict morphisms, we may identify $\mathfrak{M}_{\gamma \rightarrow \tau}$ as a subcategory of $\mathfrak{M}_{\tau}$. Indeed, $\mathfrak{M}_{\gamma \rightarrow \tau}$ is constructed by pulling back the subcategory $\mathcal{A}_{\gamma \rightarrow \tau} \rightarrow \mathcal{A}_{\mathfrak{M}_\tau}$. Moreover, we may identify the fibered category over schemes associated with $\mathfrak{M}_{\gamma\rightarrow \tau}$ with a full subcategory of the associated fibered category over log schemes. Hence, we consider morphisms between objects of $\mathfrak{M}_{\gamma \rightarrow \tau}$ to be morphisms between the corresponding objects of $\mathfrak{M}_{\tau}$. We must show that given a morphism $(\overline{C}/S,f) \rightarrow (\overline{C}'/S',f')$ in $\mathfrak{M}_{\gamma \rightarrow \tau}$, associated with non-basic maps to $X$, there is functorial assignment of morphisms $(\tilde{C}^\circ/S,\tilde{f}) \rightarrow (\tilde{C}^{'\circ}/S',\tilde{f}')$. Since the morphism $S \rightarrow S'$ is strict, $\overline{C}^\circ \rightarrow \overline{C}^{' \circ}$ is strict, and the cone complex morphism $\Sigma(\overline{C}^\circ) \rightarrow \Sigma(\overline{C}^{' \circ})$ maps cones isomorphically onto cones. Moreover, $\Sigma(\overline{C}^{\circ}) \rightarrow \Sigma(X)$ factors through $\Sigma(\overline{C}^{' \circ}) \rightarrow \Sigma(X)$. In particular, the modification $\widetilde{\Sigma(\overline{C}^\circ)} \rightarrow \Sigma(\overline{C}^\circ)$ pulls back from the modification $\widetilde{\Sigma(C^{'\circ})} \rightarrow \Sigma(C^{'\circ})$. Standard arguments with universal properties show $\tilde{C} \rightarrow \overline{C}$ pulls back from the modification $\tilde{C}' \rightarrow \overline{C}'$. Since the pullback of a prestable puncturing is prestable by Proposition of $2.15$ of \cite{punc}, we have the family $\tilde{C}^{\circ}/S$ is the pullback along the strict morphism $S \rightarrow S'$ of the family $\tilde{C}^{'\circ}/S'$. This gives the remaining data to define a morphism $\mathfrak{M}_{\gamma \rightarrow \tau} \rightarrow \mathfrak{M}_{\gamma}$ satisfying the desired property of the proposition.
\end{proof}

The final item we must show is that the composite $\mathscr{M}_\gamma \rightarrow \mathfrak{M}_{\gamma \rightarrow \tau} \rightarrow \mathfrak{M}_\gamma$ is isomorphic to the natural morphism $\mathscr{M}_\gamma \rightarrow \mathfrak{M}_\gamma$.

\begin{proposition}\label{unvmod}
The composite of the morphisms $\mathfrak{M}_{\gamma \rightarrow \tau} \rightarrow \mathfrak{M}_{\gamma}$ and $\mathscr{M}_\gamma \rightarrow \mathfrak{M}_{\gamma \rightarrow \tau}$  is isomorphic to $\mathscr{M}_{\gamma} \rightarrow \mathfrak{M}_{\gamma}$. 
\end{proposition}

\begin{proof}
Suppose we have a strict morphism $S \rightarrow \mathscr{M}_{\gamma}$. By working \'etale locally, we may assume the idealized log structure on $S$ is induced by a morphism $S \rightarrow \mathcal{A}_{\gamma',\gamma}$. The above data gives a punctured log map $(C^\circ/S,f: C^\circ \rightarrow \tilde{X})$ over a log scheme $S$ admitting a strict map $S \rightarrow \mathcal{A}_{\gamma',\gamma}$. Letting $(\overline{C}^{\circ}/S,\overline{f})$ be the basic prestable punctured log map given by the induced map $S \rightarrow \mathfrak{M}_{\tau}$, then the push forward of the log structure on $C^\circ$ to $\overline{C}$ gives a punctured log map $\overline{C}^\circ \rightarrow \mathcal{X}$, together with a prestabilization morphism $\overline{C}^\circ \rightarrow \overline{C}^{\circ}$ over $S$. To construct the morphism of unpunctured families $C\rightarrow \tilde{C} = \overline{C} \times_{\mathcal{A}_{\overline{C}}} \mathcal{A}_{\tilde{C}}$, we observe that we may identify $\Sigma(C)$ and $\Sigma(\tilde{C})$ by construction, so we have a morphism $C \rightarrow \mathcal{A}_{\tilde{C}}$, hence a morphism $C \rightarrow \overline{C}\times_{\mathcal{A}_{\overline{C}}} \mathcal{A}_{\tilde{C}}$ since the relevant square commutes. This morphism can be seen to commute with each of the morphisms to base log scheme $S$, hence $C \rightarrow \tilde{C}$ is a morphism of log curves over $S$. By Lemma $\href{https://stacks.math.columbia.edu/tag/05XD}{05XD}$ of \cite{SP}, there exists an open subset $S' \subset S$ of the base such that the morphism of families of prestable curves restricts to an isomorphism. Hence, to show $C \rightarrow \tilde{C}$ is an isomorphism, it suffices to show this is true for all geometric points of $S$, so we assume $S = Spec\text{ }\kk$. With this reduction, observe $C \rightarrow \tilde{C}$ is an isomorphism away from contracted components, and the morphism of tropical curves $\Sigma(C) \rightarrow \Sigma(\tilde{C})$ forces the morphism of contracted rational components to be a totally ramified cover of degree $1$, hence an isomorphism. These morphisms extend to isomorphisms of the log maps $(C^\circ,f) \rightarrow (\tilde{C}^\circ,\tilde{f})$ by the agreement of the associated tropical maps. That these isomorphisms define a natural isomorphism between the two functors between the fibered categories in question follow by the universal properties used to define the component isomorphisms.

 \end{proof}

\section{Proof of Main Theorems}

We now have the commutative diagram of stacks described in statement of Theorem \ref{mthm1}:

\begin{equation}\label{mdi}
\begin{tikzcd}
\mathscr{M}_\gamma \arrow[r] \arrow{d} & \mathfrak{M}_{\gamma \rightarrow \tau} \arrow[r] \arrow{d} & \mathfrak{M}_\gamma \\
\mathscr{M}_{\tau} \arrow{r} & \mathfrak{M}_\tau 
\end{tikzcd}
\end{equation}

Moreover, the composite upper horizontal arrow and the lower horizontal arrow are equipped with perfect obstruction theories which facilitate the definition of the virtual fundamental classes of the moduli spaces featured in the left column. Following the lead of \cite{bir_GW}, in order to prove Theorems \ref{mthm1} and \ref{mthm2}, we wish to show the square in the diagram is a pullback square, and $\mathfrak{M}_{\gamma \rightarrow \tau} \rightarrow \mathfrak{M}_\gamma$ is \'etale.
\begin{proposition}
The square in diagram (\ref{mdi}) is cartesian in all categories
\end{proposition}
\begin{proof}

Suppose that we have the data of a fine and saturated log scheme $S$ and log morphisms $S \rightarrow \mathscr{M}_\tau$ and $S \rightarrow \mathfrak{M}_{\gamma\rightarrow \tau}$, along with an isomorphism between the induced maps $S \rightarrow \mathfrak{M}_\tau$. By composing $S \rightarrow \mathfrak{M}_{\gamma \rightarrow \tau}$ with $\mathfrak{M}_{\gamma \rightarrow \tau} \rightarrow \mathfrak{M}_{\gamma}$, this data gives the following 2-commutative diagram of log stacks:

\[\begin{tikzcd}
C^\circ \arrow{r} \arrow{d} & \mathcal{\tilde{X}} \arrow{rd} \\
\overline{C}^\circ \arrow{r} \arrow{d}  & X \arrow{r}\arrow{d} & \mathcal{X} \arrow{d} \\
S \arrow{r} & B \arrow{r} & \mathcal{A}_B
\end{tikzcd}
\]

By the universal property satisfied by $\tilde{X}$, this data defines a unique map $C^\circ \rightarrow \tilde{X}$. Stability follows from the fact that the tropical type of any fiber of $C^\circ \rightarrow S$ must be a tropical lift of a tropical type $\tau'$ associated with a stratum of $\mathfrak{M}_{\tau}$, hence by the first part of Theorem \ref{tlift} defines a stable map. Thus, we find a unique morphism of log stacks $S \rightarrow \mathscr{M}_\gamma$ lifting the two morphisms $S \rightarrow \mathscr{M}_\tau$ and $S \rightarrow \mathfrak{M}_{\gamma \rightarrow \tau}$. Therefore, $\mathscr{M}_\gamma$ satisfies the desired universal property in fs log schemes. By the strictness of $\mathscr{M}_\tau \rightarrow \mathfrak{M}_\tau$, $\mathscr{M}_\gamma$ satisfies the desired universal property in all categories. 

\end{proof}

\begin{proposition}
The morphism $\mathfrak{M}_{\gamma \rightarrow \tau} \rightarrow \mathfrak{M}_\gamma$ is \'etale. 
\end{proposition}

\begin{proof}
Since $\mathfrak{M}_{\gamma \rightarrow \tau} \rightarrow \mathfrak{M}_{\gamma}$ is strict and idealized strict, it suffices to show $\mathfrak{M}_{\gamma \rightarrow \tau} \rightarrow \mathfrak{M}_{\gamma}$ is idealized log \'etale. Suppose we have an idealized strict square zero extension $S \rightarrow S'$ of idealized log schemes schemes, together with the following commutative square:

\[\begin{tikzcd}
S \arrow{r} \arrow{d} & \mathfrak{M}_{\gamma \rightarrow \tau} \arrow{d} \\
S' \arrow{r}\arrow[ur,dashrightarrow] & \mathfrak{M}_\gamma
\end{tikzcd}\]

We need to fill in the dashed arrow. Note that both horizontal morphisms give punctured log maps $(C^\circ/S,f)$, $(C^{\circ'}/S',f')$ with target $\tilde{\mathcal{X}}$, and we have a pair of morphisms $C/S \rightarrow C'/S'$, $C^\circ/S \rightarrow C^{'\circ}/S'$ which respect the puncturing morphisms and morphisms to $\tilde{\mathcal{X}}$. The additional data associated to $S \rightarrow \mathfrak{M}_{\gamma \rightarrow \tau}$ is a punctured log map $(\overline{C}^\circ/S,\overline{f})$ with target $\mathcal{X}$ marked by the tropical type $\tau$, together with a partial stabilization map $C^\circ \rightarrow \overline{C}^\circ$. We construct $\overline{C}'$ by taking the pushout of $C \rightarrow C'$ and $C \rightarrow \overline{C}$ in the category of log schemes. We thus produce the following commuting diagram in the category of log schemes.

\[\begin{tikzcd}
C \arrow{r} \arrow[d] & C' \arrow[d,"c"]\\
\overline{C} \arrow{r} \arrow{d} & \overline{C}' \arrow{d} \\
S \arrow{r} & S'
\end{tikzcd}\]

Using Lemma \ref{plog}, the pushforward $c_*\mathcal{M}_{C^{'\circ}}$ defines a puncturing of the log stable curve $(\overline{C}'/S',c_*\mathcal{M}_{C'})$. Moreover, the proofs of of Lemmas $5.1.1$ and $5.1.2$ of \cite{bir_GW} apply in our context to show $\overline{C}^{'\circ}$ is the pushout of $C^\circ \rightarrow C^{'\circ}$ and $C^\circ \rightarrow \overline{C}^\circ$ in the $2$-category of logarithmic stacks. Hence, the universal property gives the desired morphism $\overline{C}^{'\circ} \rightarrow \mathcal{X}$.
\end{proof}

By the above proposition, the obstruction theory for $\mathscr{M}_\gamma \rightarrow \mathfrak{M}_\gamma$ also acts as an obstruction theory for the morphism $\mathscr{M}_\gamma \rightarrow \mathfrak{M}_{\gamma \rightarrow \tau}$. In order to prove Theorems \ref{mthm1} and \ref{mthm2}, we must show $\mathfrak{M}_{\gamma \rightarrow \tau} \rightarrow \mathfrak{M}_\tau$ is proper, idealized log \'etale, DM type, and understand the generic strata in $\mathfrak{M}_\tau$ and $\mathfrak{M}_{\gamma \rightarrow \tau}$.

\begin{proposition}
$s: \mathfrak{M}_{\gamma\rightarrow \tau} \rightarrow \mathfrak{M}_{\tau}$ is proper.
\end{proposition}

\begin{proof}
After recalling the construction of $\mathfrak{M}_{\gamma \rightarrow \tau}$ at the start of Section $4$, $\mathfrak{M}_{\gamma \rightarrow \tau}$ is a closed substack of the domain of a log \'etale morphism induced by a pullback of a subdivision of an Artin fan. Since a subdivision by \cite{bir_GW} Theorem $2.4.1$ is proper, and proper morphisms are preserved under pullback and restriction to closed substacks, the conclusion follows.
\end{proof}

\begin{proposition}
$ \mathfrak{M}_{\gamma \rightarrow \tau} \rightarrow \mathfrak{M}_{\tau}$ is idealized log \'etale.
\end{proposition}

\begin{proof}
Since $\mathfrak{M}_{\gamma \rightarrow \tau} \rightarrow \mathfrak{M}_\tau$ is defined by pulling back the morphism of idealized Artin fans $\mathcal{A}_{\tau \rightarrow \gamma} \rightarrow \mathcal{A}_{\mathfrak{M}_\tau}$, it suffices to show this latter morphism is idealized log \'etale. Since this property is local on the source and target, this property is equivalent to $\Ag \rightarrow \At$ being idealized log \'etale. This follows by \cite{punc} Lemma $B.3$. 
\end{proof}

\begin{corollary}
$\mathfrak{M}_{\gamma\rightarrow \tau} \rightarrow \mathfrak{M}_{\tau}$ is DM type.
\end{corollary}

\begin{proof}
 By definition, $\mathfrak{M}_{\gamma \rightarrow \tau} \rightarrow \mathfrak{M}_\tau$ is pulled back from an idealized log \'etale map locally on the domain of the form $\Ag \rightarrow \At$. By Lemma \ref{DM}, this morphism is DM type. Hence, $\mathfrak{M}_{\gamma \rightarrow \tau} \rightarrow \mathfrak{M}_{\tau}$ is DM type.
\end{proof}

\begin{proposition}\label{vdeg}
Assuming $dim\text{ }\gamma = dim\text{ }\tau$ with $m = |coker(\gamma_\NN^{gp} \rightarrow \tau_\NN^{gp})|$, then $\mathfrak{M}_{\gamma \rightarrow \tau} \rightarrow \mathfrak{M}_\tau$ is finite of degree $\frac{1}{m}$. 
\end{proposition}

\begin{proof}
Note that $\mathfrak{M}_{\gamma \rightarrow \tau}$ has a unique open stratum $\mathfrak{M}_{\gamma\rightarrow \tau}^\circ$ which maps to the open stratum $\mathfrak{M}_{\tau}^\circ \subset \mathfrak{M}_\tau$ of curves with tropical type $\tau$. We thus compute the degree over a generic geometric point of $\mathfrak{M}_{\gamma\rightarrow \tau}^\circ \rightarrow \mathfrak{M}_\tau^\circ$.

Observe that we have a morphism $\mathfrak{M}_\tau^\circ \rightarrow \mathcal{A}_{\tau,\tau}$, and the generic stratum of $\mathfrak{M}_{\gamma \rightarrow \tau}$ is the pullback along this morphism of the map $\mathcal{A}_{\gamma,\gamma} \rightarrow \mathcal{A}_{\tau,\tau}$. From this description of the map $\mathfrak{M}_{\gamma \rightarrow \tau}^\circ \rightarrow \mathfrak{M}_\tau^\circ$, and since $\mathcal{A}_{\tau,\tau}$ has only one geometric point, we have the degree of interest is constant over all geometric points of $\mathfrak{M}_\tau$, and the degree of the map of idealized Artin cones $\mathcal{A}_{\gamma,\gamma} \rightarrow \mathcal{A}_{\tau,\tau}$. Lemma \ref{DM} then determines this degree to be $\frac{1}{m}$.

\end{proof}

The morphism $\mathscr{M}_\tau \rightarrow \mathfrak{M}_\tau$ has an obstruction theory defined in \cite{punc}. For the morphism $\mathscr{M}_\gamma \rightarrow \mathfrak{M}_{\gamma \rightarrow \tau}$, we will use the obstruction theory defined for $\mathscr{M}_\gamma \rightarrow \mathfrak{M}_\gamma$, using the fact that $\mathfrak{M}_{\gamma \rightarrow \tau} \rightarrow \mathfrak{M}_\gamma$ is \'etale.

To compare the obstruction theories for the morphisms $\mathscr{M}_\tau \rightarrow \mathfrak{M}_{\tau}$ and $\mathscr{M}_{\gamma} \rightarrow \mathfrak{M}_{\gamma \rightarrow \tau}$, we show the obstruction theory for $\mathscr{M}_\tau \rightarrow \mathfrak{M}_{\tau}$ pulls back to the obstruction theory for $\mathscr{M}_{\gamma} \rightarrow \mathfrak{M}_{\gamma\rightarrow \tau}$. Our argument for this fact follows verbatim from that given in \cite{bir_GW} section $6$, since the obstruction sheaves are the pull and derived push of the logarithmic tangent sheaf of the target, just as for ordinary log stable maps. For completeness, we reproduce the argument here. Let $S \subset S'$ be a square zero extension with ideal $J$. If $S \rightarrow \mathscr{M}_\tau$ is a family of log stable curves, the lifts of this morphism to $S' \rightarrow \mathscr{M}_\tau$ form a torsor under $f^*T_X^{log} \otimes J$. We define $\mathcal{E}(J)$ to be the stack on $S$ of $f^*T_X^{log}\otimes J$ torsors on $C$, which gives an obstruction theory for $\mathscr{M}_{\tau} \rightarrow \mathfrak{M}_{\tau}$ in the sense of \cite{wiseob}

 Similarly, given a morphism $S \rightarrow \mathscr{M}_\gamma$, lifts $S' \rightarrow \mathscr{M}_\gamma$ form a torsor under $f^*T_{\tilde{X}}^{log} \otimes J$, and we define the sheaf $\mathcal{E}'(J)$ on $S$ as above. Now let $\overline{f}: \overline{C} \rightarrow X$ the family of punctured log maps of type $\tau$ over $S$ and $f: C \rightarrow \tilde{X}$ the family of punctured log maps of type $\gamma$ over $S$. Since the stabilization map $c: C \rightarrow \overline{C}$ is a contraction of rational components, and $\tilde{X} \rightarrow X$ is log \'etale, observe:

\[Rc_*(f^*T_{\tilde{X}}^{log}\otimes J) = Rc_*(c^*\overline{f}^*T_X^{log} \otimes J) = \overline{f}^*T_X^{log} \otimes J \]

Hence, $f^*T_{\tilde{X}}^{log} \otimes J$ torsors can be identified with $\overline{f}^*T_X^{log} \otimes J$ torsors by pullback. Moreover, under this identification, the torsor of lifts of $\overline{f}$ is identified with the torsor of lifts of $\overline{f}\circ c: C \rightarrow X$, which in turn is identified with the torsor of lifts of $f$. It follows from the previous two facts that the obstruction theories for $\mathscr{M}_{\tau} \rightarrow \mathfrak{M}_{\tau}$ and $\mathscr{M}_{\gamma} \rightarrow \mathfrak{M}_{\gamma\rightarrow \tau}$ are related by pullback, and Theorem \ref{mthm1} now follows.

In the special situation of Theorem \ref{mthm2}, the morphism $\mathfrak{M}_{\gamma \rightarrow \tau} \rightarrow \mathfrak{M}_\tau$ has degree $\frac{1}{m}$ by Proposition \ref{vdeg}. Since $\mathfrak{M}_{\gamma \rightarrow \tau} \rightarrow \mathfrak{M}_\tau$ is proper, by \cite{Cos} and \cite{pusherr}, $\pi_*[\mathscr{M}_\gamma]^{vir} = \frac{1}{m}[\mathscr{M}_\tau]^{vir}$. Alternatively, since we have verified $\mathfrak{M}_{\gamma \rightarrow \tau} \rightarrow \mathfrak{M}_\tau$ is proper and DM type, we can use the more flexible set up given by the virtual pullback formalism of \cite{vpull}, together with results of \cite{pushforward} Proposition $B.18$ to conclude this result. This also allows us to compare the virtual Chow theory of $\mathscr{M}(\tilde{X},\gamma)$ and $\mathscr{M}(X,\tau)$ more generally. Indeed, we have shown the morphism $\mathfrak{M}_{\gamma \rightarrow \tau} \rightarrow \mathfrak{M}_\tau$ is a proper morphism of DM type, hence by \cite{pushforward} Proposition $B.18$, admits a push forward operation of rational Chow groups. Then starting from a class in $A_*(\mathfrak{M}(\tilde{\mathcal{X}},\gamma))$, results from \cite{vpull} give the equality of the virtual classes in $A_*(\mathscr{M}(X,\tau))$ produced from composing virtual pullback and pushforward in both orders. 

\section{Decorated version of main theorem}
Let us now consider a decorated tropical type $\pmb\tau = (\tau,\textbf{A}(v))$, and $\overline{C}\rightarrow X$ a punctured log map having decorated tropical type $\pmb\tau$. As was shown in Theorem \ref{tlift}, if $C \rightarrow \tilde{X}$ is a stable log curve which stabilizes to $\overline{C} \rightarrow X$, then the tropical type of $C$ is a tropical lift of $\tau$. One may further ask to what extent is the decorated tropical type of $C \rightarrow \tilde{X}$ determined by the decorated tropical type of $\overline{ C}\rightarrow X$. We approach this question via the following lemma:

\begin{lemma}\label{dlift}
Suppose we have a geometric point $Spec\text{ }k \rightarrow \mathscr{M}(X/B,\tau)$, which necessarily factors through $\mathscr{M}(X/B,\pmb\tau)$ for some decoration $\pmb \tau = (\tau,\textbf{A}(v))$ of $\tau$. Consider the following cartesian diagram:

\[\begin{tikzcd}
\mathscr{M}_{\gamma,Spec\text{ }\kk} \arrow{r} \arrow{d} & \mathscr{M}(\tilde{X}/B,\gamma) \arrow{d} \arrow{r} & \mathfrak{M}_{\gamma \rightarrow \tau} \arrow{d} \\ 
Spec\text{ }\kk \arrow{r} & \mathscr{M}(X/B,\tau) \arrow{r} & \mathfrak{M}_\tau
\end{tikzcd}\]

Then $\mathscr{M}_{\gamma,Spec\text{ }\kk}$ is connected. In particular, the top horizontal arrows factor through $\mathscr{M}(\tilde{X}/B,\pmb\gamma)$ for a decoration $\pmb\gamma = (\gamma,\textbf{B}(v))$ of $\gamma$ depending only on the decorated tropical type of $Spec\text{ }\kk \rightarrow \mathscr{M}(X/B,\tau)$.

\end{lemma}

\begin{proof}

Let $\pmb\tau'$ be the decorated tropical type of the geometric point $Spec\text{ }\kk \rightarrow \mathscr{M}(X/B,\tau)$. Since $\mathfrak{M}_{\gamma \rightarrow \tau} \rightarrow \mathfrak{M}_\tau$ is is pulled back from $\mathcal{A}_{\gamma \rightarrow \tau} \rightarrow \mathcal{A}_{\mathfrak{M}_\tau}$, and $Spec\text{ }\kk \rightarrow \mathcal{A}_{\mathfrak{M}_\tau}$ factors through the \'etale morphism $\mathcal{A}_{\tau'} \rightarrow \mathcal{A}_{\mathfrak{M}_\tau}$ via the unique closed point of $\mathcal{A}_{\tau'}$, we find the Deligne Mumford stack $\mathscr{M}_{\gamma,Spec\text{ }\kk}$ is given by union of strata living over the origin of a toric stack modification $S_{\widetilde{\tau'}} \rightarrow S_{\tau'}$ which lie in the closure of the stratum associated with $\gamma$. We will wish to show that zig-zags of strata inclusions connect every strata of $\mathscr{M}_{\gamma,Spec\text{ }\kk}$. 

Assuming this for the moment, consider a stratum of $\mathscr{M}_{\gamma,Spec\text{ }\kk}$, with associated type $\gamma'$. Since the stratum is connected, every geometric point of this stratum has a fixed decorated tropical type $\pmb\gamma'$. Since the graph $G_\gamma$ is a contraction of the graph $G_{\gamma'}$, the decoration of $\gamma'$ induces a unique decoration of $\pmb\gamma$ of $\gamma$ such that $\pmb\gamma'$ is marked by $\pmb\gamma$. Moreover, for any stratum $\gamma''$ which contains $\gamma'$ in its closure, the graph contraction $G_{\gamma'} \rightarrow G_{\gamma}$ factors through $G_{\gamma''} \rightarrow G_{\gamma}$. Hence, given a decoration $\pmb \gamma''$ of $ \gamma''$, the decoration of $\gamma$ induced by $\pmb\gamma'$ is also the decoration induced by $\pmb \gamma''$. Since all strata living over a point of $\mathscr{M}(X/B,\tau)$ are connected by a zig-zag of strata inclusions, the decorated types associated with all strata of $\mathscr{M}_{\gamma,Spec\text{ }\kk}$ are marked by the same decorated tropical type $\pmb\gamma$. In particular, the inclusion $\mathscr{M}_{\gamma,Spec\text{ }\kk} \subset \mathscr{M}(\tilde{X}/B,\gamma)$ factors through the substack $\mathscr{M}(\tilde{X}/B,\pmb \gamma) \subset \mathscr{M}(\tilde{X}/B,\gamma)$. 

Let us now return to our unproven assumption, and consider the cone $\tau'$ and the subdivision $\tilde{\tau}'$. We identify both cone complexes as fans sitting inside the same ambient real vector space $\mathbb{R}^m$, and $\gamma$ is a subcone of a face $\tau$ of $\tau'$. Consider the star fan of $\gamma$:
\[Star(\gamma) = \{\sigma + \gamma^{gp} \subset \tau^{'gp}/\gamma^{gp} \text{ }\mid \text{ }\gamma \subset \sigma \in \tilde{\tau}'\}.\]
This fan is a subdivision of the polyhedron:
\[Star(\tau) \times (\tau^{gp}/\gamma^{gp}) \subset \tau^{'gp}/\gamma^{gp} \cong \tau^{'gp}/\tau^{gp} \times \tau^{gp}/\gamma^{gp}.\]
 Note that to every strata in $\mathscr{M}_{\gamma,Spec\text{ }\kk}$ is associated with a cone of $Star(\gamma)$, and given an inclusion of strata $S_\sigma \subset S_{\sigma'}$, we have $\sigma' \subset \sigma$. The conclusion now follow by the convexity of the polyhedron $Star(\tau) \times \tau^{gp}/\gamma^{gp}$.

\end{proof}

From this, we deduce the decorated version of Theorem \ref{mthm1}

\begin{theorem}\label{decthm}
For a decorated lift $\pmb\gamma = (\gamma,\textbf{B}(v))$ of the decorated tropical type $\pmb \tau$, denote by $\mathscr{M}_{\pmb\tau}$ the image of $\mathscr{M}(\tilde{X}/B,\pmb\gamma) \rightarrow \mathscr{M}(X/B,\pmb\tau)$. Then the following diagram is cartesian:

\[\begin{tikzcd}
\mathscr{M}(\tilde{X}/B,\pmb\gamma) \arrow{r} \arrow{d} & \mathfrak{M}_{\gamma \rightarrow \tau} \arrow{d} \\
\mathscr{M}_{\pmb\tau} \arrow{r} & \mathfrak{M}_{\tau}
\end{tikzcd}\]
Moreover, $\mathscr{M}_{\pmb\tau}$ is a union of connected components
\end{theorem}

\begin{proof}
Consider the following cartesian diagram:

\[\begin{tikzcd}
\widetilde{\mathscr{M}}_\tau \arrow{r}\arrow{d} & \mathfrak{M}_{\gamma \rightarrow \tau} \arrow{d}\\
\mathscr{M}_{\pmb\tau} \arrow{r} & \mathfrak{M}_\tau
\end{tikzcd}\]

The fiber product is a closed substack of $\mathscr{M}(\tilde{X}/B,\gamma)$, and we have $\mathscr{M}(\tilde{X}/B,\pmb\gamma) \subset \widetilde{\mathscr{M}}_\tau$. To show equality, we must show all curves associated to points of $\widetilde{\mathscr{M}}_\tau$ admit markings by the decorated type $\pmb\gamma$. To see this, note that all curves marked by $\gamma$ admit a decorated marking by a unique decorated lift of $\tau$. Since $\mathscr{M}_\tau$ is defined to be the image of $\mathscr{M}(\tilde{X}/B,\pmb\gamma)$ in $\mathscr{M}(X/B,\tau)$, the fiber over every geometric point of $\mathscr{M}_\tau$ contains a punctured log curve marked by the decorated type $\pmb\gamma$. By Lemma \ref{dlift}, all curves in $\widetilde{\mathscr{M}}_\tau$ must be marked by $\pmb\gamma$. Hence, we also have $\widetilde{\mathscr{M}}_\tau \subset \mathscr{M}(\tilde{X}/B,\pmb\gamma)$, and $\mathscr{M}(\tilde{X}/B,\pmb\gamma)$ satisfies the universal property of interest. 

For the latter statement, note that since $\mathscr{M}(\tilde{X},\gamma) \rightarrow \mathscr{M}(X,\tau)$ is surjective, for any connected component $\mathscr{M} \subset \mathscr{M}(X,\pmb\tau)$, there exists a collection of moduli stacks $\mathscr{M}(\tilde{X},\pmb\gamma_i)$ whose images when intersected with $\mathscr{M}$ exhaust $\mathscr{M}$. Since $s: \mathscr{M}(\tilde{X},\pmb\gamma) \rightarrow \mathscr{M}(X,\pmb\tau)$ is proper, $s$ is in particular closed. By the above cartesian diagram, the images have trivial intersection. Since we assumed $\mathscr{M}$ was connected, we must have only one such decorated tropical type. 
\end{proof}

\section{Boundedness of punctured log maps}

After birational invariance of ordinary log Gromov-Witten invariants was established in \cite{bir_GW}, the main result of \cite{bound} demonstrated that the moduli stacks of log maps defined in \cite{LogGW} are proper Deligne-Mumford stacks for a wider range of log smooth targets than was originally proven when the stacks were first constructed. This result made crucial use of the main result of \cite{bir_GW}. In \cite{bound}, the authors first observed that by results of \cite{modmor}, it was sufficient to prove $\mathscr{M}(X/B,\beta)$ was finite type, separated, and satisfied the weak valuative criterion. In \cite{LogGW}, these properties were proven to hold assuming the sheaf $\overline{\mathcal{M}}_X^{gp}$ was globally generated. The key observation of \cite{bound} used to weaken this assumption was that this property held for $\tilde{X}$ a log \'etale modification of $X$, and using log \'etale invariance to prove the remaining properties.

After the moduli stacks of punctured log curves $\mathscr{M}(X/B,\pmb\tau)$ were constructed in \cite{punc}, a similar issue arose in proving the stacks $\mathscr{M}(X/B,\pmb\tau)$ were proper Deligne-Mumford stacks. To prove this property, the sheaf $\overline{\mathcal{M}}_X\otimes_{\mathbb{Z}} \mathbb{Q}$ was assumed to be globally generated. Using the main results of this paper and the proof strategy of \cite{bound}, we can now show the stacks $\mathscr{M}(X/B,\tau)$ are finite type without needing the global generation assumption.

We fix a decorated tropical class $\beta$ on $X$ with total curve class $\textbf{A}$, and note that by \cite{punc} Theorem $3.17$, $\mathscr{M}(X/B,\beta)$ satisfies the valuative criterion for properness without the global generation assumption. Since by \cite{decomp} Proposition $2.34$, every realizable decorated tropical type $\pmb\tau$ is realized as a closed substack of $\mathscr{M}(X/B,\beta)$ for some decorated tropical class $\beta$ via base change along a finite morphism, $\mathscr{M}(X/B,\pmb\tau)$ also satisfies the valuative criterion for properness. It remains to show that $\mathscr{M}(X/B,\pmb\tau)$ is finite type.

\begin{proposition}\label{bound}
Supposing $X \rightarrow B$ is projective, then $\mathscr{M}(X/B,\pmb\tau) \rightarrow B$ is finite type.
\end{proposition}

\begin{proof}
The proof is essentially \cite{bound} Proposition $5.3.1$ and $5.3.2$. We make necessary changes to account for the negative contact setting. Let $\pi: \tilde{X} \rightarrow X$ be a log \'etale modification induced by a projective subdivision such that $\mathcal{M}_{\tilde{X}}$ is globally generated. We pick a relatively ample line bundle $\mathcal{L}$ on $\tilde{X}$. Then $\mathcal{L} = \bigotimes_i \mathcal{L}_i^{m_i}$, for $m_i$ negative integers and $\mathcal{L}_i$ the normal bundle of the exceptional divisors of $\tilde{X}\rightarrow X$. By the logarithmic balancing condition, for any $(C,f) \in \mathscr{M}(\tilde{X}/B,\gamma)(Spec\text{ }\kk)$ with $\gamma$ a lift of $\tau$, $deg(f^*\mathcal{L}) = \sum_{i,j} m_ic_j(E_i)$, where $c_j(E_i)$ is contact order of the punctured point $p_j$ with the exceptional divisor $E_i$ as encoded in the tropical type $\tau$. Note that we might have $c_j(E_i)< 0$. 

We let $\gamma$ be a tropical lift of the underlying tropical type $\tau$. For any fixed decoration $\pmb\gamma$ of $\gamma$ by curve classes in $NE(\tilde{X})$, by \cite{punc} Theorem $3.12$, $\mathscr{M}(\tilde{X}/B,\pmb\gamma) \rightarrow B$ is finite type. From the decorated version of Theorem \ref{mthm1}, we have a proper surjective morphism $\coprod_{\pmb\gamma} \mathscr{M}(\tilde{X}/B,\pmb\gamma) \rightarrow \mathscr{M}(X/B,\pmb\tau)$ with the union begin taken over all decorated lifts of $\pmb\tau$ with underlying type $\gamma$. If the total degree of all decorated lifts with respect to a choice of relative ample divisor for $\tilde{X} \rightarrow B$ is bounded, the proposition follows.

To demonstrate this statement, we pick a relative ample line bundle $M$ for $X \rightarrow B$, and consider the line bundle $\mathcal{L} \otimes \pi^*(M)$ on $\tilde{X}$. This gives a relative ample bundle for $\tilde{X} \rightarrow B$, and on a curve $(C,f)\in \mathscr{M}(\tilde{X}/B,\gamma)$, we have:

\[deg(f^*(\mathcal{L}\otimes \pi^*(M))) = deg(f^*(\mathcal{L})) + deg(f^*\pi^*M) = \sum_{i,j} m_ic_j(E_i) + deg(f^*\pi^*M).\]
Thus, by the projection formula, if $\pmb\gamma$ is a decorated lift of $\pmb\tau$, the total degree of $\pmb\gamma$ have $\mathcal{L} \otimes \pi^*M$ degree bounded above by 
\[\textbf{A}\cdot M + \sum_{i,j} m_ic_j(E_i) .\]
Since $M$ is relatively ample, we in particular find that the total degree of any decorated lifts is bounded above, and the proposition follows.

\end{proof}

For a general decorated tropical type $\pmb\tau$, without assuming realizability, we find that $\mathscr{M}(X/B,\pmb\tau)$ is finite type if and only if there are finitely many non-empty realizable tropical types $\tau'$ marked by $\tau$ that are minimal realizable tropical types dominating $\tau$ in the sense of \cite{punc} Remark $3.29$. Indeed, by loc cit., the closure of a maximal strata in $\mathscr{M}(X/B,\tau)$ are the images of the stacks $\mathscr{M}(X/B,\tau')$. Moreover, fixing a decoration of $\pmb\tau$ bounds the degree of a component of a curve in $\mathscr{M}(X/B,\pmb\tau)$ with respect to an ample divisor on $X$. Since each realizable tropical type $\tau'$ above has only finitely many vertices, there are only finitely many possible decorations of $\tau'$. 

\section{Birational invariance of canonical wall structures}

We now consider the question of birational invariance of canonical wall structures. Let $(X,D)$ be a log Calabi Yau in the sense of \cite{int_mirror} satisfying the hypotheses in Assumption $1.1$ or $1.2$ of \cite{scatt} and let $D = \sum_i D_i$ be a decomposition in terms of irreducible components. In particular, we must have either $c_1(T^{log}_X)$ is nef or antinef, or we may write $c_1(\Omega_X(log\text{ }D)) = K_X  + D =_\mathbb{Q} \sum_i a_iD_i$ with $a_i\ge 0$. In the latter case, the choice of representative of $K_X + D$ determines a subcomplex $\mathscr{P} \subset \Sigma(X)$ called the Kontsevich-Soibelman skeleton, with $1$ dimensional cones associated to divisors $D_i$ with $a_i = 0$. We call such divisors good, and all other components of $D$ bad divisors. In \cite{scatt}, Gross and Siebert construct a consistent wall structure $\mathscr{S}_{can}$ on $\mathscr{P} \subset \Sigma(X)$, via data attached to wall types with target $(X,D)$, see \cite{scatt} for details. We recall the definition of a wall type, which encode walls in $\Sigma(X)$.

\begin{definition}

A wall type $\tau = (G,\sigma,u)$ is a type of punctured log curve in a target $X$ satisfying:
\begin{enumerate}
\item $G$ is a genus $0$ graph with $L(G) = \{L_{out}\}$ with $\sigma(L_{out}) \in \mathscr{P}$, and $u_\tau := u(L_{out}) \not= 0$. 
\item $\tau$ is realizable and balanced
\item Let $h: \Gamma(G,l) \rightarrow \Sigma(X)$ be the corresponding universal family of tropical maps, and $\tau_{out}$ the cone corresponding to $L_{out}$. Then dim $\tau = n-2$, and dim $h(\tau_{out}) = n-1$. 
\end{enumerate}

\end{definition}

In particular, these types have a single punctured leg, which spans a codimension $1$ subcone contained in $B$, giving walls of the wall structure. Moreover, since $dim\text{ }\tau = n-2$, and $dim\text{ }h(\tau_{out}) = n-1$, a dimension count shows that among tropical curves of a fixed wall type, the tropical curve is uniquely determined by the image of the unique leg. In what follows, for a wall type $\tau$, the universal tropical family gives a morphism of lattices $h: \tau_{L_{out},\NN}^{gp} \rightarrow \pmb\sigma(L_{out})_\NN^{gp}$, and we define $k_{\tau} = |coker\text(h)_{tors}|$. 

Suppose that $(\tilde{X},\tilde{D}) \rightarrow (X,D)$ is a log \'etale modification, induced by a subdivision $\Sigma(\tilde{X}) \rightarrow \Sigma(X)$. In the following lemma, we show that the pair $(\tilde{X},\tilde{D})$ satisfies the necessary conditions for the construction of the canonical wall structure: 

\begin{lemma}
$(\tilde{X},\tilde{D})$ is log Calabi-Yau in the sense of \cite{int_mirror}, and the restriction of the subdivision $\Sigma(\tilde{X}) \rightarrow \Sigma(X)$ when restricted to the Kontsevich-Soibelman skeleton $\widetilde{\mathscr{P}} \subset \Sigma(\tilde{X})$ gives a morphism of cone complexes $\widetilde{\mathscr{P}} \rightarrow \mathscr{P}$. In particular, $(\tilde{X},\tilde{D})$ satisfies Assumption $1.1$ of \cite{scatt}.
\end{lemma}

\begin{proof}
The lemma is clear by the projection formula in the case $c_1(T_X^{log})$ is nef or antinef, so we assume $(X,D)$ is log Calabi-Yau in the latter sense. First, observe that since $\tilde{X} \rightarrow X$ is log \'etale, $\Omega_{\tilde{X}}(log\text{ }\tilde{D}) = \pi^*\Omega_X(log\text{ }D)$. Thus, $\Lambda^{dim\text{ }\tilde{X}}\Omega_{\tilde{X}}(log\text{ }\tilde{D}) = \pi^*(\Lambda^{dim\text{ }X}\Omega_X(Log\text{ }D))$. The Kontsevich-Soibelman skeleton on $X$ depends on choice of $\mathbb{Q}$-Cartier divisor $\sum_i a_iD_i$ which is numerically equivalent to $c_1(T_X(log\text{ }D))$. Suppose that the exceptional locus of the log \'etale map $\tilde{X} \rightarrow X$ has center the stratum $D_I := \cap_{i \in I} D_i$. For all $j \notin I$, then $\pi^*(D_j)$ is equal to an irreducible component of $\tilde{D}$. If $i \in I$, then $\pi^*(D_i) = \tilde{D}_i + \sum_{j}m_{i,j}E_j$, with $E_j$ an enumeration of exceptional divisors, $\tilde{D}_i$ is the strict transform of $D_i$, and $m_{i,j}\ge 0$. Now observe:

\[c_1(\Omega_{\tilde{X}}(log\text{ }\tilde{D})) \equiv_{\mathbb{Q}} \pi^*(\sum_i a_iD_i) = \sum_{k \notin I} a_k\tilde{D}_k + \sum_{i \in I} a_i(\tilde{D}_i + \sum_{j}m_{i,j}E_j) = \sum_i a_i\tilde{D}_i + \sum_{i,j} m_{i,j}a_iE_j\]

Since the $\mathbb{Q}$-divisor on the right-most side of the above equalities is supported on $\tilde{D}$, effective and is numerically equivalent to $c_1(\Omega_{\tilde{X}}(log\text{ }\tilde{D}))$, $(\tilde{X},\tilde{D})$ is log Calabi-Yau in the sense of \cite{int_mirror}. In particular, the choice of $\mathbb{Q}$-divisor above determines a Kontsevich-Soibleman skeleton $\widetilde{\mathscr{P}} \subset \Sigma(\tilde{X})$. 

If $D_I$ was only contained in good divisor, then $\sum_{i,j} m_{i,j}a_i = 0$, and we have a decomposition of $c_1(\Omega_{\tilde{X}}(log\text{ }\tilde{D}))$ with associated Kontsevich-Soibelman skeleton given by the strict transform of the good components of $D$, together with the exceptional divisor $E$. If $D_I$ were contained in at least one bad divisor, then since $a_i \ge 0 $ for all $i$, the exceptional divisors are bad divisors, and the Kontsevich-Soibelman skeleton would only be the strict transforms of the good divisors in $X$. In either case, the restriction of the subdivision $\Sigma(\tilde{X}) \rightarrow \Sigma(X)$ to the Kontsevich-Soibelman skeleton of $\tilde{X}$ induced by $\pi^*(\sum_i a_iD_i)$ has image the Kontsevich-Soibelman skeleton of $X$. 

Since all log \'etale modifications are centered on log strata of $X$, the conclusion follows. The fact that $\widetilde{\mathscr{P}}\subset \Sigma(\tilde{X})$ satisfies Assumption $1.1$ or $1.2$ of \cite{scatt} follows by the fact that $\mathscr{P} \subset \Sigma(X)$ does since the morphism of cone complexes is simply a subdivision.

\end{proof}

Having specified the Kontsevich-Soibelman skeleton $\widetilde{\mathscr{P}} \subset \Sigma(\tilde{X})$, we let $\widetilde{\mathscr{S}}_{can}$ be the canonical wall structure on $\widetilde{\mathscr{P}}$. Suppose we have a wall type $\tau$ with target $X$, with $u_\tau$ the contact order of the unique leg. In particular, $dim\text{ }\tau = n-2$. Note that the tropical modulus of a fixed wall type is completely determined by the image of the unique vertex $v_{out}$ contained in the leg $L_{out} \in L(G_\tau)$. Hence, the real cone complex $\tilde{\tau}$ is equal to the induced subdivision of the codimension $2$ cone $h(\tau_{v_{out}}) \subset \pmb\sigma(v_{out}) \in \Sigma(X)$. We consider the tropical lifts $\gamma$ which correspond to top tropical dimension cones of this subdivision, i.e subdivisions with dim $\gamma = n-2$. Note that all of the top tropical dimension lifts have a unique leg as well, are realizable, balanced and dim $h(\gamma_{out}) = n-1$. Since the subdivision restricts to a morphism of Kontsevich-Soibelman skeletons, we conclude that all top tropical dimension lifts of $\tau$ define wall types of $(\tilde{X},\tilde{D})$. Due to the fact that the unique leg of $\tau$ might be subdivided by a lift, we get a wall type in $(\tilde{X},\tilde{D})$ for every choice of contact order in $\Sigma(\tilde{X})$ which projects to $u_\tau$. In total however, after projecting these wall to $\Sigma(X)$, by varying over both top dimensional tropical lifts of $\tau$ as well as possible leg lengths, the collections of associated walls has support equal to the the wall associated to $\tau$.

We wish to show all contributing wall types to $\widetilde{\mathscr{S}}_{can}$ arise in this manner.
\begin{lemma} \label{awal}
Assume $\mathscr{M}(\tilde{X},\pmb\gamma)$ is non-empty, for a decorated wall type $\pmb\gamma$. If $N_{\pmb\gamma} := deg[\mathscr{M}(\tilde{X},\pmb\gamma)]^{vir} \not= 0$, then $\pmb\gamma$ is a decorated lift of a decorated wall type $\pmb\tau$ on $(X,D)$.
\end{lemma}

\begin{proof}
Suppose we have a decorated wall type $\pmb\gamma$ on $\tilde{X}$ such that $\mathscr{M}(\tilde{X},\pmb\gamma)$ is non-empty. This tropical type has a single leg $l$, corresponding to a punctured point with contact order $u_\gamma$. This contact order can be identified with a unique contact order in $\Sigma(X)$ after forgetting the subdivision. Composing with the log \'etale modification followed by stabilizing gives us a morphism of moduli spaces $\mathscr{M}(\tilde{X},\pmb\gamma) \rightarrow \mathscr{M}(X,u_\gamma)$, where the codomain is the moduli space of the potentially unrealizable tropical type with a single vertex at the origin, and a leg with contact order $u_\gamma$. As remarked in the discussion following Proposition \ref{stab}, there exists a unique decorated and realizable tropical type $\pmb\tau$ of punctured curve to $X$, for which $\pmb\gamma$ is a lift. In particular, we have dim $\tau \ge \text{dim }\gamma = n-2$. By replacing the moduli stack $\mathscr{M}(X,\tau)$ with the closed substack $\mathscr{M}_{\pmb\tau}$ given by the image of $\mathscr{M}(\tilde{X},\pmb\gamma)$, by Theorem \ref{decthm}, we produce the following cartesian diagram:

\[\begin{tikzcd}
\mathscr{M}(\tilde{X},\pmb\gamma) \arrow{r}\arrow{d} & \mathfrak{M}_{\gamma \rightarrow \tau} \arrow{d}\\
\mathscr{M}_{\pmb\tau} \arrow{r} & \mathfrak{M}_\tau
\end{tikzcd}\]

Suppose $\pmb\tau$ is not a decorated wall type. As $\tau$ inherits properties $1$ and $2$ of a wall type from $\pmb\gamma$, we must have dim $\tau > n-2$. Hence, by \cite{punc} Theorem $3.28$, $\mathscr{M}(X,\pmb\tau)$ has negative virtual dimension. Note that this remains true after replacing $\mathscr{M}(X,\pmb\tau)$ with the closed substack $\mathscr{M}_{\pmb\tau}$. By \cite{punc} Proposition $3.28$, $\mathfrak{M}_\gamma$ is reduced, but not necessarily irreducible. Since $\mathfrak{M}_{\gamma \rightarrow \tau} \rightarrow \mathfrak{M}_\gamma$ is \'etale, $\mathfrak{M}_{\gamma \rightarrow \tau}$ is also reduced of pure dimension, but again not necessarily irreducible. We restrict the morphism $\mathfrak{M}_{\gamma \rightarrow \tau} \rightarrow \mathfrak{M}_{\tau}$ to irreducible components of $\mathfrak{M}_{\gamma \rightarrow \tau}$. Then the domain of this restricted morphism is integral, and all stacks are stratified by quotient stacks. Hence, we can apply \cite{int_mirror} Lemma $A.12$ to show that the virtual pullback of each irreducible component of $\mathfrak{M}_{\gamma \rightarrow \tau}$ contributes $0$ to the virtual fundamental class on $\mathscr{M}(\tilde{X},\pmb\gamma)$. As the virtual fundamental class is a sum over the virtual pullbacks of the irreducible components of $\mathfrak{M}_{\gamma\rightarrow \tau}$, this implies $N_{\pmb\gamma} = 0$, as required. 

\end{proof}

Now given an undecorated wall type $\gamma$ lifting a wall type $\tau$, let $\mathfrak{d}_{\gamma}$ be the wall spanned by the image of the unique leg of $\gamma$ in a cone of $\Sigma(\tilde{X})$, and $x \in int(\mathfrak{d}_{\gamma})$. By taking the image of $x$ under $\Sigma(\tilde{X}) \rightarrow \Sigma(X)$, we also identify a point $x \in int(\mathfrak{d}_{\tau})$. We let $Q_{\tilde{X}}$ be a finitely generated submonoid of $H_2(\tilde{X})$ containing all effective curves classes, $Q_{X} = \pi_*(Q_{\tilde{X}})$, and $I$ be a monomial ideal of $Q_{X}$ such that $\sqrt{I} = Q_{X}\setminus \{0\}$. Since $\pi_*: NE(\tilde{X}) \rightarrow NE(X)$ is surjective, $Q_{X}$ contains all effective curve classes. We also let $J =  \pi_*^{-1}(I) \subset Q_{\tilde{X}}$. In \cite{scatt}, associated to each wall type is a ring $\mathcal{P}_{x,X}$, which contains the monoid ring $\kk[h(\tau_{out})_{\NN}^{gp} \times Q_X]$, and the element:

\[\rho_{\pmb\tau} = exp(k_\tau N_{\pmb\tau} z^{-u_\tau}q^A).\]

In order for this power series expansion to give an element of $\mathcal{P}_{x,X}$, we use the fact that $\sqrt{I} = \mathfrak{m}$ to say that $Q_X\setminus I$ is finite, hence we truncate the sum at all terms divisible by $q^A$ for $A \in I$. We similarly have functions $\rho_{\pmb\gamma} \in \mathcal{P}_{x,\tilde{X}}$ associated with decorated wall types lifting $\pmb\tau$. Note that by the previous lemma, if $\rho_{\pmb\gamma} \not= 1$, and the total curve class of $\pmb\gamma$ is $A'$, then $\pi_*(A') \not= 0 \in NE(X)$. Thus, $kA' \in \pi^{-1}_*(I)$ for some $k > 0$. Thus, by considering the expansion only up to order $k$, we produce an element in $\mathcal{P}_{x,\tilde{X}}$. After identifying $h(\gamma_{out})_\NN^{gp}$ with a subgroup of $h(\tau_{out})_{\NN}^{gp}$, the pushforward homomorphism $Q_{\tilde{X}} \rightarrow Q_X$ induces a pushforward homomorphism $\mathcal{P}_{x,\tilde{X}} \rightarrow \mathcal{P}_{x,X}$ which respects the subrings containing the functions attached to the walls. Applying this homomorphism to all attached functions gives a wall structure on $\Sigma(\tilde{X})$ supporting the same walls as before, but with attached functions given by the pushforward of the previous attached functions. We may further replace this wall structure with an equivalent wall structure in which we replace the walls associated with decorated lifts $\pmb \gamma$ of $\pmb \tau$ with the same undecorated tropical type with a single wall in the usual manner by multiplying the attached wall functions. Finally, we consider the image of this wall structure in $\Sigma(X)$, which we denote by $\pi_*(\widetilde{\mathscr{S}}_{can})$.

We now proceed to the main result of this section:

\begin{corollary}
The wall structures $\pi_*(\widetilde{\mathscr{S}}_{can})$ and $\mathscr{S}_{can}$ on $\mathscr{P}$ are equivalent.
\end{corollary}

\begin{proof}
Since all contributing wall types of $\pi_*(\widetilde{\mathscr{S}}_{can})$ are lifts of wall types potentially contributing to $\mathscr{S}_{can}$ by Lemma \ref{awal}, we at least have equality of supports of the wall structures. To show the wall structure are equivalent, it suffice to show that for $\gamma$ a wall type lifting a wall type $\tau$, and $\pmb\tau$ a decoration of $\tau$, we have the equality $\rho_{\pmb\tau} = \prod_{\pmb \gamma} \rho_{\pmb\gamma}$ of wall functions, with the product running over all decorations of $\gamma$ which lift the decoration $\pmb \tau$ of $\tau$. Put more explicitly, we wish to show:

\[exp(k_\tau N_{\pmb\tau}z^{-u_\tau}q^A) =exp(k_{\gamma}z^{-u_\tau}q^A \sum_{\pmb\gamma} N_{\pmb\gamma}) \]

It suffices to show that $k_\tau N_{\pmb\tau} =  k_{\gamma}\sum_{\pmb\gamma} N_{\pmb\gamma}$, which we use Theorem \ref{mthm2} to prove. By Lemma \ref{decthm}, the image of the stabilization morphisms $\mathscr{M}(\tilde{X},\pmb{\gamma})\rightarrow \mathscr{M}(X,\pmb\tau)$ are unions of disjoint components and cover $\mathscr{M}(X,\pmb\tau)$. Combining this observation with the decorated version of Theorem \ref{mthm2}, we have $\sum_{\pmb \gamma} N_{\pmb\gamma} = \frac{1}{m} N_{\pmb\tau}$, with $m$ the lattice index of the inclusion $\gamma \rightarrow \tau$. As observed previously, the tropical modulus of a curve of types $\gamma$ or $\tau$ is determined by the image of the unique leg. The index may therefore be calculated to be $m = \frac{k_\gamma}{k_\tau}$. Hence $\sum_{\pmb\gamma} N_{\pmb\gamma} = \frac{k_{\tau}}{k_\gamma} N_{\pmb\tau}$, and multiplying both sides of the equality by $k_{\gamma}$ gives the desired equality. 
\end{proof}

\section{Birational invariance of intrinsic mirrors}

As a corollary of the birational invariance of the canonical wall structure derived in the previous section, we have also proven Corollary \ref{mcr1}  for a log \'etale modification $\tilde{X} \rightarrow X$, with $X$ satisfying either Assumptions $1.1$ or $1.2$ of \cite{scatt}. In \cite{int_mirror}, this mirror construction is generalized to arbitrary log Calabi-Yau varieties, without either Assumption $1.1$ or $1.2$ of \cite{scatt}. In the following, we show that this result can be generalized to hold for all examples of pairs of log smooth pairs $(X,D)$ considered in the less constrained setting of \cite{int_mirror}. Along the way, we prove certain pointed versions of Theorem \ref{mthm1}.

We first recall the relevant moduli spaces of punctured log maps used to construct the mirror families from \cite{int_mirror}. In \cite{int_mirror} Section $3.1$, Gross and Siebert define evaluation spaces $\mathscr{P}(X,r)$, for $r \in \Sigma(X)(\mathbb{Z})$. To describe these log stacks, first let $Z_r \subset X$ be the closure of the stratum $Z^\circ_r$ of $X$ associated to the smallest cone $\sigma_r \in \Sigma(X)$ containing $r$. The underlying stack of $\mathscr{P}(X,r)$ is defined to be $Z_r \times B\mathbb{G}_m$. This stack has a natural structure of a log stack $\widetilde{\mathscr{P}(X,r)}$, by giving $B\mathbb{G}_m \subset \mathcal{A}_{\NN}$ its pullback log structure, and viewing $\widetilde{\mathscr{P}(X,r)}$ as the product of log stacks. The log structure on $\mathscr{P}(X,r)$ is defined to be the sub log structure $\mathscr{M}_{\mathscr{P}(X,r)} \subset \mathscr{M}_{\widetilde{\mathscr{P}(X,r)}}$ with ghost sheaf given by:

\[\Gamma(U,\overline{\mathcal{M}}_{\mathscr{P}(X,r)}) = \{(m,r(m)) \text{ }|\text{ }m \in \Gamma(U,\overline{\mathcal{M}}_Z)\} \subset \Gamma(U,\overline{\mathcal{M}}_Z) \oplus \mathbb{N} = \Gamma(U,\overline{\mathcal{M}}_{\widetilde{\mathscr{P}(X,r)}})\]

For a contact order $r \in Sk(\tilde{X})(\mathbb{Z})$, we may similarly define $\widetilde{\mathscr{P}(\tilde{X},r)}$, $\mathscr{P}(\tilde{X},r)$ and $\tilde{Z}_r$. Using the fact the log \'etale modification restricts to $\pi: \tilde{Z}_r \rightarrow Z_r$, we have a natural morphism $\widetilde{\mathscr{P}(\tilde{X},r)} \rightarrow \widetilde{\mathscr{P}(X,r)}$. One can check that the appropriate sub log structures are respected by this morphism, hence we have a morphism $\mathscr{P}(\tilde{X},r) \rightarrow \mathscr{P}(X,r)$. Thus, given a morphism $B\mathbb{G}_m \rightarrow \mathscr{P}(\tilde{X},r)$, we also produce a morphism $B\mathbb{G}_m \rightarrow \mathscr{P}(X,r)$. By \cite{int_mirror} Proposition $3.8$, after a choice of $z \in \tilde{Z}_r^\circ$, we produce a morphism $B\mathbb{G}_m \rightarrow \mathscr{P}(\tilde{X},r)$. Since the modification restricts to a morphism $\tilde{Z}_r^\circ \rightarrow Z_r^\circ$, the resulting morphism $B\mathbb{G}_m \rightarrow Z_r$ has image contained in $Z_r^\circ$. From now on, we let $z' \in \tilde{Z}^\circ$ be a geometric point, and $z = \pi(z') \in Z^\circ$. 

Suppose now we have two contact order $p,q \in B(\tilde{X})(\mathbb{Z})$ which we call input contact orders, and another contact order $r \in B(\tilde{X})(\mathbb{Z})$ which we call an output contact order. After picking a choice of $A' \in NE(\tilde{X})$, we may define a decorated tropical type of punctured log curve $\beta'$ consisting of a single vertex $v$ with $\pmb\sigma(v) = 0$ decorated by $B$, and three legs $l_p,l_q$ and $l_r$ with associated contact orders $p,q$ and $-r$ respectively. By taking the image of the above data under $\pi_*$, we get a corresponding decorated tropical type $\beta$ on $X$, and a moduli space. Stabilization induces a morphism $\mathscr{M}(\tilde{X},\beta') \rightarrow \mathscr{M}(X,\beta)$. Each of these stacks are stratified by stacks associated to decorated and realizable tropical types. Fixing $\gamma$ a realizable tropical type marked by $\beta$, \cite{punc} Section $4$ describes obstruction theories related to the morphisms $\mathscr{M}(\tilde{X}, \gamma) \rightarrow \mathfrak{M}(\tilde{\mathcal{X}},\gamma)$ and $\mathscr{M}(\tilde{X},\gamma) \rightarrow \mathfrak{M}^{ev(x_{out})}(\tilde{\mathcal{X}},\gamma)= \mathfrak{M}^{ev}(\tilde{\mathcal{X}},\gamma) := \mathfrak{M}(\tilde{\mathcal{X}},\gamma) \times_{\tilde{\mathcal{X}}} \tilde{X}$. Since $\tilde{X} \rightarrow \tilde{\mathcal{X}}$ is smooth, the projection $\mathfrak{M}^{ev}(\tilde{\mathcal{X}},\gamma) \rightarrow \mathfrak{M}(\tilde{\mathcal{X}},\gamma)$ is smooth, hence equipped with a canonical perfect obstruction theory given by the relative cotangent complex, in which virtual pullback in the sense of \cite{vpull} is given by flat pullback. By construction of the obstruction theories in \cite{punc}, this triple yields a compatible triple. 

By the remarks following Proposition \ref{stab}, there is a unique decorated realizable tropical type $\pmb\tau$ marked by $\beta$ such that $\mathscr{M}(\tilde{X},\beta') \rightarrow \mathscr{M}(X,\beta)$ restricts to $\mathscr{M}(\tilde{X},\pmb\gamma) \rightarrow \mathscr{M}(X,\pmb\tau)$. The discussion above applies after replacing $\tilde{X}$ and $\pmb\gamma$ with $X$ and $\pmb\tau$. Using $z'\in \tilde{Z}^\circ$, we may form the point constrained moduli space:

\[
\begin{tikzcd}
\mathscr{M}(\tilde{X},\gamma,z') := \mathscr{M}(\tilde{X},\gamma) \times_{\mathscr{P}(\tilde{X},r)} B\mathbb{G}_m  &    \mathfrak{M}^{ev}(\tilde{\mathcal{X}},\gamma,z') = \mathfrak{M}^{ev}(\tilde{\mathcal{X}},\gamma) \times_{\mathscr{P}(\tilde{X},r)} B\mathbb{G}_m
\end{tikzcd}
\]

Using $z \in Z^\circ$, we may compose $B\mathbb{G}_m \rightarrow \mathscr{P}(\tilde{X},r) \rightarrow \mathscr{P}(X,r)$, and similarly define $\mathscr{M}(X,\tau,z)$ and $\mathfrak{M}^{ev}(\mathcal{X},\tau,z)$.

In order to compare the product of theta functions in $R_{(X,D)}$ and $R_{(\tilde{X},\tilde{D})}$, we first identify the contact orders in $\widetilde{\mathscr{P}} \subset Sk(\tilde{X},\mathbb{Z})$ with their image contact orders in $\mathscr{P} \subset Sk(X,\mathbb{Z})$. Then we may identify the basis of theta functions of $R_{(X,D)}$ with the basis of theta functions of $R_{(\tilde{X},\tilde{D)}}$, and write the product of theta functions in the two associated mirrors:

\begin{align}
\vartheta_p\vartheta_q  &= \sum_{A \in NE(X),r \in Sk(X,\mathbb{Z})} N_{p,q,r}^A \vartheta_rz^A \\
\vartheta_p\vartheta_q &= \sum_{A' \in NE(\tilde{X}), r\in Sk(\tilde{X},\mathbb{Z})} N_{p,q,r}^{A'} \vartheta_rz^{\pi_*(A')} 
\end{align}

In the sums on the righthand sides of the two equalities, the coefficients are defined as:

\begin{align}
N_{p,q,r}^A &= deg [\mathscr{M}(X,\beta,z)]^{vir}\\
N_{p,q,r}^{A'} &= deg [\mathscr{M}(\tilde{X},\beta',z')]^{vir}
\end{align}

To show the resulting rational combinations of theta functions are equal, it suffices to show for every term $N_{p,q,r}^{A}$ appearing as the coefficient in front of a $\vartheta_rz^A$ term in the first line, we have:

\[\sum_{A' \text{, }\pi_*(A') = A} N_{p,q,r}^{A'} = N_{p,q,r}^{A}.\] 

As a first step in deducing this equality using the results from this paper, we decompose the relevant moduli stacks. Although we focus on decomposing $\mathfrak{M}^{ev}(\mathcal{X},\beta,z)$ in what follows, an analogous description holds for $\mathfrak{M}^{ev}(\tilde{\mathcal{X}},\beta',z)$. By construction of the moduli stacks, there exists projection morphisms $ev: \mathfrak{M}^{ev}(\mathcal{X},\beta,z) \rightarrow B\mathbb{G}_m^\dagger$. Since the log structure on $B\mathbb{G}_m^\dagger$ has tropicalization $\mathbb{R}_{\ge 0}$, the projection is integral. Thus, the log and standard fiber dimensions are equal. By taking $\xi$ a generic point of an irreducible component of $\mathfrak{M}^{ev}(\mathcal{X},\beta,z)$, and $Q_\xi$ the stalk of the ghost sheaf at this point, we have $dim$ $Q_\xi^{gp} = $ $dim$ $\mathbb{N}^{gp} = 1$. Thus, we must have $Q_\xi = \mathbb{N}$.

To determine the multiplicity of the component of $\mathfrak{M}_\xi$, we follow the argument given in \cite{scatt} Section $6$ step $3$. First, the morphism $\mathfrak{M}^{ev}(\mathcal{X},\beta,z) \rightarrow B\mathbb{G}_m^\dagger$ is log smooth. Recalling that $B\mathbb{G}_m^\dagger$ has a natural idealized log structure given by $\mathbb{N}\setminus 0 \subset \mathbb{N}$, we may pullback this idealized log structure to $\mathfrak{M}^{ev}(\mathcal{X},\beta,z)$. After equipping these log stacks with this idealized log structure, by \cite{Og} IV.$3.1.22$, we have the resulting morphism of log stacks is idealized log smooth. Hence, by \cite{punc} Proposition B$4$, we have $\mathfrak{M}^{ev}(\mathcal{X},\beta,z)$ is, smooth locally around $\xi$, smooth over $B\mathbb{G}_m^\dagger \times_{B\mathbb{G}_m^{\dagger}} \mathcal{A}_{Q_\xi,K} = \mathcal{A}_{Q_\xi,K}$, where $K$ is the ideal generated by the image of $\mathbb{N}\setminus 0$ in $Q_\xi$ under the morphism of monoids $ev^*:\NN \rightarrow Q_\xi$. Thus, $\mu_\xi$ can be calculated as the multiplicity of $\mathcal{A}_{Q_{\xi},K}$, or equivalently $dim_\kk\text{ }\kk[\NN]/K$, which a straightforward calculation shows is $coker(ev^*:\mathbb{N}^{gp} \rightarrow Q_\xi^{gp} = \mathbb{Z})$. In particular, the multiplicity $m_\xi$ depends only on the tropical type of $\xi$ and the tropicalization of the evaluation map, so we may write $m_\xi = m_\tau$.

Now write $\mathfrak{M}^{ev}(\mathcal{X},\tau,z) := \mathfrak{M}^{ev}(\mathcal{X},\tau)\times_{\mathscr{P}(X,r)} B\mathbb{G}_m$. Since $\tau$ is marked by the tropical type $\beta$, there exists a finite morphism $\mathfrak{M}^{ev}(\mathcal{X},\tau,z) \rightarrow \mathfrak{M}^{ev}(\mathcal{X},\beta,z)$ of degree $|Aut\text{ }{\tau}|$. Using this moduli spaces of punctured log maps marked by tropical types $\tau$, we have the following commutative diagram:

\[\begin{tikzcd}
  \coprod_{\pmb\tau = (\tau,A)} \mathscr{M}(X,\pmb\tau,z)  \arrow{r} \arrow[d,"\epsilon_{\tau,z}"] & \mathscr{M}(X,\beta,z) \arrow[d,"\epsilon_z"] \\
 \coprod_\tau \mathfrak{M}^{ev}(\mathcal{X},\tau,z) \arrow[r,"i"] & \mathfrak{M}^{ev}(\mathcal{X},\beta,z)
\end{tikzcd}\]

Since the multiplicity of a component of $\mathfrak{M}^{ev}(\mathcal{X},\beta,z)$ depended only on the tropical type of the generic point, we have the decomposition:

\begin{equation}\label{eq1}
[\mathfrak{M}^{ev}(\mathcal{X},\beta,z)]= \sum_\tau \frac{1}{|Aut(\tau)|}i_*[\mathfrak{M}^{ev}(\mathcal{X},\tau,z)]
\end{equation}

After identifying the Chow theories of $\mathfrak{M}^{ev}(\mathcal{X},\tau,z)$ and $\mathfrak{M}^{ev}(\mathcal{X},\tau,z)_{red}$ and denoting by $\epsilon_{\pmb\tau,z}$ for the forgetful morphism $\epsilon_{\pmb\tau,z}: \mathscr{M}(X,\pmb\tau,z) \rightarrow \mathfrak{M}^{ev}(\mathcal{X},\tau,z)$, we let 
\[N_{\pmb\tau} := \frac{1}{|Aut(\tau)|}deg[\mathscr{M}(X,\pmb\tau,z)]^{vir} = \frac{m_\tau}{|Aut(\tau)|}deg(\epsilon^!_{\pmb\tau,z}([\mathfrak{M}^{ev}(\mathcal{X},\tau,z)_{red}])).\]
 
Letting $\pmb\tau \vdash (\tau,\textbf{A})$ mean the decorated tropical type $\pmb\tau$ has total curve class $\textbf{A}$, we have the following decomposition of $N_{p,q,r}^A$:

\[N_{p,q,r}^A = \sum_{\pmb\tau \vdash (\tau,A)} N_{\pmb \tau}.\]

By analogous arguments for $\tilde{X}$ and the decorated tropical type $\beta'$, after setting
\[N_{\pmb\gamma} := \frac{1}{|Aut(\gamma)|}deg[\mathscr{M}(X,\pmb\gamma,z')]^{vir} = \frac{m_\gamma}{|Aut(\gamma)|}deg(\epsilon_{\pmb\gamma,z'}^!([\mathfrak{M}^{ev}(\tilde{\mathcal{X}},\gamma,z')_{red}])),\] we have:

\[N_{p,q,r}^{A'} = \sum_{\pmb \gamma \vdash (\gamma,A')} N_{\pmb \gamma}.\]

To compare these structure constants, note that given a choice of decorated tropical type $\pmb\gamma$ contributing to $N_{p,q,r}^{A'}$, there is a unique decorated tropical type $\pmb\tau$, which is easily seen to correspond to a contributing tropical type to $N_{p,q,r}^A$. On the other hand, given a contributing tropical type to $N_{p,q,r}^A$, since each generic point of $\mathfrak{M}^{ev}(\mathcal{X},\beta,z)$ has monoid $\NN$, the tropicalization of the universal tropical map associated to this point is associated with a morphism of cones $\mathbb{R}_{\ge 0} \rightarrow \tau$. In particular, this morphism factors uniquely through a cone $\gamma \subset \tau$ associated with a subdivision of $\tau$. Hence, there is a unique maximally extended tropical lift $\gamma$ of $\tau$ which satisfies the tropical point constraint. 

With this relationship between contributing tropical types to $N_{p,q,r}^{A'}$ and $N_{p,q,r}^{\pi_*(A')}$ in mind, the key result of this section is the following:

\begin{proposition}\label{strcomp}
For each decorated type $\pmb \tau$ contributing to $N_{p,q,r}^A$, we have
\[ N_{\pmb\tau} = \sum_{\pmb\gamma \rightarrow \pmb\tau}N_{\pmb\gamma}\]
where the sum ranges over all decorated lifts of $\pmb\tau$ with underlying tropical type $\gamma$.

\end{proposition}

\begin{proof}
Observe that stabilization induces a morphism $\mathscr{M}(\tilde{X},\pmb\gamma,z') \rightarrow \mathscr{M}(X,\beta,z)$. Since $\pmb\gamma$ is decorated, there exists a unique decorated type $\pmb\tau$ such that the stabilization morphism factors through the proper map $\mathscr{M}(X,\pmb\tau,z) \rightarrow \mathscr{M}(X,\beta,z)$. In order to use our main theorem, we define the stacks $\mathfrak{M}^{ev}_{\gamma\rightarrow \tau} = \mathfrak{M}_{\gamma \rightarrow \tau}\times_{\mathscr{P}(\tilde{\mathcal{X}},r)} \mathscr{P}(\tilde{X},r)$ and $\mathfrak{M}^{ev}_{\gamma\rightarrow \tau,z'} = \mathfrak{M}^{ev}_{\gamma \rightarrow \tau} \times_{\mathscr{P}(\tilde{X},r)}^{fs} B\mathbb{G}_m$, where $\mathfrak{M}_{\gamma \rightarrow \tau} \rightarrow \mathscr{P}(\tilde{\mathcal{X}},r)$ is defined by precomposing the \'etale map $\mathfrak{M}_{\gamma \rightarrow \tau} \rightarrow \mathfrak{M}_\gamma$ with the evaluation morphism $\mathfrak{M}_\gamma \rightarrow \mathscr{P}(\tilde{\mathcal{X}},r)$. We may pullback the obstruction theory for the morphism $\mathscr{M}(\tilde{X},\pmb\gamma) \rightarrow \mathfrak{M}^{ev}_{\gamma\rightarrow \tau}$ to give an obstruction theory to $\mathscr{M}(\tilde{X},\pmb\gamma,z') \rightarrow \mathfrak{M}^{ev}_{\gamma\rightarrow\tau,z'}$. Chasing cartesian diagrams given by definition shows that $\mathfrak{M}^{ev}_{\gamma \rightarrow \tau,z'}$ is \'etale over $\mathfrak{M}^{ev}_{\gamma,z'}$ and the virtual fundamental class associated with the above defined obstruction theory coincides with the one defined in \cite{int_mirror} Proposition $3.12$. After writing $\mathscr{M}_{\pmb\tau,z}$ for the closed substack given by the image of $\mathscr{M}(\tilde{X},\pmb\gamma,z') \rightarrow \mathscr{M}(X,\pmb\tau,z)$, we show the following lemma:

\begin{lemma}
The following commutative diagram is cartesian in all categories, and the perfect obstruction theories associated with the horizontal maps are compatible

\begin{equation}\label{di1}
\begin{tikzcd}
\mathscr{M}(\tilde{X},\pmb\gamma,z') \arrow{r} \arrow[d,"\pi"] & \mathfrak{M}^{ev}_{\gamma\rightarrow \tau,z'} \arrow[d,"p"] \\
\mathscr{M}_{\pmb\tau,z} \arrow{r} & \mathfrak{M}^{ev}(\mathcal{X},\tau,z)
\end{tikzcd}
\end{equation}

\end{lemma}

\begin{proof}
 Note that since the bottom arrow is strict, it suffices to show the diagram is cartesian in the category of fs log stacks. With this end in mind, suppose $S$ is a fine and saturated scheme, and we have morphisms $S \rightarrow \mathscr{M}_{\pmb\tau,z}$ and $S \rightarrow \mathfrak{M}^{ev}_{\gamma \rightarrow \tau,z'}$ such that the two composites $S \rightarrow \mathfrak{M}(\mathcal{X},\tau,z)$ are isomorphic. Note that by definition, $\mathscr{M}(\tilde{X},\pmb\gamma,z')$ is constructed as the following fs pullback:

\[\begin{tikzcd}
\mathscr{M}(\tilde{X},\pmb\gamma,z') \arrow{r} \arrow{d} & B\mathbb{G}_m \arrow{d} \\
\mathscr{M}(\tilde{X},\pmb\gamma) \arrow{r} & \mathscr{P}(\tilde{X},r)
\end{tikzcd}\]

The morphism $S \rightarrow \mathfrak{M}^{ev}_{\gamma \rightarrow \tau,z'}$ provides a morphism $S \rightarrow B\mathbb{G}_m$ after composing with the projection $\mathfrak{M}^{ev}_{\gamma \rightarrow \tau,z'} \rightarrow B\mathbb{G}_m$. To construct a morphism $S \rightarrow \mathscr{M}(\tilde{X},\pmb\gamma)$, note that by composing $S \rightarrow \mathfrak{M}^{ev}_{\gamma \rightarrow \tau,z'}$ with the projection $\mathfrak{M}^{ev}_{\gamma \rightarrow \tau,z'} \rightarrow \mathfrak{M}_{\gamma \rightarrow \tau}$, we have a morphism $S \rightarrow \mathfrak{M}_{\gamma \rightarrow \tau}$. Since we have a morphism $S \rightarrow \mathscr{M}(X,\pmb\tau)$ as well, by the commutativity of diagram \ref{di1} and Theorem \ref{mthm1}, we have a uniquely defined morphism $S \rightarrow \mathscr{M}(\tilde{X},\gamma)$. By Theorem \ref{decthm}, we in fact have that $S$ factors through $\mathscr{M}(\tilde{X},\pmb\gamma)$. 

To see that the two induced morphisms $S \rightarrow \mathscr{P}(\tilde{X},r)$ are isomorphic, we prove the following diagram is fs cartesian:

\begin{lemma}
The following diagram is cartesian in the category of fine and saturated log stacks:
\begin{equation}\label{di2}
\begin{tikzcd}
B\mathbb{G}_m^\dagger \arrow[r,"i"] \arrow[d,"id"] &  \mathscr{P}(\tilde{X},r) \arrow{d}\\
B\mathbb{G}_m^\dagger \arrow{r} & \mathscr{P}(X,r) 
\end{tikzcd}
\end{equation}

\end{lemma}

\begin{proof}

First, observe the following diagram of log schemes is fs cartesian:

\begin{equation}\label{di2.1}
\begin{tikzcd}
(z',Spec\text{ }\NN) \arrow{r} \arrow{d} &  \tilde{X} \arrow{d} \\
(z,Spec\text{ }\NN) \arrow{r} & X
\end{tikzcd}
\end{equation}
Indeed, as log maps, the morphism $(z,Spec\text{ }\NN) \rightarrow \mathcal{A}_X$ factors uniquely through the subdivision $\widetilde{\mathcal{A}}_X \rightarrow \mathcal{A}_X$, hence the diagram is fs cartesian after replacing $X$ and $\tilde{X}$ with $\mathcal{A}_X$ and $\widetilde{\mathcal{A}}_X$, and since $\tilde{X} \rightarrow X$ is pulled back from the subdivision $\widetilde{\mathcal{A}}_X \rightarrow \mathcal{A}_X$, standard arguments with cartesian squares show the commutative square above is fs cartesian. 

Using the fs cartesian diagram of (\ref{di2.1}), we have the following commuting diagram of fs log stacks, with the left most square fs cartesian:

\begin{equation}\label{di2.2}
\begin{tikzcd}
(z,\NN)\times B\mathbb{G}_m^\dagger \arrow{r} \arrow{d} & \widetilde{\mathscr{P}(\tilde{X},r)} \arrow{r} \arrow{d} & \mathscr{P}(\tilde{X},r)\arrow{d} \\
(z,\NN)\times B\mathbb{G}_m^\dagger \arrow{r} & \widetilde{\mathscr{P}(X,r)} \arrow{r} & \mathscr{P}(X,r)
\end{tikzcd}
\end{equation}
Note that both of the morphisms $(z,\NN) \times B\mathbb{G}_m^\dagger \rightarrow \mathscr{P}(\tilde{X},r),\mathscr{P}(X,r)$ factor through fs log stacks with underlying stacks $B\mathbb{G}_m$, but with log structures given by the sub log structure of $(z,\NN) \times B\mathbb{G}_m^\dagger$ whose ghost sheaf is globally generated by the diagonal in $\NN^2 = \Gamma((z,\NN) \times B\mathbb{G}_m^\dagger,\overline{\mathcal{M}})$. Additionally, we note that this log stack may be identified with $B\mathbb{G}_m^\dagger$, although the morphism $(z,\NN) \times B\mathbb{G}_m^\dagger \rightarrow B\mathbb{G}_m^\dagger$ induced by inclusion of log structures is not the projection morphism after this identification. 

Now suppose $S$ is a fine and saturated log scheme, and $f:S \rightarrow \mathscr{P}(\tilde{X},r)$, $g: S \rightarrow B\mathbb{G}_m^\dagger$ are morphisms such that the two induced morphisms $S \rightarrow \mathscr{P}(X,r)$ are isomorphic. To check the diagram of interest is cartesian, since the morphism $B\mathbb{G}_m^\dagger \rightarrow B\mathbb{G}_m^\dagger$ in (\ref{di2}) is the identity, it suffices to verify that the morphism $f$ factors as $ig$. To do so, consider the fiber product $S' = S \times_{B\mathbb{G}_m^\dagger} ((z,\NN)\times B\mathbb{G}_m^\dagger)$. Since the diagonal morphism $\NN \rightarrow \NN^2$ is fine and saturated, we have $S' = S^{'sat}$. Moreover, the underlying stack of the fiber product is $S$.

By the universal property of the left square of (\ref{di2.2}), we have the morphism $S' = S^{'sat} \rightarrow \mathscr{P}(\tilde{X},r)$ factors through $(z,\NN) \times B\mathbb{G}_m^\dagger$, and thus factors through $B\mathbb{G}_m^\dagger$. Since $S' \rightarrow S$ is an epimorphism, we must have $f$ factors as $ig$, as desired. 
\end{proof}

Note that since the two morphisms $S \rightarrow \mathfrak{M}^{ev}(\mathcal{X},\tau,z)$ are isomorphic by assumption, the two induced morphisms $S \rightarrow \mathscr{P}(X,r)$ are isomorphic and factor through $B\mathbb{G}_m^\dagger \rightarrow \mathscr{P}(X,r)$. As shown above, we have any two lifts of $B\mathbb{G}_m^\dagger \rightarrow \mathscr{P}(X,r)$ to $B\mathbb{G}_m^\dagger \rightarrow \mathscr{P}(\tilde{X},r)$ are isomorphic, showing diagram (\ref{di1}) is cartesian.

To see that the obstruction theories are compatible, we note that these have explicit descriptions given in \cite{punc} Proposition $4.4$ as $R\pi_*(f^*T^{log}_{X}(-x_{out}))$ and $R\pi_*(\tilde{f}^*T^{log}_{\tilde{X}}(-x_{out}))$, for $f: \mathfrak{C} \rightarrow X$ and $\tilde{f}: \widetilde{\mathfrak{C}} \rightarrow \tilde{X}$ the universal curves over $\mathscr{M}(\tilde{X},\gamma,z')$ and $\mathscr{M}_{\pmb\tau}$ respectively. Observe that the partial stabilization map on universal curves $\tilde{\mathfrak{C}} \rightarrow \mathfrak{C}$ is an isomorphism around the section $x_{out}$. Indeed, we chose the tropical lift $\gamma$ so that $r \in \pmb\sigma(L_{out})$. In particular, the associated leg of $\tau$ is not subdivided. Thus, over a geometric point of $\mathscr{M}(X,\gamma,z')$, the stabilization map does not contract a component containing the punctured point $x_{out}$. Thus, $\mathcal{O}_\mathfrak{C}(-x_{out})$ pulls back to $\mathcal{O}_{\widetilde{\mathfrak{C}}}(-x_{out})$. Using this fact, together with the argument used to compare the obstruction theories in the proof of Theorem \ref{mthm1}, we see that the pullback of the obstruction theory on $\mathscr{M}(X,\tau,z)$ gives the obstruction theory on $\mathscr{M}(\tilde{X},\gamma,z')$. 

\end{proof}

Having shown diagram (\ref{di1}) is cartesian with compatible perfect obstruction theories, note that by \cite{vpull} Proposition $4.1$ and since the image of $\mathfrak{M}_{\gamma \rightarrow \tau,z',red}^{ev}$ in $\mathfrak{M}^{ev}_{\tau,z}$ is $\mathfrak{M}^{ev}_{\tau,z,red}$, we have: 

\begin{equation}\label{pushpoint}
\pi_*[\mathscr{M}(\tilde{X},\pmb\gamma,z')]^{vir} = m_{\gamma}deg(\mathfrak{M}_{\gamma\rightarrow \tau,z'}^{ev} \rightarrow \mathfrak{M}_{\tau,z}^{ev})\epsilon_{\pmb\tau,z}^!(\mathfrak{M}^{ev}_{\tau,z,red}).
\end{equation}

Note by construction, the virtual dimensions of both $\mathscr{M}(\tilde{X},\gamma,z')$ and $\mathscr{M}(X,\tau,z)$ are zero.

To compute this degree, we will first show the morphism $\mathfrak{M}^{ev}_{\gamma \rightarrow \tau,z'} \rightarrow \mathfrak{M}^{ev}_{\tau,z}$ is the following fine and saturated pullback:

\begin{lemma}
The following diagram is cartesian in the category of fine and saturated log stacks:

\[\begin{tikzcd}
\mathfrak{M}^{ev}_{\gamma \rightarrow \tau,z'} \arrow{r} \arrow{d} & \mathfrak{M}_{\gamma \rightarrow\tau} \arrow{d} \\
\mathfrak{M}^{ev}_{\tau,z} \arrow{r} & \mathfrak{M}_\tau
\end{tikzcd}\]

\end{lemma}

\begin{proof}
We prove this by arguing for the existence of a cartesian square beyond what we have seen, and combining this extra input with the cartesian squares we know. Note first that the following commutative diagram of algebraic log stacks is cartesian in all categories:

\[\begin{tikzcd}\label{di5}
\mathscr{P}(\tilde{X},r) \arrow{r}\arrow{d}& \mathscr{P}(\tilde{\mathcal{X}},r) \arrow{d}\\
\mathscr{P}(X,r) \arrow{r} & \mathscr{P}(\mathcal{X},r)
\end{tikzcd}.\]
Indeed, the bottom horizontal arrow is strict, so the underlying stack of the fiber product is $\tilde{X}\times B\mathbb{G}_m^\dagger$, and the log structure defining $\mathscr{P}(\tilde{X},r)$ is pulled back from the log structure defining $\mathscr{P}(\tilde{\mathcal{X}},r)$. Now consider the following commutative diagram:

\[\begin{tikzcd}
\mathfrak{M}^{ev}_{\gamma \rightarrow \tau,z'} \arrow{r}\arrow{d}& B\mathbb{G}_m^\dagger \arrow{r}\arrow{d} & B\mathbb{G}_m^\dagger\arrow{d}\\
\mathfrak{M}^{ev}_{\gamma \rightarrow \tau} \arrow{r}\arrow{d}& \mathscr{P}(\tilde{X},r) \arrow{r}\arrow{d}& \mathscr{P}(X,r)\arrow{d}\\
\mathfrak{M}_{\gamma \rightarrow \tau} \arrow{r}& \mathscr{P}(\tilde{\mathcal{X}},r) \arrow{r}& \mathscr{P}(\mathcal{X},r)
\end{tikzcd}\]

The top right square is the fine and saturated cartesian square of Diagram \ref{di2}, and the bottom right square is the cartesian Diagram \ref{di5}. The rest of the squares are cartesian either by construction or by the two out of three property for cartesian squares. Finally, we show the commuting square of interest is cartesian by considering the following commuting diagram.

\[\begin{tikzcd}
\mathfrak{M}^{ev}_{\gamma \rightarrow \tau,z'} \arrow{r}\arrow{d} & \mathfrak{M}^{ev}_{\tau,z} \arrow{r}\arrow{d} & B\mathbb{G}_m^\dagger\arrow{d}\\
\mathfrak{M}^{ev}_{\gamma \rightarrow \tau} \arrow{r} \arrow{d} & \mathfrak{M}^{ev}_{\tau} \arrow{r} \arrow{d} & \mathscr{P}(X,r)\arrow{d}\\
\mathfrak{M}_{\gamma \rightarrow \tau} \arrow{r} & \mathfrak{M}_{\tau} \arrow{r} & \mathscr{P}(\mathcal{X},r)
\end{tikzcd}\]

The bottom right square is cartesian by Lemma $3.7$ of \cite{int_mirror}. All other squares are cartesian as previously shown, or by the two out of three property for cartesian squares. 
\end{proof}

With this lemma in mind, consider a geometric generic point $Spec\text{ }\kappa \rightarrow \mathfrak{M}^{ev}_{\tau,z'}$. By the above lemma, in order to compute the fiber product $Spec\text{ }\kappa \times_{\mathfrak{M}^{ev}_{\tau,z'}} \mathfrak{M}^{ev}_{\gamma\rightarrow \tau,z}$, it suffices to compute the fine and saturated fiber product $Spec\text{ }\kappa \times^{fs}_{\mathfrak{M}_\tau} \mathfrak{M}_{\gamma \rightarrow \tau}$, where $Spec\text{ }\kappa$ is equipped with the pullback log structure from $\mathfrak{M}^{ev}_{\tau,z'}$. By the fact we picked a generic point and by Theorem \ref{mthm1}, it suffices to compute the fine and saturated pullback $Spec\text{ }\kappa \times^{fs}_{\mathcal{A}_\tau} \mathcal{A}_\gamma$. To compute this, note that $Spec\text{ }\kappa \rightarrow \mathcal{A}_\tau$ factors through the strict morphism $Spec\text{ }\kappa \rightarrow \mathcal{A}_{\NN}$. Recalling $\gamma \rightarrow \tau$ is a cone inclusion which includes the image of $\NN^\vee \rightarrow \tau$, we may use the characterization of morphisms to Artin cones given in Lemma \ref{mtoZA} to show the following diagram is cartesian in fine and saturated log stacks:

\[\begin{tikzcd}
\mathcal{A}_{\NN} \arrow{r} \arrow{d} & \mathcal{A}_{\gamma} \arrow{d}\\
\mathcal{A}_{\NN} \arrow{r} & \mathcal{A}_{\tau}
\end{tikzcd}\]

Thus, the fiber product we are interested in is $Spec \text{ }\kappa \times_{\mathcal{A}_\NN} \mathcal{A}_{\NN}$, with $Spec\text{ }\kappa \rightarrow \mathcal{A}_{\NN}$ having image in the deepest stratum of the Artin cone. By Lemma \ref{DM}, this stack is \'etale over $Spec\text{ }\kappa$ of degree $\frac{1}{coker(\NN \rightarrow \NN)} =  \frac{m_{\tau}}{m_{\gamma}}$. Hence:
\[deg(\mathfrak{M}^{ev}_{\gamma\rightarrow \tau,z'} \rightarrow \mathfrak{M}^{ev}_{\tau,z}) = \frac{m_{\tau}}{m_{\gamma}}.\]
By similar arguments to those given in Theorem \ref{decthm}, this time using the fact that proper morphisms are preserved under fine and saturated fiber products, the image of the morphism $\mathscr{M}(\tilde{X},\pmb\gamma,z') \rightarrow \mathscr{M}(X,\pmb\tau,z)$ is a union of connected components. Moreover, using equation \ref{pushpoint}, we find $\pi_*([\mathscr{M}(\tilde{X},\pmb\gamma,z')]^{vir}) = [\mathscr{M}_{\pmb\tau,z}]^{vir}$. Putting everything together, we have:

\[N_{\pmb\tau} = \frac{1}{|Aut(\tau)|} deg\text{ } [\mathscr{M}(X,\pmb\tau,z)]^{vir} =  deg\text{ } \sum_{\pmb\gamma \rightarrow \pmb\tau} \pi_*\frac{1}{|Aut(\gamma)|}[\mathscr{M}(\tilde{X},\pmb\gamma,z')]^{vir} =\sum_{\pmb\gamma \rightarrow \pmb\tau}N_{\pmb\gamma}.\]

\end{proof}

To conclude the proof of Corollary $\ref{mcr1}$, we observe the following chain of equalities:

\[
N_{p,q,r}^A = \sum_{\pmb\tau \vdash (\tau,A)} N_{\pmb\tau} = \sum_{\pmb\tau \vdash (\tau,A)} \sum_{\pmb\gamma \rightarrow \pmb\tau} N_{\pmb\gamma} = \sum_{\pi_*(A') =A} \sum_{\pmb\gamma \vdash (\gamma,A')} N_{\pmb\gamma} = \sum_{\pi_*(A') = A}  N_{p,q,r}^{A'}
\]

In the above display, the second equality follows from Proposition \ref{strcomp}. This gives our desired relation of structure constants of the mirror algebra, and hence completes the proof of Corollary \ref{mcr1}.

\nocite{*}
\bibliographystyle{amsalpha}
\bibliography{birinvar}

\end{document}